\documentclass[reqno]{amsart}

\usepackage{amsmath,amsfonts,amsthm,amssymb,graphics, paralist, mathabx}
\usepackage{hyperref}

\newcommand{\g}{\geqslant}

\newcommand{\RR}{\mathbb{R}}
\newcommand{\ZZ}{\mathbb{Z}}
\newcommand{\CC}{\mathbb{C}}
\newcommand{\s}{\mathcal{S}}

\newcommand{\sph}{\mathbb{S}}
\newcommand{\NN}{\mathbb{N}}
\newcommand{\p}{\partial}

\newcommand{\les}{\leqslant}
\newcommand{\lesa}{\lesssim}

\newcommand{\supp}{\text{supp }\,}

\newcommand{\mc}[1]{\mathcal{#1}}
\newcommand{\mf}[1]{\mathfrak{#1}}

\DeclareSymbolFont{bbold}{U}{bbold}{m}{n}
\DeclareSymbolFontAlphabet{\mathbbold}{bbold}

\newcommand{\ind}{\mathbbold{1}}

\DeclareMathOperator*{\sgn}{\text{sgn}}

\theoremstyle{plain}
\newtheorem{theorem}{Theorem}[section]
\newtheorem{proposition}[theorem]{Proposition}
\newtheorem{lemma}[theorem]{Lemma}
\newtheorem{corollary}[theorem]{Corollary}

\theoremstyle{remark}
\newtheorem{remark}[theorem]{Remark}

\setdefaultenum{(i)}{(a)}{1}{A}
\begin{document}
\title{Global well-posedness for the massless cubic Dirac equation}

\author{Nikolaos Bournaveas}
\address{Department of Mathematics, University of Edinburgh,
Edinburgh EH9 3JE, United Kingdom}
\email{n.bournaveas@ed.ac.uk}

\author{Timothy Candy}
\address{Department of Mathematics,
Imperial College London, London SW7 2AZ, United Kingdom}
\email{t.candy@imperial.ac.uk}

\date{\today}

\begin{abstract}
We show that the cubic Dirac equation with zero mass is globally well-posed for small data in the scale invariant space $\dot{H}^{\frac{n-1}{2}}(\RR^n)$ for $n=2, 3$. The proof proceeds by using the Fierz identities to rewrite the equation in a form where the null structure of the system is readily apparent. This null structure is then exploited via bilinear estimates in spaces based on the  null frame spaces of Tataru. We hope that the spaces and estimates used here can be applied to other nonlinear Dirac equations in the scale invariant setting. Our work complements recent results of Bejenaru-Herr who proved a similar result for $n=3$ in the  massive case.
\end{abstract}

\maketitle

\tableofcontents

%------------------------------------------------------------------------------------------------------------------------------%
%------------------------------------------------------------------------------------------------------------------------------%
\section{Introduction}
%------------------------------------------------------------------------------------------------------------------------------%
%------------------------------------------------------------------------------------------------------------------------------%

Given a mass $m \g 0$ we consider the nonlinear Dirac equation
        \begin{equation} \label{eqn - NLD eqn} \begin{split}- i \gamma^\mu \p_\mu \psi + m \psi &= F(\psi)\\
                                    \psi(0) &= \psi_0 \end{split}
        \end{equation}
for a spinor $\psi(t, x): \RR^{1+n} \rightarrow \CC^N$ where $N = 2^{ [ \frac{n+1}{2}]}$ and $[x]$ denotes the integer part of $x \in \RR$. The Gamma matrices $\gamma^\mu$ are constant $N \times N$ matrices satisfying the anti-commutativity properties
        $$ \gamma^\mu \gamma^\nu + \gamma^\nu \gamma^\mu =  2 I g^{\mu \nu}$$
where $g^{\mu \nu}$ is the Minkowski metric $g = \text{ diag }(1, -1, ... , -1)$, and repeated upper and lower indices are summed over $\mu = 0, ..., n$. We are interested in the special case of (\ref{eqn - NLD eqn}) where the nonlinearity $F$ is cubic and has some additional structure. More precisely, we consider the Lorentz invariant cubic nonlinearities
        \begin{equation}\label{eqn - cond on F} F(\psi) = \begin{cases}
      (\overline{\psi} \psi) \psi \\
      (\overline{\psi} \gamma^\mu \psi) \gamma_\mu \psi
    \end{cases}\end{equation}
which are known as the Soler model \cite{Soler1970}, and the Thirring model \cite{Thirring1958} respectively. Here $\overline{\psi} = \psi^\dagger \gamma^0$ is the Dirac adjoint, and $\psi^\dagger$ is the complex conjugate transpose of the vector $\psi$.  The nonlinear Dirac equation is an important equation in relativistic quantum mechanics, and models the self-interaction of Dirac fermions, we refer the reader to \cite{Finkelstein1956, Thaller1992} for more on the physical background of the Dirac equation.\\

The nonlinear Dirac equation (\ref{eqn - NLD eqn}) with cubic nonlinearity (\ref{eqn - cond on F}) and mass $m=0$ is invariant under the scaling $\psi(t, x) \mapsto \lambda^\frac{1}{2} \psi( \lambda t, \lambda x)$. Thus the scale invariant regularity is $s_c = \frac{n-1}{2}$ and it is expected that we have some form of ill-posedness for data $\psi_0 \in H^s(\RR^n)$ with $s< \frac{n-1}{2}$. In terms of the well-posedness of the Cauchy problem, in the $n=3$ case, work of Tzvetkov \cite{Tzvetkov1998} via the method of commuting vector fields, shows that we have global existence in time for small smooth data in the case $|F(\psi)| \lesa |\psi|^p$, $p>2$. This extends earlier results of Reed \cite{Reed1976}, Dias-Figueira \cite{Dias1987}, and Escobedo-Vega \cite{Escobedo1997}. In the low regularity setting, Machihara-Nakanishi-Ozawa \cite{Machihara2003} obtained global existence for small data in $H^s(\RR^3)$  in the almost critical case $s>1$ for positive mass $m>0$, and cubic nonlinearities (\ref{eqn - cond on F}). This was improved to radial data (or data with some additional angular regularity) in $H^1(\RR^3)$ by Machihara-Nakamura-Nakanishi-Ozawa \cite{Machihara2005a}. Very recently, Bejenaru-Herr
\cite{Bejenaru2013} proved that provided $m>0$ and $F(\psi) = (\overline{\psi} \psi)\psi$, we have global well-posedness and scattering for small data in the critical space $H^1(\RR^3)$.

On the other hand, in the $n=2$ case, it was shown by Pecher \cite{Pecher2013a, Pecher2015} that we have local well-posedness from data in $H^s(\RR^2)$ in the almost critical case $s>\frac{1}{2}$. In the $n=1$ case, global well-posedness for the Thirring model with large data in $H^s(\RR)$ with $s\g 1$ is due to Delgado \cite{Delgado1978}, this was improved to $s>\frac{1}{2}$ by Selberg-Tesfahun \cite{Selberg2010b}. The critical case $s=0$ was considered by the second author in \cite{Candy2010} where it was shown that the Thirring model is globally well-posed for large data in $L^2(\RR)$. The question of scattering for the massive case $m>0$ is still open. As well as the above mentioned results, if $n=1$ it is known that the Thirring model $F(\psi) = (\overline{\psi}\gamma^\mu \psi) \gamma_\mu \psi$ is completely integrable \cite{Kaup1977, Zakharov1980}, and the stability of stationary solutions has been studied \cite{Contreras2013, Pelinovsky2014}. The existence of stationary solutions in $n=3$ is also known \cite{Cazenave1986, Merle1988, Strauss1986}. \\

In the current article we are interested in the global well-posedness of small data in the critical space $\dot{H}^\frac{n-1}{2}(\RR^n)$ for $n=2, 3$. Our main result is the following.

\begin{theorem}\label{thm - main thm intro}
Let $n=2, 3$, $m=0$,  and $s \g \frac{n-1}{2}$. Assume $F$ is as in (\ref{eqn - cond on F}). There exists $\epsilon>0$ such that if $\psi_0 \in \dot{H}^\frac{n-1}{2} \cap \dot{H}^s(\RR^n)$ with
    $$\| \psi_0 \|_{\dot{H}^\frac{n-1}{2}(\RR^n)} < \epsilon$$
then we have a global  solution $\psi \in C\big( \RR, \dot{H}^\frac{n-1}{2} \cap \dot{H}^s(\RR^n)\big)$ to (\ref{eqn - NLD eqn}) with
        $$\sup_{t\in \RR} \,\, \| \psi(t) \|_{\dot{H}^\frac{n-1}{2} \cap \dot{H}^s(\RR^n)} \lesa \| \psi_0 \|_{\dot{H}^\frac{n-1}{2} \cap \dot{H}^s(\RR^n)}.$$
Moreover, the solution $\psi$ depends continuously on the initial data  and is the unique limit of smooth solutions. Finally, there exists  $\psi_{\pm \infty} \in C\big( \RR, \dot{H}^\frac{n-1}{2} \cap \dot{H}^s(\RR^n)\big)$ with $\gamma^\mu \p_\mu \psi_{\pm \infty} = 0$ such that
        $$ \lim_{t \rightarrow \pm \infty} \big\| \psi(t) - \psi_{\pm \infty}(t) \big\|_{\dot{H}^\frac{n-1}{2} \cap \dot{H}^s(\RR^n)} = 0 .$$
\end{theorem}

\begin{remark}
  If $m>0$ then for small data in $\psi_0 \in \dot{H}^\frac{n-1}{2}(\RR^n)$ we have existence of a solution up to time $T \ll m^{-1}$, see Remark \ref{rem on positive mass}. This is essentially due to the fact that for times $T \ll m^{-1}$ the solution to the wave equation and Klein-Gordon equation is more or less the same. Of course if $m>0$ and $n=3$, then we already have global existence due to the work of Bejenaru-Herr \cite{Bejenaru2013}. On the other hand, local existence for $m>0$ with data in $\dot{H}^\frac{1}{2}(\RR^2)$ is new in the case $n=2$. Similarly, it is possible to use finite speed of propagation to deduce local in time existence for large data in $\dot{H}^{\frac{n-1}{2}}(\RR^n)$ (this is true for any $m \g 0$).
\end{remark}

\begin{remark}
  If $n=3$ and we let $\gamma^5 = i \gamma^0 \gamma^1 \gamma^2 \gamma^3$, then Theorem \ref{thm - main thm intro} also holds in the cases where $F(\psi)$ is given by
    $$ (\bar{\psi} \gamma^5 \psi) \gamma^5 \psi, \qquad (\bar{\psi} \psi) \gamma^5 \psi, \qquad (\bar{\psi} \gamma^5 \psi) \psi.$$
   In other words, we can more or less handle any nonlinearity built up using the bilinear Dirac null forms $\overline{\psi} \psi$ and $\overline{\psi} \gamma^5 \psi$.
\end{remark}

\begin{remark}
  The nonlinear Dirac equation (\ref{eqn - NLD eqn}) together with the nonlinearity (\ref{eqn - cond on F}), satisfies \emph{conservation of charge} $\| \psi(t) \|_{L^2_x} = \| \psi(0) \|_{L^2_x}$. Thus in Theorem \ref{thm - main thm intro} we may replace the homogeneous Sobolev spaces $\dot{H}^s$ with the inhomogeneous spaces $H^s$. We should point out that the Dirac equation has other conserved quantities. However, they are not strictly positive, and thus do not appear to be immediately useful in the large data theory.
\end{remark}

The first step in the proof of Theorem \ref{thm - main thm intro} is to rewrite the equation so that the null structure of the system is easy to exploit. The standard way to do this is to use projections to reduce (\ref{eqn - NLD eqn}) to studying the scalar half wave equations $(\p_t \pm |\nabla|)$. In particular this method was used in the recent work of Pecher \cite{Pecher2013a} and Bejenaru-Herr \cite{Bejenaru2013}. In the current article, we instead work with a \emph{vector valued} formulation.  Working in the vector valued setting has two key advantages. The first is that it only makes use of the derivatives\footnote{This can be thought of as a higher dimensional analogue of the two $n=1$ formulations of the Dirac equation $(\p_t \pm \p_x)$ and $(\p_t \pm |\p_x|)$. The first formulation has the benefit that it is easy to write in the null coordinates $(t \pm x)$, and this is more or less the key property that led to the $L^2(\RR)$ critical result in \cite{Candy2010}.} $\p_\mu$ (as opposed to the Fourier multipliers $|\nabla|$), and thus behaves well under changes of coordinates. The second advantage is that, after an application of a \emph{Fierz type identity} \cite{Fierz1937}, the null structure hidden in the nonlinearities (\ref{eqn - cond on F}) manifests itself in products of vector valued waves traveling in opposite directions. On the other hand, the cost of avoiding the $\p_t \pm |\nabla|$ formulation of (\ref{eqn - NLD eqn}) and using a vector valued formulation, is that the function spaces we construct need to retain this vectorial information in order to be able to prove the bilinear estimates that are needed to close an iteration argument.

The second, and more difficult, step in the proof of Theorem \ref{thm - main thm intro} is to construct appropriate function spaces and prove a number of bilinear null form estimates. The spaces used are a combination of vector valued version of the null frame spaces of Tataru \cite{Tataru2001}, together with $X^{s, b}$ and energy type components. In slightly more detail, we define a norm that is schematically of the form $\| \psi \|_{F} = \| \psi \|_{L^\infty_t \dot{H}^\frac{n-1}{2}_x} + \| \gamma^\mu \p_\mu \psi \|_{Y}$ and take $Y = L^1_t \dot{H}^\frac{n-1}{2}_x + X^{-\frac{1}{2}, 1} + NF$ where $X^{-\frac{1}{2}, 1}$ is an $X^{s, b}$ type space with $\ell^1$ sum over distances to the cone, and $NF$ is based on the null frame spaces of Tataru \cite{Tataru2001}.  The construction of the required spaces and the study of their basic properties is rather involved, and takes up a significant portion of the current paper. However, we believe that these spaces should be applicable to other endpoint well-posedness results for related systems such as the Dirac-Klien-Gordon, Maxwell-Dirac, Chern-Simons-Dirac etc. We plan to return to this problem in the future. \\

To explain the key difficulties in the proof of Theorem \ref{thm - main thm intro}, note that the well-posedness theory for (\ref{eqn - NLD eqn}) would follow easily via standard energy estimates provided we had the $L^2_t L^\infty_x$ estimate
            \begin{equation}\label{eqn - intro L2Linfty est} \| \psi \|_{L^2_t L^\infty_x(\RR^{1+n})} \lesa \| \psi(0) \|_{\dot{H}^\frac{n-1}{2}_x(\RR^n) } \end{equation}
for solutions to (\ref{eqn - NLD eqn}) with $F=0$. Unfortunately, it is well-know that this estimate \emph{just} fails in the $n=3$ case, and is far from true in the $n=2$ case. Thus the low regularity well-posedness theory for (\ref{eqn - NLD eqn}) has more or less proceeded by trying to find suitable substitutes for the missing $L^2_t L^\infty_x$ Strichartz estimate (\ref{eqn - intro L2Linfty est}). One approach, used by Pecher \cite{Pecher2013a, Pecher2015}, is to move to the \emph{bilinear} setting, and exploit the additional structure of the nonlinearity (\ref{eqn - cond on F}) via bilinear estimates in $X^{s, b}$ type spaces. While this works in the  subcritical setting, it does not appear sufficient to handle the critical case $s = \frac{n-1}{2}$, see Remark \ref{rem - Xsb not sufficient to prove critical result}. An alternative approach used in the work of Machihara-Nakamura-Nakanishi-Ozawa \cite{Machihara2005a}, is to exploit the fact that (\ref{eqn - intro L2Linfty est}) is in fact true for \emph{radial} data in $n=3$. This is not quite enough on its own, as the Dirac equation does not commute with rotations, and thus radial data does not lead to radial solutions. Instead, Machihara-Nakamura-Nakanishi-Ozawa proved a version of (\ref{eqn - intro L2Linfty est}) with additional regularity in the angular variable.\\

The approach in the current article relies on the following crucial observation of Tataru \cite{Tataru2001}. Although the estimate (\ref{eqn - intro L2Linfty est}) fails for $n=2, 3$ (and $m=0$), the solution \emph{can} be place into spaces of the form $L^2 L^\infty$ provided we work in rotated null frames $(t_\omega, x_\omega)$ where $\sqrt{2} t_\omega =  (t, x) \cdot ( 1, \omega)$, $x_\omega = (t, x) - \tfrac{1}{\sqrt{2}}t_\omega ( 1, \omega)$ where $\omega \in \sph^{n-1}$ is a direction on the sphere. These spaces exploit the fact that if $\supp \widehat{f} \subset \{ |\xi| \approx \lambda, |\tfrac{\xi}{|\xi|} - \omega| \les (T \lambda)^{-\frac{1}{2}}\}$, then for times $|t|\les T$ we expect $(e^{i t |\nabla|} f)(x) \approx f( x + t \omega)$. A computation then shows that $\| \ind_{|t|<T}(t) e^{i t |\nabla|} f \|_{L^2_{t_\omega} L^\infty_{x_\omega}} \lesa (\frac{\lambda}{T})^{\frac{n-1}{2}} \| f\|_{L^2}$. Of course to exploit this concentration property, requires localising the Fourier support to small sets. Thus to control a general function, we need to use many frames simultaneously.

In the $n=1$ case, the gain formed by working in null frames is particularly easy to observe as the solution can only propagate in 2 directions $x \pm t$. More precisely, note that in the case $n=1$, we can write the solution to (\ref{eqn - NLD eqn}) as $\psi(t, x) = f(x - t) + g( x + t)$. Clearly $\psi \not \in L^2_t L^\infty_x$, however, we \emph{do} have
        $$ \| f( x - t) \|_{L^2_{x-t} L^\infty_{ x+ t}(\RR^{1+1})} = \| f \|_{L^2_x(\RR)}.$$
Thus despite the fact that (\ref{eqn - intro L2Linfty est}) fails, we can place our solution in spaces of the form $L^2_{t\pm x} L^\infty_{t \mp x}$. This simple observation played a key role in the $n=1$ proof of critical well-posedness \cite{Candy2010}. In higher dimensions, the solution can now travel in many directions $\omega \in \sph^{n-1}$, thus instead of a fixed frame $L^2_{t_\omega} L^\infty_{x_\omega}$, following the work of Tataru \cite{Tataru2001} we are forced to work in atomic Banach spaces made up of $\ell^1$ sums of $L^2_{t_\omega} L^\infty_{x_\omega}$ functions for various directions $\omega$. \\

It is worth comparing the results presented here with the work of Bejenaru-Herr \cite{Bejenaru2013} on the positive mass case $m>0$. There it was observed that if $m>0$ and $n=3$, then the estimate (\ref{eqn - intro L2Linfty est}) is true,  provided we localise to frequencies $\lesa 1$, or small angular caps. Unfortunately, while the additional dispersion given by the positive mass is helpful for small frequencies, the loss of scaling and the additional curvature of the characteristic surface complicates the analysis for high frequencies. In particular, to control the high frequency components of the evolution, the work of Bejenaru-Herr required the use of null frames adapted adapted to the  hyperboloid $\tau= \pm \sqrt{ |\xi|^2 + m}$.\\

The outline of the paper is as follows. In Subsection \ref{subsec - struc of Dirac} we rewrite the equation (\ref{eqn - NLD eqn}) in a more accessible form, and use this formulation to provide a simple proof of a bilinear null form estimate in $L^2_{t, x}$. The main notation used is introduced in Section \ref{sec - notation}. In Section \ref{sec - function spaces}, we define the function spaces used to prove Theorem \ref{thm - main thm intro}. The main linear estimates we require are stated in Section \ref{sec - linear est}. In Sections \ref{sec - bilinear est} and \ref{sec - cubic est} we prove our key  bilinear and trilinear estimates. The proof of Theorem \ref{thm - main thm intro} is then given in Section \ref{sec - proof of GWP}. In Sections \ref{sec - proof of null frame bounds} and \ref{sec - proof of stricharz est} we prove the linear estimates stated in Section \ref{sec - linear est}. Finally, in Section \ref{sec - energy inequality}, we prove a version of the energy inequality needed in the proof of Theorem \ref{thm - main thm intro}. \\

\noindent

\textbf{Acknowledgements.} The authors would like to thank Prof. Bejenaru and Prof. Herr for corrections and helpful conversations regarding the work \cite{Bejenaru2013}.\\

%------------------------------------------------------------------------------------------------------------------------------%
\subsection{Structure of Dirac equation}\label{subsec - struc of Dirac}
%------------------------------------------------------------------------------------------------------------------------------%

We take the standard representations of the Gamma matrices in the terms of the Pauli matrices
 $$ \sigma^1 = \begin{pmatrix}
      0 & 1 \\
      1 & 0
    \end{pmatrix}, \qquad \sigma^2 = \begin{pmatrix}
      0 & -i \\
      i & 0
    \end{pmatrix},\qquad  \sigma^3 = \begin{pmatrix}
      1 & 0 \\
      0 & -1
    \end{pmatrix}. $$
In particular, for $n=2$ we take
    $$ \gamma^0 = \sigma^3, \qquad \gamma^1 = i \sigma^2, \qquad \gamma^2 = -i \sigma^1$$
and if $n=3$ we let
    $$ \gamma^0 = \begin{pmatrix}
      I & 0 \\
      0 & -I
    \end{pmatrix}, \qquad \gamma^j = \begin{pmatrix} 0 & \sigma^j \\ -\sigma^j & 0 \end{pmatrix}, \qquad j=1, 2, 3, \qquad \gamma^5 = i \gamma^0 \gamma^1 \gamma^2 \gamma^3 = \begin{pmatrix} 0 & I \\ I & 0 \end{pmatrix}.$$
To proceed further, we note that we have the special case of a Fierz type identity\footnote{This is essentially a special case of a \emph{Fierz Identity} \cite{Fierz1937} which states that, in the $n=3$ case,  given $z_j \in \CC^4$ we have
    $$ (\overline{z_1} \gamma^\mu z_2) ( \overline{z_3} \gamma_\mu z_4) = (\overline{z_1} z_4) (\overline{z_3} z_2)  - \tfrac{1}{2} (\overline{z_1} \gamma^\mu z_4) (\overline{z_3}\gamma_\mu z_2) - \tfrac{1}{2} (\overline{z_1} \gamma^\mu \gamma^5 z_4) (\overline{z_3} \gamma_\mu \gamma^5 z_2) - (\overline{z_1} \gamma^5 z_4) (\overline{z_3} \gamma^5 z_2),$$
see also  \cite{Nieves2004}. The appendix to \cite{Ortin2004} contains the identities in general dimensions. Rearranging the Fierz identity easily gives the identity (\ref{eqn - spinor identity}). Alternatively one can show (\ref{eqn - spinor identity}) by first noting that for vectors $w_j \in \CC^2$ we have
        ${ \sum_{j =1}^3 (w_1^\dagger \sigma^j w_2 ) \sigma^j w_3 = 2 ( w_1^\dagger w_3) w_2 - ( w^\dagger_1 w_2) w_3 }$
and then computing the identity by hand. }
\begin{equation}\label{eqn - spinor identity} \big(\overline{\psi} \gamma^\mu \psi\big)   \gamma_\mu \psi = \begin{cases}
      \big(\overline{\psi} \psi \big)  \psi \qquad \qquad \qquad \qquad & n=2\\
      \big(\overline{\psi} \psi \big) \psi - (\overline{\psi} \gamma^5 \psi)\gamma^5 \psi &n=3.
    \end{cases}\end{equation}
This somewhat magical identity is the key to showing that the Thirring model nonlinearity is also a null form, and also shows that in $n=2$ the Thirring and Soler models are identical. Define
        $$ \sigma \cdot \nabla = \sigma^j \p_j = \begin{cases} \sigma^1 \p_1  + \sigma^2 \p_2   \qquad &n=2\\
                                                    \sigma^1 \p_1  + \sigma^2 \p_2 + \sigma^3 \p_3 &n=3 \end{cases} \qquad \text{ and } \qquad \beta = \begin{cases} \sigma^3 \qquad &n=2 \\
                0 \qquad &n=3. \end{cases}$$
We claim that (\ref{eqn - NLD eqn}) with $m=0$  is a special case of the system
    \begin{equation}\label{eqn - general form of eqn} \begin{split}
      (\p_t + \sigma \cdot \nabla) u  &= B_1(u, v) v + B_2(u, v) \beta u \\
      (\p_t - \sigma \cdot \nabla) v  &= B_3(u, v) u + B_4(u, v) \beta v
    \end{split}
    \end{equation}
where the $B_j(u, v)$ are a linear combination of the bilinear forms
    \begin{equation}\label{eqn - bilinear uv null forms} u^\dagger v, \qquad v^\dagger u, \qquad v^\dagger \beta v, \qquad u^\dagger \beta u\end{equation}
and $u, v : \RR^{1+n} \rightarrow \CC^2$. To prove the claim, if $n=3$ we decompose  $\psi = \begin{pmatrix} u + v \\ u-v \end{pmatrix}$ into left and right spinors\footnote{Essentially we are decomposing into standard left and right spinors $\psi = \psi_L + \psi_R$ where $\psi_R = \tfrac{1}{2}( I - \gamma^5) \psi$ and $\psi_L = \tfrac{1}{2}( I + \gamma^5) \psi$ and then writing $\psi_L = \begin{pmatrix} u \\ u \end{pmatrix}$ and $\psi_R = \begin{pmatrix} v \\ - v \end{pmatrix}$.}, then a short computation using the Fierz identity (\ref{eqn - spinor identity}) shows that the pair $(u, v)$ is a solution to (\ref{eqn - general form of eqn}) with $B_1(u, v) = B_3(u, v) = 2 i ( u^\dagger v + v^\dagger u)$ in the Soler model case, and $B_1(u, v) = 4 i (v^\dagger u)$, $B_3(u, v) = 4 i ( u^\dagger v)$ in the Thirring model case (note that $\beta = 0$ when $n=3$ so the $B_2$, $B_4$ terms vanish). On the other hand, in the $n=2$ case, if we multiply both sides of (\ref{eqn - NLD eqn}) by $\beta=\gamma^0 = \sigma^3$ and use the Fierz type identity (\ref{eqn - spinor identity}) together with $\gamma^0 \gamma^1 = \sigma^1$, $\gamma^0 \gamma^2 = \sigma^2$ we get (\ref{eqn - general form of eqn}) with $B_1 = B_3 = B_4 = 0$ and $B_2(u, v)= (u^\dagger \beta u) $. To summarise, Theorem \ref{thm - main thm intro} follows from the following.

\begin{theorem}\label{thm - main thm with u, v}
Let $n=2, 3$ and $s \g \frac{n-1}{2}$. There exists $\epsilon>0$ such that if $(u(0), v(0)) \in \dot{H}^\frac{n-1}{2}\cap \dot{H}^s(\RR^n)$ with
    $$\| u(0) \|_{\dot{H}^\frac{n-1}{2}(\RR^n)} + \| v(0) \|_{\dot{H}^\frac{n-1}{2}} < \epsilon$$
then we have a global  solution $(u, v) \in C\big( \RR, \dot{H}^\frac{n-1}{2}\cap \dot{H}^s(\RR^n)\big)$ to (\ref{eqn - general form of eqn}) such that
        $$ \sup_{t \in \RR} \,\,\| (u, v)(t) \|_{\dot{H}^\frac{n-1}{2}\cap \dot{H}^s(\RR^{n})} \lesa \| (u, v)(0)  \|_{\dot{H}^\frac{n-1}{2}\cap \dot{H}^s(\RR^n)}.$$
Moreover, the solution $(u, v)$ depends continuously on the initial data and  is the unique limit of smooth solutions. Finally, there exists $u_{\pm \infty}, v_{\pm \infty}\in C\big( \RR, \dot{H}^\frac{n-1}{2}\cap \dot{H}^s(\RR^n)\big)$ with $(\p_t + \sigma \cdot \nabla) u_{\pm \infty} = (\p_t - \sigma \cdot \nabla) v_{\pm \infty} = 0$ such that
        $$ \lim_{t \rightarrow \pm \infty} \Big( \big\| u(t) - u_{\pm \infty}(t) \big\|_{\dot{H}^\frac{n-1}{2}\cap \dot{H}^s(\RR^n)} + \| v(t) - v_{\pm \infty} \big\|_{\dot{H}^\frac{n-1}{2}\cap \dot{H}^s(\RR^n)}\Big) = 0 .$$
\end{theorem}

To prove Theorem \ref{thm - main thm with u, v}, we need to study the linear operator $(\p_t \pm \sigma \cdot \nabla)$. To start with, note that in the $n=2$ case, we have
    \begin{equation}\label{eqn - beta flips sign of linear eqn}( \p \pm \sigma \cdot \nabla) \beta = \beta (\p_t \mp \sigma \cdot \nabla).\end{equation}
In particular, if $(\p_t + \sigma \cdot \nabla) u = 0$, then $(\p_t - \sigma \cdot \nabla) \beta u = 0$. This has the important, and very useful, consequence that to study the nonlinearity in (\ref{eqn - general form of eqn}), it suffices to study products $u^\dagger v$ where $u$ and $v$ are solutions to
        \begin{equation}\label{eqn - u v eqns homogeneous}\begin{split}
          (\p_t + \sigma \cdot \nabla) u &= 0 \\
          (\p_t - \sigma \cdot \nabla) v &= 0
        \end{split}\end{equation}
since, clearly, products like $u^\dagger \beta u$ can be reduced to products of the form $u^\dagger v$ after an application of (\ref{eqn - beta flips sign of linear eqn}). We now claim that the product $u^\dagger v$ is a \emph{null form}, in other words it satisfies improved bilinear estimates when compared to a product like $|u|^2$. This is intuitively clear as since $u$ and $v$ should resemble waves traveling in opposite directions, we expect that their product should decay faster than a corresponding product like $|u|^2$. An estimate that makes this idea more explicit, is the following.

\begin{lemma}\label{lem - hom bilinear est} Let $n=1, 2, 3$. Assume $(u, v)$ solve (\ref{eqn - u v eqns homogeneous}) with $u(0) = f$ and $v(0) = g$. Then
        \begin{equation}\label{eqn - L2 bilinear est} \| u^\dagger v \|_{L^2_{t, x}(\RR^{n+1})} \lesa \| f \|_{L^2_x(\RR^n)} \|g \|_{\dot{H}^{\frac{n-1}{2}}(\RR^n)}. \end{equation}
\end{lemma}

Note that this estimate is certainly \emph{not} true for a product like $|u|^2$, if $n=1, 2$ this is easy to see as $u \not \in L^4_{t, x}(\RR^{1+2})$ for solutions to (\ref{eqn - u v eqns homogeneous}). It is also worth noting that (\ref{eqn - L2 bilinear est}) is closely related to the missing $L^2_t L^\infty_x$ Strichartz estimate. More precisely, if we had an $L^2_t L^\infty_x$ control over $v$, then the bilinear estimate (\ref{eqn - L2 bilinear est}) would follow from a simple application of H\"{o}lder's inequality. Thus, in some cases, Lemma \ref{lem - hom bilinear est} can form a suitable substitute to the missing endpoint Strichartz estimate. \\

One way to prove Lemma \ref{lem - hom bilinear est} (at least in the case $n=2, 3$) is to introduce potentials $\phi$ and $\varphi$ such that $$ (\p_t - \sigma \cdot \nabla) \phi = u, \qquad \qquad (\p_t + \sigma \cdot \nabla) \varphi = v.$$
Then a short computation shows that $\Box \phi = \Box \varphi = 0$ and furthermore, that $u^\dagger v$ is made up of a linear combination of the classical null forms
    $$ \p_t \phi \p_t \varphi - \nabla \phi \cdot \nabla \varphi, \qquad \p_\mu \phi \p_\nu \varphi - \p_\nu \phi \p_\mu \varphi.$$
Lemma \ref{lem - hom bilinear est} then follows by applying the sharp bilinear null form estimates of Foschi-Klainerman \cite{Foschi2000}. Alternatively, and more in the spirit of the current article, we present a softer argument that just relies on a decomposition into traveling waves, followed by H\"{o}lder's inequality and a change of variables. This is similar to the approach used by Tataru \cite{Tataru2001} and Klainerman-Rodnianski \cite{Klainerman2005}.

   \begin{proof}[Proof of Lemma \ref{lem - hom bilinear est}] In the $n=1$ case, we can reduce the estimate (\ref{eqn - L2 bilinear est}) to a product of the form $\| f(x-t)g(x+t) \|_{L^2_{t, x}(\RR^{1+1})}$ and so lemma follows by a simple change of variables. On the other hand, if $n=2, 3$ we begin by decomposing $v$ into an average of traveling waves. More precisely, define $\Pi_\omega = \frac{1}{2} \big( I + \sigma \cdot \omega\big)$ and $\widehat{\Pi_\pm f} = \Pi_{\pm \frac{\xi}{|\xi|}} \widehat{f}$, note that $\Pi_\omega^\dagger = \Pi_\omega$ and $\Pi_\omega^2 = \Pi_\omega$. Then writing the solution $v$ in Polar coordinates gives
     \begin{align} v(t, x) &= e^{ it |\nabla|} \Pi_+ g + e^{-i t |\nabla|} \Pi_- g \notag\\
                            %&=\int_{\RR^3} e^{  i (t|\xi|+ x \cdot \xi)} \widehat{\Pi_{\frac{\xi}{|\xi|}} g}(\xi) d\xi + \int_{\RR^n} e^{ %- i (t |\xi| - x \cdot \xi) } \widehat{\Pi_{-\frac{\xi}{|\xi|}} g}(\xi) d\xi \notag \\
                    &=\int_{ \sph^{n-1} } \Pi_{ \omega} \int_0^\infty e^{  i r (t  + x \cdot \omega)} \widehat{ g}(r\omega) r^{n-1}\, dr\, d\sph(\omega)  + \int_{ \sph^{n-1} }  \Pi_{-\omega} \int_0^\infty e^{  -i r (t  - x \cdot \omega) } \widehat{g}( r \omega) r^{n-1} \,dr\, d\sph(\omega)  \notag \\
                    &= \int_{\sph^{n-1} } \Pi_{\omega} g_\omega( t + x \cdot \omega )  \,d \sph(\omega)\label{eqn - hom - solution as average of free wave}
     \end{align}
where $g_\omega(a) = \int_0^\infty \big[e^{  i r a} \widehat{ g}(  r\omega) +  e^{ - i r a} \widehat{g}(- r \omega) \big] r^{n-1}\, dr$. Consequently, by the self-adjointness of the projections $\Pi_\omega$, we have the bound
    \begin{align*}
      \| u^\dagger v \|_{L^2_{t, x}} &\les  \int_{\sph^{n-1}} \big\| u^\dagger(t, x) \Pi_{\omega} g_\omega( t + x \cdot \omega )\big\|_{L^2_{t, x}} d\sph(\omega) \\
      &= \int_{\sph^{n-1}} \big\| \big(\Pi_\omega u\big)^\dagger(t-\omega \cdot x, x) g_\omega(t) \big\|_{L^2_{t, x}} d\sph(\omega) \\
      &\les  \sup_{\omega\in \sph^{n-1}} \big\| \big(\Pi_\omega u\big)(t-\omega \cdot x, x) \big\|_{L^\infty_t L^2_x} \int_{\sph^{n-1}} \|g_\omega\|_{L^2_t} d\sph(\omega).
         \end{align*}
It is easy enough to check that by undoing the Polar coordinates, and using an application of Holder in the $\omega$ variables we obtain
    $$ \int_{\sph^{n-1}} \|g_\omega\|_{L^2_t} d\sph(\omega) \lesa \| g \|_{\dot{H}^{\frac{n-1}{2}}}.$$
Thus we reduce (\ref{eqn - L2 bilinear est}) to proving
     $$  \sup_{\omega\in \sph^{n-1}} \big\| \big(\Pi_\omega u\big)(t-\omega \cdot x, x) \big\|_{L^\infty_t L^2_x} \lesa \| f \|_{L^2_x}.$$
Note that $(-x \cdot \omega, x)$ is a parameterisation of the null plane $NP(\omega)$ orthogonal to the null vector $(1, \omega)$,
        $$NP(\omega) = \big\{ (t, x) \in \RR^{1+n}\, \big| \,\, ( t, x) \cdot (1, \omega) = 0\, \big\}.$$
In other words, we need to control the integral of $\Pi_\omega u$ over $L^2\big(NP(\omega) \big)$. By a change of variables we deduce that
    \begin{align*}
      e^{ \mp i (t - x \cdot \omega) |\nabla|} f (x) =\int_{\RR^n}  \widehat{f}(\xi) e^{\mp i (t - x \cdot \omega) |\xi|} e^{ i x \cdot \xi} d\xi
            &= \int_{\RR^n} \Big[  \widehat{f}(\xi) e^{\mp i t |\xi|} J^{-1}(\xi)\Big](y) e^{i y \cdot x} dy
    \end{align*}
where $J(\xi) = 1 \pm \tfrac{\xi}{|\xi|} \cdot \omega \approx \theta(\omega, \mp \xi)^2$ is the Jacobian of the change of variables $y = \xi \pm |\xi| \omega$, and  $\theta(\xi, \xi')$ denotes the angle of the two vectors $\xi, \xi' \in \RR^n$. Hence using the ``null form'' estimate $|\Pi_\omega \Pi_{\pm \frac{\xi}{|\xi|}}| \lesa \theta(\omega, \mp \xi)$ (see (\ref{eqn - null structure estimate}) below) together with Plancheral, we have
    \begin{align} \big\|  e^{ \mp i (t - x \cdot \omega) |\nabla|} \Pi_\omega \Pi_\pm f (x) \big\|_{L^2_x} &= \Big\| \Big[  \Pi_\omega \Pi_{\pm\frac{\xi}{|\xi|}} \widehat{f}(\xi) e^{\mp i t |\xi|} J^{-1}(\xi)\Big](y)\Big\|_{L^2_y} \notag\\
                            &= \big\|  J^{-\frac{1}{2}}(\xi) \Pi_\omega \Pi_{\pm\frac{\xi}{|\xi|}} \widehat{f} \big\|_{L^2_\xi} \notag\\
                            &\lesa \big\| \theta(\omega, \mp \xi)^{-1} \theta(\omega, \mp \xi) \widehat{f} \big\|_{L^2_\xi} = \| f \|_{L^2}. \label{eqn - hom soln in null coord}
    \end{align}
If we apply this inequality to $u = e^{ i t |\nabla|} \Pi_- f + e^{-i t |\nabla|} \Pi_+ f$ we obtain (\ref{eqn - L2 bilinear est}). Thus lemma follows.
\end{proof}
\vspace{1cm}

%------------------------------------------------------------------------------------------------------------------------------%
%------------------------------------------------------------------------------------------------------------------------------%
\section{Notation}\label{sec - notation}
%------------------------------------------------------------------------------------------------------------------------------%
%------------------------------------------------------------------------------------------------------------------------------%

Throughout this article we take $n=2, 3$.  We use the notation $ a \lesa b$ to denote the inequality $a \les C b$ for some constant $C>0$ which is independent of the variables under consideration. Similarly, we write $a \ll b$ if $a \les C b$ with a small constant $C<\frac{1}{4}$.  For a complex valued $n\times m$ matrix $A$, we let $A^\dagger$ denote the conjugate transpose. If $\Omega \subset \RR^{1+n}$, we define $\ind_\Omega(t, x)$ to be the corresponding indicator function.  \\

Let $L^q_t L^r_x(\RR^{n+1})$ denote the usual mixed-norm Lebesgue space with the associated norm
    $$ \| u \|_{L^q_t L^r_x(\RR^{n+1})} = \Big( \int_\RR \Big[ \int_{\RR^n} |u(t, x)|^r dx \Big]^{\frac{q}{r}} dt \Big)^\frac{1}{q}.$$
Occasionally we omit the domain $\RR^{n+1}$ when we can do so without causing confusion. Most functions that occur in this paper are $\CC^2$ valued, although occasionally we make use of scalar valued maps as well. The Schwartz class of smooth functions on $\RR^{n}$  with rapidly decreasing derivatives is denoted by $\s(\RR^n)$, we let $\s'(\RR^n)$ denote its dual, the collection of all tempered distributions. For a function $f \in \s(\RR^n)$ we let
            $$ \widehat{f}(\xi) = \int_{\RR^n} f(x) e^{-i x \cdot \xi} dx$$
denote the spatial Fourier transform. Similarly for $u(t, x) \in \s(\RR^{n+1})$ we let $\widetilde{u}(\tau, \xi)$ denote the space-time Fourier transform. The Fourier transform is extended to $\s'$ by duality in the usual manner. For $s>-\frac{n}{2}$ we define the homogeneous Sobolev space $\dot{H}^s(\RR^n)$ as the completion of $\s$ using the norm
        $$ \| f \|_{\dot{H}^s(\RR^n)} = \big\| |\xi|^s \widehat{f}(\xi) \big\|_{L^2_\xi(\RR^n)}.$$
  Fix $\Phi \in C^\infty_0(\RR)$ with $\supp \Phi  \subset \{ 2^{-1} \les a \les 2\}$ and for $a \not = 0$
  \begin{equation}\label{eqn - defn of Phi sum} \sum_{\lambda \in 2^{ \ZZ}} \Phi( \lambda^{-1} a) = 1. \end{equation}
We define the (homogeneous) Besov-Lipschitz spaces $\dot{B}^s_{p, q}$ via the norm
        $$ \| f \|_{\dot{B}^s_{p, q}} = \bigg( \sum_{\lambda \in 2^\ZZ} \Big( \lambda^s \|P_\lambda f \|_{L^p} \Big)^q \bigg)^\frac{1}{q}$$
where $\widehat{P_\lambda f} = \Phi( \lambda^{-1} |\xi| ) \widehat{f}(\xi)$ is the Fourier cutoff to the region $|\xi| \approx \lambda$. Given a Banach space $X$, we let $C(\RR, X)$ denote the collection of all continuous maps $u: \RR \rightarrow X$.\\

Let $\sph^{n-1} = \big\{  x \in \RR^n\, \big| \,\, |x| = 1\, \big\}$ denote the standard unit sphere in $\RR^n$. If $\xi, \xi' \in \RR^n$, then we let $\theta(\xi, \xi')$ denote the positive, smallest, angle between the unit vectors $\tfrac{\xi}{|\xi|}, \tfrac{\xi'}{|\xi'|} \in \sph^{n-1}$.  We frequently use the estimate $\theta(\xi, \xi') \approx 1 - \tfrac{\xi}{|\xi|} \cdot \tfrac{\xi'}{|\xi'|}$ as well as the more explicit\footnote{This can be deduced by letting $\theta = \theta(\omega, \omega')$,
        $$|\omega - \omega'|^2 = 2 - 2 \omega \cdot \omega' = 2- 2 \cos(\theta) = 4 \sin^2\Big( \frac{\theta}{2}\Big).$$
  The estimate now follows by using the estimate $ \frac{\sin(a)}{a} x \les \sin(x) \les x$  for $0<x<a$ and the fact that $8\sin(\tfrac{1}{8}) \g \frac{49}{50}$.}
        \begin{equation}\label{eqn - sharp angle est}  \frac{49}{50} \theta(\omega, \omega') \les |\omega - \omega'| \les \theta(\omega, \omega') \end{equation}
which holds for $\omega, \omega' \in \sph^{n-1}$ provided $\theta(\omega, \omega') \les \frac{1}{4}$. Given a subset $\kappa \subset \sph^{n-1}$ and vector $\omega \in \sph^{n-1}$, we let $\theta(\omega, \kappa) = \inf \{ \theta(\omega, \omega') \,| \omega' \in \kappa\,\,\}$.\\

We often restrict the Fourier transform of a function to lie in a certain subsets of $\RR^{n+1}$. To exploit this restriction, we make use of Bernstein's inequality which states that if $\supp \widehat{f} \subset \Omega$ and $p \g 2$, then for any $q\les p$ we have
        $$ \| f \|_{L^p} \lesa |\Omega|^{\frac{1}{q} - \frac{1}{p}} \| f \|_{L^q}.$$
Similarly, when considering products, if  $\supp \widehat{f} \subset \Omega$ and $\supp \widehat{g} \subset \Omega'$, we observe that the product $fg$ satisfies $\widehat{fg} \subset \Omega + \Omega'$.

%------------------------------------------------------------------------------------------------------------------------------%
\subsection{Null Coordinates}
%------------------------------------------------------------------------------------------------------------------------------%

As mentioned in the introduction, the standard $(t, x)$ coordinate frame is not sufficient to give the bilinear estimates that we require in the present paper. Instead, to exploit the type of arguments leading used in the proof of Lemma \ref{lem - hom bilinear est}, we need the flexibility to be able to work in adapted \emph{null} coordinate frames which are chosen depending on the Fourier support of the function under consideration. The definitions are as follows. \\

Let $\omega \in \sph^{n-1}$ and $\vartheta = \frac{1}{\sqrt{2}}( 1, \omega) \in \RR^{1+n}$. We define the \emph{null coordinates} $(t_\omega, x_\omega) \in \RR \times \RR^n$ as
   $$ t_\omega =  (t, x) \cdot \vartheta = \frac{1}{\sqrt{2}} \big( t + \omega \cdot x \big), \qquad \qquad x_\omega = x - \tfrac{1}{\sqrt{2}}\big[ (t , x) \cdot \vartheta\big] \omega.$$
Note that $t_\omega \vartheta$ is the projection of $(t, x)$ onto the span of the null vector $\vartheta$, while $( - \omega \cdot x_\omega, x_\omega)$ is a parameterisation of the associated null hyperplane $\{ ( t, x) \in \RR^{1+n} \, |\, (t, x) \cdot \vartheta = 0\, \}$. Moreover we have the identity
        $$ (t, x) = t_\omega \vartheta  + ( - \omega \cdot x_\omega, \,x_\omega).$$
To facilitate the computations we use later, we also decompose $x_\omega =  x^\bot_\omega - \tfrac{1}{\sqrt{2}} x^1_\omega \omega$ where
        $$x^1_\omega = \tfrac{1}{\sqrt{2}} ( t - \omega \cdot x), \qquad x^\bot_\omega = x - (x\cdot \omega) \omega.$$
Thus $x^1_\omega$ denotes the component of the vector $(t, x)$ on the null cone in the direction $(1, -\omega)$, while $x^\bot_\omega$ is the remaining component orthogonal to $\omega$. We can translate from the $(t_\omega, x_\omega)$ coordinate frame back into the standard $(t, x)$ frame by using the identities
    \begin{align*}
        t &= \tfrac{1}{\sqrt{2}} t_\omega - \omega \cdot x_\omega        &&x = x_\omega +  \tfrac{1}{\sqrt{2}} t_\omega \,\omega \\
          &= \tfrac{1}{\sqrt{2}} (t_\omega + x^1_\omega)                    &&\phantom{x}= x^\bot_\omega + \tfrac{1}{\sqrt{2}} ( t_\omega -                                                                                                                x^1_\omega)\, \omega.
    \end{align*}
We also make use of the dual or frequency variables in null frames. If $(\tau, \xi)$ denote the Fourier variables associated to $(t, x)$, then we define the corresponding null frame versions $(\tau_\omega, \xi_\omega)$  by letting $(\tau, \xi) \cdot (t, x) = (\tau_\omega, \xi_\omega) \cdot (t_\omega, x_\omega)$. In other words we let
    $$ \tau_\omega = \frac{1}{\sqrt{2}} ( \tau + \xi \cdot \omega), \qquad \xi_\omega= \xi - \tau \omega = \xi^\bot_\omega - \sqrt{2} \xi^1_\omega \,\omega$$
where as before $\xi^\bot_\omega$ denotes the component of $\xi_\omega$ orthogonal to $\omega$, and $\xi^1_\omega = \frac{1}{\sqrt{2}} ( \tau- \xi \cdot \omega)$. We can translate from $(\tau_\omega, \xi_\omega)$ to $(\tau, \xi)$ by using the identities
    \begin{align*}
        \tau &= \tfrac{1}{\sqrt{2}} \tau_\omega - \tfrac{1}{2} \omega \cdot \xi_\omega        &&\xi = \xi_\omega +  ( \tfrac{1}{\sqrt{2}} \tau_\omega -\tfrac{1}{2} \omega \cdot \xi_\omega) \,\omega \\
          &= \tfrac{1}{\sqrt{2}} (\tau_\omega + \xi^1_\omega)                    &&\phantom{\xi}= \xi^\bot_\omega + \tfrac{1}{\sqrt{2}} ( \tau_\omega -                                                                                                                \xi^1_\omega)\, \omega.
    \end{align*}
Finally we note the fundamental fact that the symbol of the wave operator $\,\Box = \p^\gamma \p_\gamma\,$ satisfies the key inequality
        $$ \tau^2 - |\xi|^2 = 2 \tau_\omega \xi^1_\omega - |\xi^\bot_\omega|^2.$$
This simple identity plays an important role in the arguments used in this paper. \\

If we have a function $\phi(t, x)$ on $\RR^{1+n}$,  by default we use $(t, x)$ coordinates. If we want to specify that $\phi$ is in $(t_\omega, x_\omega)$ coordinates we write $\phi^*$, thus
        $$\phi(t, x) = \phi^*(t_\omega, x_\omega).$$
This convention also applies to the Fourier transform, $\widehat{\phi}(t, \xi)$ denotes the Fourier transform with respect to $x$, while $\widehat{\phi^*}(t_\omega, \xi_\omega)$ is the Fourier transform with respect to $x_\omega$. A similar comment applies to the spacetime Fourier transform $\widetilde{\phi}(\tau, \xi)$. \\

%------------------------------------------------------------------------------------------------------------------------------%
\subsection{The Projections $\Pi_\omega$ and $\Pi_\pm$.}
%------------------------------------------------------------------------------------------------------------------------------%

Let $\omega \in \sph^{n-1}$ and define the projections $\Pi_\omega$ by
        $$\Pi_\omega =  \frac{1}{2} \Big( I + \sigma \cdot \omega \Big)$$
where $I$ denotes the $2\times 2$ identity matrix. The properties of the matrices $\sigma$ implies that we have the important identities
    \begin{equation}\label{eqn - projection identities} I = \Pi_\omega + \Pi_{-\omega}, \qquad \sigma \cdot \omega =  \Pi_\omega - \Pi_{-\omega}, \qquad \Pi_\omega^\dagger = \Pi_\omega, \qquad \Pi_\omega \Pi_{-\omega}=0, \qquad \Pi_\omega^2 = \Pi_\omega.\end{equation}
Moreover we have the crucial (and well known) \emph{null structure} estimate $ \big| \Pi_\omega \Pi_{\omega'} \big| \lesa \theta(\omega, -\omega')$ which follows from the orthogonality of the projections $\Pi_{\pm \omega}$ by writing
    \begin{equation}\label{eqn - null structure estimate} \big|\Pi_\omega  \Pi_{\omega'} \big| = \big| \big( \Pi_\omega - \Pi_{-\omega'}\big) \Pi_{\omega'} \big| = \frac{1}{2} \big| (\omega + \omega') \cdot \sigma \Pi_{\omega'} \big| \lesa |\omega +  \omega'| \lesa \theta(\omega, -\omega').\end{equation}
This angle estimate plays a crucial role in eliminating a number of dangerous bilinear interactions. \\

Aside from using the projections $\Pi_\omega$ to exploit the null structure present in the Thirring model, they can also be used decompose the Dirac equation into half wave operators $\p_t \pm i |\nabla|$. More precisely, define the Fourier multipliers $\Pi_\pm$ as
        $$ \widehat{\Pi_\pm f}(\xi) = \Pi_{\pm \frac{\xi}{|\xi|}} \widehat{f}(\xi).$$
Then using the identities (\ref{eqn - projection identities}) we see that the Dirac equation $ (\p_t \pm \sigma \cdot \nabla) u = F$
is equivalent to
    \begin{align*}
      (\p_t \pm i |\nabla| ) \Pi_+ u &= \Pi_+ F \\
      (\p_t \mp i |\nabla|) \Pi_- u &= \Pi_- F.
    \end{align*}
This formulation for the Dirac equation has played a crucial role in the low regularity well-posedness theory developed over the last decade or so. See for instance the work of D'Ancona-Foschi-Selberg  \cite{D'Ancona2007b, D'Ancona2010a}, and Pecher
\cite{Pecher2006, Pecher2013a} and the second author \cite{Candy2013}, for the Dirac equation coupled to a scalar field,
as well as the of the current authors \cite{Bournaveas2012} for related ideas for the Spacetime-Monopole equation. \\

%------------------------------------------------------------------------------------------------------------------------------%
\subsection{Solution Operators}
%------------------------------------------------------------------------------------------------------------------------------%

Define the unitary operator $\mc{U}_\pm(t)$ on $L^2_x(\RR^n)$ by the formula
        $$ \mc{U}_\pm(t)[f] = e^{\mp i t |\nabla|} \Pi_+ f + e^{ \pm i t |\nabla|} \Pi_- f.$$
If we note that $\sigma \cdot \nabla =  i |\nabla| ( \Pi_+ + \Pi_-)$ then a short computation shows that
        $$ (\p_t \pm \sigma \cdot \nabla) \mc{U}_\pm(t)[f] = 0$$
and $\mc{U}_\pm(0) f = 0$. Thus $\mc{U}_\pm(t)f$ gives the homogeneous solution to $(\p_t \pm \sigma \cdot \nabla) u = 0$ with data $u(0) = f$. \\

%------------------------------------------------------------------------------------------------------------------------------%
\subsection{Sets and Multipliers}
%------------------------------------------------------------------------------------------------------------------------------%

The global well-posedness result in Theorem \ref{thm - main thm intro} depends on a number of sharp bilinear estimates. The proof of these bilinear estimates relies on being able to localise to certain frequency regions. The key tool to do this is the standard technique of dyadic decomposition. \\

 Take $\Phi \in C^\infty_0(\RR)$ as in (\ref{eqn - defn of Phi sum}) and let $\Phi_0(\xi) = \sum_{ \lambda \les 2^{-1}} \Phi( \tfrac{a}{\lambda})$ with $\Phi_0(0) = 1$. Define the Fourier multipliers $P_\lambda$, $C_d$, and $C^\pm_d$ via
    $$ \widehat{ P_\lambda f}(\xi) = \Phi\Big( \frac{|\xi|}{\lambda} \Big)\widehat{f}(\xi), \qquad \widetilde{C_d F}(\tau, \xi) = \Phi\Big( \frac{ \big| |\tau| - |\xi| \big|}{d} \Big) \widetilde{F}(\tau, \xi), \qquad \widetilde{C^\pm_d F}(\tau, \xi) = \Phi\Big( \frac{ \big|\tau \pm |\xi|\big|}{d} \Big) \widetilde{F}(\tau, \xi).$$
Note that $P_\lambda$ restricts the Fourier support to the set $\{ 2^{-1} \lambda \les |\xi| \les 2 \lambda\}$, $C_d$ restricts the Fourier support to be at distance $\approx d$ from the cone, and $C^\pm_d$ restricts the Fourier support onto the forward and backward components of the cone. Similarly we define multipliers $C_{\les d} $, $C^\pm_{\les d}$ as
  $$\widetilde{C_d F}(\tau, \xi) = \Phi_0\Big( \frac{ \big| |\tau| - |\xi| \big|}{d} \Big) \widetilde{F}(\tau, \xi), \qquad \widetilde{C^\pm_d F}(\tau, \xi) = \Phi_0\Big( \frac{\big| \tau \pm |\xi|\big|}{d} \Big) \widetilde{F}(\tau, \xi),$$
thus $C_{\les d}^\pm$ and $C_{\les d}$ are the (smooth) restriction of the Fourier support to the sets $\big\{ \big||\tau| - |\xi| \big| \les d \big\}$ and $\big\{ \big| \tau \pm |\xi| \big| \les d \big\}$. Note that if $F \in L^2(\RR^{1+n})$ we can decompose
            \begin{equation}\label{eqn - decomp of L^2 into distance from cone} F = \sum_{d \in 2^{\ZZ}} C^\pm_d F\end{equation}
where the sum converges in $L^2(\RR^{1+n})$. This is not true for $F \in L^\infty_t L^2_x$ for instance, as the the righthand side of (\ref{eqn - decomp of L^2 into distance from cone}) vanishes for functions with Fourier transforms supported on the lightcone, i.e. solutions to the wave equation. Thus some care has to be taken when decomposing functions into dyadic distances from the cone, as in general, (\ref{eqn - decomp of L^2 into distance from cone}) only holds modulo solutions to the wave equation. \\

The number of $\pm$ signs that will be floating around in various formula throughout this article can be daunting. To alleviate this somewhat, we define
        $$ \mf{C}^\pm_d = \Pi_+ C^\pm_d + \Pi_- C^\mp_d .$$
Thus $\mf{C}^\pm_d$ is the vector valued analogue of the $C^\pm_d$ multipliers. Note that $\mf{C}^\pm_d$ roughly corresponds to localising spacetime frequencies to distance $\sim d$ from the characteristic surface of the equation $(\p_t \pm \sigma \cdot \nabla) u = 0$. In a similar vein, we define
        $$ S^\pm_{\lambda, d} u =  \mf{C}^\pm_d P_\lambda u.$$
The multipliers $\mf{C}^\pm_{\les d}$ and $S^\pm_{ \lambda, \les d}$ are defined in the obvious manner. \\

As well as the above multipliers, we also need to be able to decompose into angular regions. Let $\alpha\ll 1$ and define $\mc{C}_\alpha$ to be a finitely overlapping cover of $\sph^{n-1}$ where every cap $\kappa \in \mc{C}_\alpha$ has radius $\alpha$. We use $\omega(\kappa)$ to denote the centre of the cap $\kappa \in \mc{C}_\alpha$ and so  $\kappa = \big\{ \omega \in \sph^{n-1} \, \big| \, \theta\big(\omega, \omega(\kappa)\big) \les \alpha \,\}$. For constants $C>1$ and $\kappa \in \mc{C}_\alpha$, we also define $C\kappa = \big\{ \omega \in \sph^{n-1}\, \big| \, \theta\big(\omega, \omega(\kappa)\big) \les C \alpha \big\}$.

Given a subset $\kappa \in \mc{C}_\alpha$ we define the sets
    $$A_\lambda(\kappa) = \big\{ (\tau, \xi) \in \RR^{1+n} \,  \big| \, \, \lambda 2^{-1} \les |\xi| \les 2 \lambda,\, \sgn(\tau) \tfrac{\xi}{|\xi|} \in \kappa \big\}, \qquad A^\pm_{\lambda}(\kappa) = \big\{\xi \in \RR^n \, \big| \,  \lambda 2^{-1} \les |\xi| \les 2 \lambda, \, \mp \tfrac{\xi}{|\xi|} \in \kappa\big\}$$
note that $A_{\lambda}(\kappa) \subset \RR^{1+n}$ while $A^\pm_{\lambda}(\kappa) \subset \RR^{n}$. These sets decompose the annulus $\{|\xi| \approx \lambda\}$ into radially directed, rectangularly shaped sets of size $\lambda \times (\alpha \lambda)^{n-1}$. Similarly we let
    $$A_{\alpha, \lambda}(\kappa) = \big\{ \lambda 2^{-1} \les |\xi| \les 2 \lambda,\, \big| |\tau| - |\xi| \big| \les c \alpha^2 \lambda, \, \sgn(\tau) \tfrac{\xi}{|\xi|} \in \kappa \big\},$$
and
     $$A^\pm_{\alpha, \lambda}(\kappa) = \big\{ \lambda 2^{-1} \les |\xi| \les 2 \lambda,\, \big| \tau \pm |\xi| \big| \les c\alpha^2 \lambda, \, \mp \tfrac{\xi}{|\xi|} \in \kappa \big\},$$
where $c\ll $ is some small constant. Clearly we have $\RR \times A^\pm_{\lambda}(\kappa) \subset A_{\lambda}(\kappa)$ and $A^{\pm}_{\alpha, \lambda} \subset A_{\alpha, \lambda}(\kappa)$. \\

For each of the angular sets defined above, we need the corresponding Fourier cutoffs. Fix $\alpha\ll 1$ and let $\Phi_\kappa$ be a smooth partition of unity on $\sph^{n-1}$ subordinate  to the caps $\kappa \in \mc{C}_\kappa$. Note that we may ensure that, after a rotation to centre the cap $\kappa$ on the $\xi_1$ axis, we have for $\xi \not = 0$ the derivative bounds
       \begin{equation}\label{eqn - derivative bounds for multipliers} \big|\p^N_{\xi_1} \big[\Phi_{\kappa}\big(\tfrac{\xi}{|\xi|})\big]\big| \lesa |\xi|^{-N}, \qquad \qquad \big|\p^N_{\xi_j} \big[\Phi_{\kappa}\big(\tfrac{\xi}{|\xi|})\big]\big| \lesa (\alpha |\xi|)^{-N}\qquad j\not = 1.\end{equation}
We now define the corresponding Fourier multiplier
        $$\widehat{{R^\pm_\kappa f}}(\xi) = \Phi_\kappa\big( \mp \tfrac{\xi}{|\xi|} \big) \widehat{f}(\xi)$$
and take
        $$R^\pm_{\kappa, \, d} = C^\pm_{\ll d} R^\pm_\kappa, \qquad  P_{\lambda, \kappa}^\pm = P_\lambda R^\pm_\kappa,  \qquad P^{\pm, \alpha}_{\lambda, \kappa} = C^\pm_{\ll \alpha^2 \lambda} P_\lambda R^\pm_\kappa.  $$
The multipliers and corresponding sets are summerised in Table \ref{table - sets + fourier multipliers}. \\

\begin{table}[h!]\centering
\begin{tabular}{|c|c|}
\hline
Sets & Fourier multipliers \\[2pt]
\hline
   $\{ |\xi| \approx \lambda \}$ & $P_\lambda$ \\[2pt]
   $\big\{ \big| |\tau| - |\xi| \big| \approx d \big\}$ & $C_d$ \\[2pt]
   $\big\{ \big| \tau \pm |\xi| \big| \approx d \big\}$ & $C^\pm_d$ \\[2pt]
   $\big\{ \mp\tfrac{\xi}{|\xi|} \in \kappa\big\} $ & $R_\kappa^\pm$ \\[2pt]
   $\big\{ \big| \tau \pm |\xi| \big| \ll d \,\, , \mp \tfrac{\xi}{|\xi|} \in \kappa\big\} $ & $R_{\kappa, d}^\pm$ \\[2pt]
   $A^\pm_{\lambda}(\kappa)$ & $P^\pm_{\lambda, \kappa}$ \\[2pt]
   $A^{\pm}_{\alpha, \lambda}( \kappa)$ & $P^{\pm, \alpha}_{\lambda, \kappa}$ \\[2pt]
  \hline
  %lots of \rule{0pt}{2ex} were added to the above table in front of $A_\lambda(\kappa)$ etc... not sure why?
  %\vspace{0.1cm}
\end{tabular}\vspace{0.3cm}
 \caption{Sets and corresponding Fourier multipliers}\label{table - sets + fourier multipliers}
\end{table}

We would like to pretend that the operators introduced about are idempotent, i.e. satisfy $P^2 = P$. Unfortunately, this clearly fails (although it is \emph{almost} the case, in the sense that $P^2$ is a cutoff to the same region of frequency space). Thus, to work around this difficulty, we introduce cutoffs to slight enlargements of the sets used above. More precisely, if $A$ is one the sets defined above, then we let $^{\natural}A$ be the set which is $\frac{101}{100}$ times larger, thus $A \subset    {^{\natural}A}$. For example, we let
    $$^{\natural} A^\pm_{\lambda}(\kappa) = \{ \, \lambda 2^{-1} \tfrac{100}{101} \les |\xi| \les 2\lambda \tfrac{101}{100}, \,\, \tfrac{\xi}{|\xi|} \in \tfrac{101}{100} \kappa \}.$$
The sets $^{\natural}A_{\lambda}(\kappa)$, $^{\natural}A^\pm_{\lambda, \alpha}(\kappa)$, and $^{\natural}A_{\lambda, \alpha}(\kappa)$ are defined similarly. Moreover, if $A$ is one of previous sets,  we let $^{\natural}P$ denote a corresponding multiplier that is $1$ on $A$, and has support inside the corresponding set $^{\natural}A$. For instance $^{\natural}P^\pm_{\lambda, \kappa}$ restricts the Fourier transform to the set $^{\natural}A^\pm_\lambda(\kappa)$. Note that we always have identities of the form $^{\natural}P^\pm_{\lambda, \kappa} P^\pm_{ \lambda, \kappa} = P^\pm_{\lambda, \kappa}$ and furthermore we may assume that the new multipliers  $^{\natural}R^\pm_{\kappa}$ still satisfy the derivative bounds (\ref{eqn - derivative bounds for multipliers}).

%------------------------------------------------------------------------------------------------------------------------------%
\subsection{Estimate on coordinates in $A_{\alpha, \lambda}(\kappa)$}\label{subsec - est on dual coords}
%------------------------------------------------------------------------------------------------------------------------------%

For later use, we record here the following useful estimate on the dual coordinates $(\tau_\omega, \xi_\omega)$. We start by noting that
    \begin{equation}\label{eqn - estimate on dual coordinates no supp restrict}
       \big| \xi^1_\omega\big| =\frac{1}{\sqrt{2}} \big| |\tau| - |\xi| + |\xi| - \sgn(\tau) \xi \cdot \omega\big|, \qquad |\xi^\bot_\omega|^2 = \big| |\xi| + \xi \cdot \omega \big|\times \big| |\xi| - \xi \cdot \omega \big|, \qquad |\tau_\omega| \les |\tau| + |\xi|.
    \end{equation}
In particular, if $\alpha \ll 1$,  $\kappa \in \mc{C}_\alpha$, and $(\tau, \xi) \in A_{\lambda, \alpha}(\kappa)$ then have
    \begin{equation}\label{eqn - estimate on dual coordinates on the sets A} |\xi^1_\omega| \lesa  \big(\max\{ \alpha, \theta(\omega, \kappa)\}\big)^2 \lambda, \qquad |\xi^\bot_\omega| \lesa \theta(\omega, \kappa) \lambda , \qquad |\tau_\omega| \lesa \lambda.\end{equation}
Clearly the same bounds also hold for $(\tau, \xi) \in {^{\natural}A_{\lambda, \alpha}(\kappa)}$.

A slightly sharper estimate is available if $\omega \not \in 2 \kappa$. More precisely, the additional assumption on $\omega$ implies that $\alpha \les \theta(\omega, \kappa)$ and so $\big| |\tau| - |\xi| \big| \ll \theta^2(\omega, \kappa) \lambda$. Consequently
    \begin{equation}\label{eqn - estimate on dual coordinates on the sets A with omega restriction} |\xi^1_\omega| \approx  \theta^2(\omega, \kappa)\lambda, \qquad |\xi^\bot_\omega| \lesa   \theta(\omega, \kappa)\lambda, \qquad |\tau_\omega| \lesa \lambda.\end{equation}

%------------------------------------------------------------------------------------------------------------------------------%
%------------------------------------------------------------------------------------------------------------------------------%
\section{Function spaces}\label{sec - function spaces}
%------------------------------------------------------------------------------------------------------------------------------%
%------------------------------------------------------------------------------------------------------------------------------%

A standard method used to handle the critical wave equation, is to take a Banach space $Y$, and then define a norm at scale $\lambda$ via
            $$\| u \|_{F^\pm_\lambda} = \| P_\lambda u \|_{L^\infty_t L^2_x} + \| (\p_t \pm \nabla\cdot \sigma)P_\lambda u\|_{Y}$$
A good first choice for $Y$, (one that has worked well for the critical wave equation in high dimensions), is to take
        $$ Y= L^1_t L^2_x + X^{-\frac{1}{2}, 1}$$
where $X^{b, 1}_\lambda$ is the relevant $X^{s, b}$ type space with an $\ell^1$ sum in the distance to the cone, at scale $|\xi| \approx \lambda $. The idea is that away from the light cone, we use the $X^{s, b}$ type spaces, while close to the light cone, where the symbol blows up, we use the $L^1_t L^2_x$ type norm.

If now apply our $Y$ type norm to our well-posedness problem,  the essentially point would be to control the term
            $$ \big\| (\p_t + \sigma \cdot \nabla)^{-1} \big[(u^\dagger v) v\big] \big\|_{Y}. $$
 Let $F=u^\dagger v$. Since the product $u^\dagger v$ is a null form, we should be able to put $F \in L^2_{t, x}$ (i.e. as in Lemma \ref{lem - hom bilinear est}). If we also let $F_\lambda= P_\lambda F$, and $v_\mu = P_\mu v$, we need to prove estimates of the form
        \begin{equation}\label{eqn - hi-lo case intro} \| F_\lambda v_\mu \|_{Y} \lesa \mu \| F_\lambda \|_{L^2_{t,  x}} \| v_\mu \|_{F^-}. \end{equation}
This estimate is essentially true for $Y=L^1_t L^2_x + X^{-\frac{1}{2}, 1}_\lambda$ except for one particularly bad case where the output $F_\lambda v_\mu$ is concentrated near the null cone (so we are forced to use the $L^1_t L^2_x$ space), $v_\mu$ is also close to null cone (so is essentially a homogeneous solution), but $F$ is far from the null cone. Our only option is put $F_\lambda v_\mu$ in $L^1_t L^2_x$, but then since $F_\lambda \in L^2$, we need $v_\mu \in L^2_t L^\infty_x$ which fails (since $v_\mu$ is essentially a homogeneous solution). Note that this interaction is not a null interaction (as $F_\lambda$ is far from the cone) so null structure doesn't help.

The key observation, due to Tataru, is that we do have a $L^2 L^\infty$ type estimate, provided we look at null coordinates $(t_\omega, x_\omega)$ instead. This means that we can control the product in $L^1_{t_\omega} L^2_{x_\omega}$ \textit{but not} $L^1_t L^2_x$. Thus, for certain interactions, we need to replace the $L^1_t L^2_x$ component of the $Y$ norm, with a $L^1_{t_\omega} L^2_{x_\omega}$ type norm instead. This is possible but the construction of the required function spaces is a little involved.\\

In the rest of this section, we construct an appropriate replacement for the space $Y$. Essentially we will take $Y$ to be roughly $ L^1_t L^2_x + X^{s, b} +NF$ with an added term to deal with the regions far from the cone. Here $NF$ are the null frame spaces originally appearing in the work of Tataru \cite{Tataru2001}, and developed further by Tao \cite{Tao2001a}. See also the results in \cite{Krieger2003, Krieger2012, Sterbenz2010, Sterbenz2010a}.\\

%------------------------------------------------------------------------------------------------------------------------------%
\subsection{$X^{s, b}$ type norms}
%------------------------------------------------------------------------------------------------------------------------------%

We define the Dirac version of the Bourgain-Klainerman-Machedon spaces by using the norm\footnote{Note that $\| \cdot \|_{\dot{\mc{X}}^{b, q}_\pm}$ is not technically a norm, as it vanishes for distributions with Fourier support on the cone, thus it is only a $\emph{semi-norm}$. However we make the (fairly) standard abuse of notation and refer to all semi-norms as norms.}
            $$ \| u \|_{\dot{\mc{X}}^{b, q}_{ \pm}} = \Big( \sum_{d \in 2^{\ZZ}} d^{qb} \big\| \mf{C}^\pm_d u \big\|_{L^2_{t, x}}^q\Big)^\frac{1}{q}.$$
The norm $\dot{\mc{X}}^{b, q}_{ \pm}$ is related to the more standard norms   $\dot{X}^{b, q}_{\pm}$  adapted to the cone $\{ \tau \pm |\xi| = 0\}$ by the formula
           \begin{equation}\label{eqn - equivalence of Xsb defns} \| u \|_{\dot{\mc{X}}^{b, q}_{ \pm}} \approx \| \Pi_+ u \|_{\dot{X}^{b, q}_{\pm}} + \| \Pi_- u \|_{\dot{X}^{b, q}_{\mp}} \end{equation}
where
                $$ \| u \|_{\dot{X}^{b, q}_{\pm}} = \Big(\sum_{ d \in 2^{\ZZ}} d^{qb} \| C_d^{\pm} u \|_{L^2_{t,x}}^q \Big)^{\frac{1}{q}}$$
and we recall that $\mf{C}^\pm_d =  C^\pm_d \Pi_+ + C^\mp_d \Pi_-$. Note the this implies that $\| \Pi_+ u \|_{\dot{\mc{X}}^{b, q}_{ \pm}} \approx \| \Pi_+ u \|_{\dot{X}^{b, q}_{ \pm}}$ and a similar equality in the $\Pi_-$ case. The equivalence of the two norms (\ref{eqn - equivalence of Xsb defns}) follows by simply using the self-adjointness of the projections $\Pi_\pm$ to obtain $ \int (\Pi_\pm u)^\dagger \Pi_\mp v dx = \int u^\dagger \Pi_\pm \Pi_\mp v dx = 0$.

The $\dot{\mc{X}}^{b, q}_{\pm}$ norm is designed to exploit the fact that, at least for small times or small data, we expect that the $\Pi_+$ component of the solution to
    $$ (\p_t  \pm \sigma \cdot \nabla) u = F$$
to concentrate close to the cone $\{ \tau \pm |\xi| = 0\}$. The $X^{s, b}$ type norms have been a standard tool in the low regularity theory of nonlinear dispersive PDE since the work of Bourgain \cite{Bourgain1993a}, Kenig-Ponce-Vega \cite{Kenig1993}, and  Klainerman-Machedon \cite{Klainerman1995}. See also the earlier work of Beals \cite{Beals1983} who used similar spaces in the study of singularity formation for the nonlinear wave equation.\\

We make use of the following basic results.
\begin{lemma}\label{lem - Xsb decomp into free waves}
Let $u \in L^2_{t, x}$ with $\supp \widetilde{\Pi_+ u} \subset \big\{ \big| \tau \pm |\xi| \big| \approx d \}$ and  $\supp \widetilde{\Pi_- u} \subset \big\{ \big| \tau \mp |\xi| \big| \approx d \}$.
Then we can write
       \begin{equation}\label{eqn - Xsb decomp of u into average of free waves}  u(t, x) = \frac{1}{2\pi} \int_{|\tau| \approx d} e^{ i t \tau}\, \mc{U}_\pm(t)[f_\tau] d\tau \end{equation}
where $f_\tau \in L^2_x$ has the same $\xi$ support as $\widehat{u}$, and $\| f_\tau \|_{L^2_x} \lesa \| \widetilde{u}(\tau, \xi) \|_{L^2_\xi}$.
\begin{proof}
  We simply let
        $$ f_\tau(x) = \frac{1}{(2 \pi)^n} \int_{\RR^n} \big[ \widetilde{\Pi_+ u}(\tau \mp |\xi|, \xi)  + \widetilde{\Pi_- u}(\tau \pm |\xi|, \xi) \big]e^{ i x \cdot \xi}\, d \xi.$$
\end{proof}
\end{lemma}

The identity (\ref{eqn - Xsb decomp of u into average of free waves}) easily implies the well-known \emph{transference principle}. Namely,  if for every $\tau \in \RR$  we have the bound $\| e^{ i t \tau}  \mc{U}_\pm(t) f \|_X \lesa \| f \|_{L^2_x}$, then $\| u \|_X \lesa \| u \|_{\dot{\mc{X}}^{\frac{1}{2}, 1}_\pm}$ (for $u \in L^2_{t, x}$ say). In other words, any homogeneous estimate for the Dirac equation, immediately implies the same estimate holds for elements of $\dot{\mc{X}}^{\frac{1}{2}, 1}_\pm$. See for instance \cite[Proposition 5.1]{Tataru2001} or \cite[Proposition 3.7]{Klainerman2002}. In particular, $\dot{\mc{X}}^{\frac{1}{2}, 1}_\pm$ controls the Strichartz norms $L^q_t L^r_x$. More precisely, if say $u \in L^2_{t, x}$ with $\supp \widehat{u} \subset \{ |\xi| \approx \lambda\}$, then for any $\frac{1}{q} + \frac{n-1}{2 r} \les \frac{n-1}{4}$ with $(q, r) \not = (2, \infty)$ we
have\footnote{After an application of the triangle inequality and scaling, it is enough to consider the case $u = \mathfrak{C}^\pm_1 u$. We now apply (\ref{eqn - Xsb decomp of u into average of free waves}) followed by the homogeneous Strichartz estimate  to deduce
  $$ \| u \|_{L^q_t L^r_x} \les \int_{|\tau| \approx 1} \| \mc{U}_\pm(t) f_{\tau} \|_{L^q_t L^r_x} d\tau \lesa \lambda^{n(\frac{1}{2} - \frac{1}{r}) - \frac{1}{q}} \int_{|\tau| \approx 1}  \| f_\tau \|_{L^2_x} d\tau \lesa \lambda^{n(\frac{1}{2} - \frac{1}{r}) - \frac{1}{q}} \| u \|_{L^2_{t, x}}.  $$
}
        \begin{equation}\label{eqn - Xsb controls strichartz}
        \| u \|_{L^q_t L^r_x} \lesa \lambda^{n ( \frac{1}{2} - \frac{1}{r}) - \frac{1}{q}} \| u \|_{\dot{\mc{X}}^{\frac{1}{2}, 1}_\pm}
        \end{equation}
The transference principle can save a significant amount of work when working with $\dot{\mc{X}}^{\frac{1}{2}, 1}_\pm$ type norms. Finally, we recall the well-known fact that after truncating in time,  homogenous solutions belong to  $\dot{\mc{X}}^{\frac{1}{2}, 1}_\pm$.

\begin{lemma}[Homogeneous solutions belong to $\dot{\mc{X}}^{\frac{1}{2}, 1}_\pm$]\label{lem - Xsb, homogeneous solns, and time cutoffs}
Let $\rho \in C^\infty_0(\RR)$ and $T>0$. Then
    $$ \big\| \rho\big(\tfrac{t}{T}\big) \mc{U}_\pm(t) f \big\|_{\dot{\mc{X}}^{\frac{1}{2}, 1}_{\pm}} \lesa \| f \|_{L^2_x}$$
where the constant is independent of $T$.
\begin{proof}
  Let $\rho_T(t) = \rho( \frac{t}{T})$. If we observe that $\widetilde{[\rho_T(t) \mc{U}_\pm(t) f]}(\tau, \xi) = \widehat{\rho}( \tau \pm |\xi|) \widehat{\Pi_+ f}(\xi) +  \widehat{\rho}( \tau \mp |\xi|) \widehat{\Pi_- f}(\xi)$ then
    $$ \big\| \rho\big(\tfrac{t}{T}\big) \mc{U}_\pm(t) f \big\|_{\dot{\mc{X}}^{\frac{1}{2}, 1}_{\pm}} \les 2\| \rho_T \|_{\dot{B}^\frac{1}{2}_{2, 1}} \| f \|_{L^2_x}. $$
  Hence result follows by recalling that $\| \rho_T \|_{\dot{B}^{\frac{1}{2}}_{2, 1}(\RR)} = \| \rho \|_{\dot{B}^\frac{1}{2}_{2, 1}(\RR)}<\infty$.
\end{proof}
\end{lemma}

\begin{remark}\label{rem - Xsb not sufficient to prove critical result}
  An obvious question immediately arises, namely, can we simply prove Theorem \ref{thm - main thm with u, v} by iterating the equation in the norm $\dot{\mc{X}}^{\frac{1}{2}, 1}_\pm$? In other words, we are asking if we can bound the cubic term in $\dot{\mc{X}}^{-\frac{1}{2},1}_\pm$, which in view of Lemma \ref{lem - hom bilinear est} and the transference principle,  would more or less require the estimate
                \begin{equation}\label{eqn - Xsb failure} \| F_\lambda v_\mu \|_{\dot{\mc{X}}_+^{-\frac{1}{2}, 1}} \lesa \mu \| F_\lambda \|_{L^2_{t, x}} \| v_\mu \|_{\dot{\mc{X}}^{\frac{1}{2}, 1}_-} \end{equation}
  for $\mu \ll \lambda$.   Unfortunately, (\ref{eqn - Xsb failure}) fails. This can be seen by making the choice $\widetilde{F} = \chi_{\Omega_1}$, $\widetilde{v}=\chi_{\Omega_2}$ where
        $$ \Omega_1 = \{ \lambda - 4\les |\tau| \les \lambda + 4, \, \lambda - 4 \les |\xi| \les \lambda + 4 \}, \qquad \Omega_2= \{ |\tau| \les 1, \, 2\les |\xi| \les 3 \}.$$
  Note that if $ \frac{1}{2} d \les |\tau \pm |\xi|| \les 2 d$ and $\lambda \les |\xi| \les \lambda + 1$, (for $d\ll1 $ say), and $(\tau', \xi') \in \Omega_2$, then $(\tau - \tau', \xi - \xi') \in \Omega_1$ since
            $$ |\tau - \tau'| \les |\tau \pm |\xi| | + |\xi| + |\tau'| \les 2d + \lambda + 1 + 1 \les \lambda + 4$$
  and similarly
        $$ |\tau - \tau'| \g |\xi| - |\tau \pm |\xi| | - |\tau'| \g \lambda - 2d - 1 \g \lambda - 4.$$
The argument for the $\xi - \xi'$ variable is similar. Therefore,
    $$ \| C_d( F v) \|_{L^2_{t, x}} = \Big\| \int_{\Omega_2} \chi_{\Omega_1}(\tau - \tau', \xi - \xi') d\tau'd\xi' \Big\|_{L^2_{\tau, \xi}(|\tau \pm |\xi| \approx d)} \gtrsim |\Omega_1|  \,\big| \{ |\tau \pm |\xi|| \approx d, \lambda\les |\xi| \les \lambda + 1\}\big|^\frac{1}{2} \approx d^\frac{1}{2}$$
and consequently
        $$ \| F v \|_{X^{-\frac{1}{2}, 1}} \g \sum_{d\ll 1}  d^{-\frac{1}{2}} d^\frac{1}{2} = \infty.$$
On the other hand it is easy to check that the righthand side of (\ref{eqn - Xsb failure}) is finite, thus (\ref{eqn - Xsb failure}) fails. We make the remark that this counterexample does not include interactions close to the cone, thus null structure would not help. To summarise, endpoint $X^{s, b}$ type spaces together with bilinear estimates, do not appear to be enough to obtain critical well-posedness results. \end{remark}

%-----------------------------------------------------------------------------------------------------------------------------%
\subsection{Atomic Banach Spaces}
%-----------------------------------------------------------------------------------------------------------------------------%

The remaining function spaces used in this article have a complicated structure as they need to capture certain space-time integrability properties of our solution in arbitrary null frames. The method to define these spaces, going back to the work of Tataru \cite{Tataru2001}, is via an \emph{atomic} construction. The standard set up is as follows. We start with a subset $E \subset \s'$ such that for every  $\phi \in \s$ we have
        \begin{equation}\label{eqn - gen cond on atoms} \sup_{ f \in E} \big| f (\phi ) \big| <\infty. \end{equation}
The set $E$ consists of our \emph{atoms}. We then define the atomic Banach space $\mc{A}(E)$ as
        \begin{equation}\label{eqn - defn atomic banach space} \mc{A}(E) = \Big\{ \sum_{j \in \NN} c_j f_j \,\, \Big| \,\, (c_j)_{j \in \NN} \in \ell^1(\NN), \,\, f_j \in E \,\Big\} \end{equation}
with the norm
        \begin{equation}\label{eqn - defn atomic banach space norm} \| f \|_{\mc{A}(E)} = \inf \Big\{\, \sum_{j \in \NN} |c_j| \, \, \Big| \,\, f = \sum_j c_j f, \,\,\, (c_j)_{j \in \NN} \in \ell^1(\NN), \,\,\, f_j \in E \,\, \Big\}
        %\inf_{ \substack{ f = \sum_j c_j f_j \\ c_j \in \ell^1, \,\, f_j \in E }} \sum_j |c_j|.
        \end{equation}
It is easy to check that provided $c_j \in \ell^1(\NN)$ and $f_j \in E$, the condition (\ref{eqn - gen cond on atoms}) implies that the sum $\sum_j c_j f_j$ converges in $\s'$ and thus $\mc{A}(E)$ is a well-defined subset of $\s'$. Moreover  a standard computation shows that $\| \cdot \|_{\mc{A}(E)}$ is indeed a norm on $\mc{A}(E)$ (which is stronger than the standard Schwartz topology on $\s'$), and the pair $(\mc{A}(E), \| \cdot \|_{\mc{A}(E)})$ form a Banach space. \\

Given a linear operator $T$, and a Banach space $X \subset \s'$, we often need to prove inequalities of the form
    \begin{equation}\label{eqn - gen A(E) controls X}
            \big\| T f \big\|_X \lesa \| f \|_{\mc{A}(E)}.
    \end{equation}
In general, this can be broken down into two steps. The first step is to show that if $f = \sum_{ j \in \NN} c_j f_j$ is a decomposition of $f$ into atoms $f_j \in E$, then
        \begin{equation}\label{eqn - gen Tf = sum atoms}  Tf = \sum_{j \in \NN} c_j Tf_j \end{equation}
with convergence in $\s'$ say. The second is to obtain (\ref{eqn - gen A(E) controls X}) in the special case where  $f \in E$ is an atom. In other words show that we have the bound
        \begin{equation}\label{eqn - gen boundedness on atoms} \sup_{ f \in E} \big\| T f \big\|_X \lesa 1. \end{equation}
It is a simple exercise to show that (\ref{eqn - gen Tf = sum atoms}) and  (\ref{eqn - gen boundedness on atoms}), together with the uniqueness of limits in $\s'$, implies the bound (\ref{eqn - gen A(E) controls X}). Note that in general, it is \emph{not} true that boundedness on atoms (\ref{eqn - gen boundedness on atoms}) directly implies the bound (\ref{eqn - gen A(E) controls X}), see for instance  \cite{Bownik2005} for an example related to the Hardy space. Thus some care has to be taken to first check the identity (\ref{eqn - gen Tf = sum atoms}) as well as the boundedness on atoms. However, in the arguments used in the current paper, the identity (\ref{eqn - gen Tf = sum atoms}) is almost immediately, and thus we often leave the proof of (\ref{eqn - gen Tf = sum atoms}) to the reader. The reduction of (\ref{eqn - gen A(E) controls X}) to (\ref{eqn - gen boundedness on atoms}) is used frequently in the arguments to follow.\\

As a special case of (\ref{eqn - gen A(E) controls X}), note that if $X \subset \s'$ is a Banach space with $E \subset \{ \| f \|_{X} \lesa 1 \}$  (thus the set of atoms is contained inside the unit ball of $X$), then we immediately deduce the continuous embedding $\mc{A}(E) \subset X$. Conversely, if the unit ball of $X$ is contained in the set of atoms $E$, then we have $X \subset \mc{A}(E)$. Of course this condition can be weakened considerably, for instance if $E$ contains a \emph{dense} subset of $\{ \| f \|_{X}\les 1\}$, then we still have $X \subset \mc{A}(E)$. See \cite{Bonsall1991} for a more general result of this nature.

%------------------------------------------------------------------------------------------------------------------------------%
\subsection{Null Frame spaces - $NF^\pm(\kappa)$, $PW^\pm(\kappa)$, and $[NF^\pm]^*(\kappa)$.}
%------------------------------------------------------------------------------------------------------------------------------%

As mentioned previously, the wave equation satisfies improved regularity properties in certain null frames $(t_\omega, x_\omega)$. However, we cannot pick a fixed frame $(t_\omega, x_\omega)$ to work in, and instead have to work in certain \emph{averages} over directions $\omega \in \kappa$. The fact that we have to control our solution in many coordinates frames simultaneously forces us to use the rather complicated atomic construction (\ref{eqn - defn atomic banach space}) and  (\ref{eqn - defn atomic banach space norm}) to define the necessary spaces. The construction below is heavily based on the original work of Tataru on the wave maps problem \cite{Tataru2001}. Accordingly we follow, as much as possible, the notation introduce in \cite{Tataru2001}.  \\

The first null frame space we introduce is based on $L^1_{t_\omega} L^2_{x_\omega}$, and should be thought of as a suitable replacement for the $L^1_t L^2_x$ norm. It is designed to capture the improved space-time estimates that we get in null coordinates, and will handle the case where we are very close to the cone, in which case the  $\dot{\mc{X}}^{-\frac{1}{2}, 1}_{ \pm}$ norm is not so effective. The definition is as follows.

Let $\kappa \in \mc{C}_\alpha$ be a cap on the sphere. We say that $F$ is a $NF^\pm(\kappa)$ \emph{atom} if there exists $\omega \not \in 2 \kappa$ such that
        $$ \|  \Pi_{\pm\omega} F \|_{L^1_{t_\omega} L^2_{x_\omega}} + \theta(\omega, \kappa)^{-1} \| \Pi_{\mp \omega} F \|_{L^1_{t_\omega} L^2_{x_\omega}} \les 1.$$
We then define the atomic Banach space $NF^\pm(\kappa)$ via\footnote{Note that if $F$ is a $NF^\pm(\kappa)$ atom then for every $\phi \in \s(\RR^{1+n})$,
    $$ |F(\phi)| \les \big( \|  \Pi_{\pm\omega} F \|_{L^1_{t_\omega} L^2_{x_\omega}} + \theta(\omega, \kappa)^{-1} \| \Pi_{\mp \omega} F \|_{L^1_{t_\omega} L^2_{x_\omega}}\big) \| \phi \|_{L^\infty_{t_\omega} L^2_{x_\omega}} \lesa \big\| ( 1 + |t| + |x|)^{n+1} \phi \|_{L^\infty_{t, x}} $$
and so (\ref{eqn - gen cond on atoms}) holds.}
(\ref{eqn - defn atomic banach space}) where we take $E$ to be the set of all $NF^\pm(\kappa)$ atoms, thus
    $$ NF^\pm(\kappa) = \Big\{ \sum_j c_j F_j \,\,  \Big|\, \, (c_j ) \in \ell^1, \,\, F_j \text{ is a $NF^\pm(\kappa)$ atom } \, \Big\}$$
with the obvious norm defined as in (\ref{eqn - defn atomic banach space norm}). We frequently make use of the immediate inequality
        $$ \| F \|_{NF^\pm(\kappa)} \les \inf_{\omega \not \in 2 \kappa} \Big(  \big\|  \Pi_{\pm\omega} F \big\|_{L^1_{t_\omega} L^2_{x_\omega}} + \theta(\omega, \kappa)^{-1} \big\| \Pi_{\mp \omega} F \big\|_{L^1_{t_\omega} L^2_{x_\omega}}\Big).$$\\

The second null frame space we define forms a replacement for the missing $L^2_t L^\infty_x$ Strichartz estimate and is based on $L^2_{t_\omega} L^\infty_{x_\omega}$ type norms. Similar to the $NF^\pm(\kappa)$ space we use an atomic definition.

Let $\kappa \in \mc{C}_\alpha$. We say $\psi$ is a $PW^\pm(\kappa)$ \emph{atom}, if there exists $\omega \in 2 \kappa$ such that
        $$ \| \Pi_{\pm \omega} \psi \|_{L^2_{t_\omega} L^\infty_{x_\omega}} + \alpha^{-1} \| \Pi_{\mp \omega} \psi \|_{L^2_{t_\omega} L^\infty_{x_\omega}} \les 1.$$
The atomic Banach space $PW^\pm(\kappa)$ is then defined to be made up of sums of $PW^\pm(\kappa)$ atoms as in (\ref{eqn - defn atomic banach space}) with the induced norm (\ref{eqn - defn atomic banach space norm}). Provided we have two sufficiently separated caps $\kappa$ and $\bar{\kappa}$, the null frame space $NF^\pm(\bar{\kappa})$ and the plane wave type space $PW^\pm(\kappa)$ have a simple relation via what is essentially an application of Holder's inequality.

\begin{lemma}\label{lem - PW NF L2 duality}
Let $0<\alpha, \beta \ll 1$. Assume $\kappa \in \mc{C}_\alpha$ and $\bar{\kappa} \in \mc{C}_\beta$ with $\theta(\kappa, \bar{\kappa}) \g 5 \max\{ \alpha, \beta\}$.  Let $\psi : \RR^{n+1} \rightarrow \CC^2$, and $F$ a scalar valued function. Then
            $$ \| F \psi \|_{NF^\pm(\bar{\kappa})} \lesa \| F \|_{L^2_{t, x}} \| \psi \|_{PW^\pm(\kappa)}.$$
More generally, for a fixed $\kappa \in \mc{C}_\alpha$, we have the orthogonality property
    \begin{equation}\label{eqn - lem PW NF L2 duality - orthog est} \bigg( \sum_{\substack{\bar{\kappa} \in \mc{C}_\beta \\ \theta(\kappa, \bar{\kappa})\g 5 \max\{\alpha, \beta\}  }} \Big\| P^{\pm,  \beta}_{\lambda, \bar{\kappa}} \Pi_+ \big( F v \big) \Big\|_{NF^\pm(\bar{\kappa})}^2 + \Big\| P^{\mp,  \beta}_{\lambda, \bar{\kappa}} \Pi_- \big( F v \big) \Big\|_{NF^\pm(\bar{\kappa})}^2\bigg)^\frac{1}{2} \lesa \| F \|_{L^2_{t, x}} \| v \|_{PW^\pm(\kappa)}.\end{equation}
  \begin{proof}
   We start by assuming $\psi$ is a $PW^\pm(\kappa)$ atom $\psi$, thus there exists $\omega \in  2\kappa$ such that
                $$ \| \Pi_{\pm\omega} \psi \|_{L^2_{t_\omega} L^\infty_{x_\omega}} + \alpha^{-1} \| \Pi_{\mp\omega} \psi \|_{L^2_{t_\omega}L^\infty_{x_\omega}} \les 1.$$
    The assumption $\theta(\kappa, \bar{\kappa}) \g 5\max\{\alpha,  \beta\}$ implies that $\theta(\omega, \bar{\kappa}) \g \theta(\kappa, \bar{\kappa}) - \theta(\omega, \kappa) \g 3 \max\{ \alpha, \beta\}$. In particular, $\omega \not \in 2 \bar{\kappa}$ and $\theta(\omega, \bar{\kappa})^{-1} \lesa  \alpha^{-1}$. Hence via Holder's inequality, we obtain
        \begin{align*}
          \| F \psi \|_{NF^\pm(\kappa)} &\les   \big\| F\Pi_{\pm \omega}  \psi \big\|_{L^1_{t_\omega} L^2_{x_\omega}} + \theta(\omega, \bar{\kappa})^{-1} \big\| F \Pi_{\mp \omega} \psi \big\|_{L^1_{t_\omega} L^2_{x_\omega}}\\
                    &\lesa \| F \|_{L^2_{t, x}} \Big(\| \Pi_{\pm\omega}\psi \|_{L^2_{t_\omega} L^\infty_{x_\omega}} + \alpha^{-1} \| \Pi_{\mp\omega} \psi \|_{L^2_{t_\omega} L^\infty_{x_\omega}}\Big)\\
                    &\les \| F \|_{L^2_{t, x}}.
        \end{align*}
   The argument for a general $\psi \in PW^\pm(\kappa)$ follows by decomposing $\psi = \sum_j c_j \psi_j$ where $\psi_j$ are atoms, and noting that $F \psi = \sum_j c_j F \psi_j$ in $\s'$.

   The proof of (\ref{eqn - lem PW NF L2 duality - orthog est}) is similar, but requires the additional complication of the orthogonality estimate
        $$\bigg( \sum_{\substack{\bar{\kappa} \in \mc{C}_\beta \\ \theta(\kappa, \bar{\kappa})\g 5 \max\{\alpha, \beta\}  }} \Big\| P^{\pm,  \beta}_{\lambda, \bar{\kappa}} \Pi_+ G \Big\|_{NF^\pm(\bar{\kappa})}^2 + \Big\| P^{\mp,  \beta}_{\lambda, \bar{\kappa}} \Pi_- G \Big\|_{NF^\pm(\bar{\kappa})}^2\bigg)^\frac{1}{2}\lesa \|\Pi_{\pm \omega} G \|_{L^1_{t_\omega} L^2_{x_\omega}} + \alpha^{-1} \| \Pi_{\mp \omega} G \|_{L^1_{t_\omega} L^2_{x_\omega}}$$
   which can be found in (iii) Corollary \ref{cor - orthog in null frame} below.
  \end{proof}
\end{lemma}

The final null frame space we require is a version of the energy type norm $L^\infty_t L^2_x$ in null frames. Given a cap $\kappa \subset \mc{C}_\alpha$, we define the norm $\| \cdot \|_{[NF^\pm]^*(\kappa)}$ as
    $$ \| u \|_{[NF^\pm]^*(\kappa)} = \sup_{\omega \not \in 2 \kappa} \big( \| \Pi_{\pm\omega} u \|_{L^\infty_{t_\omega}L^2_{x_\omega}} + \theta(\omega, \kappa) \| \Pi_{\mp\omega} u \|_{L^\infty_{t_\omega} L^2_{x_\omega}}\big). $$
It is easy enough to check that we have the duality relation
    \begin{equation}\label{eqn - NF, NF* duality} \Big| \int u^\dagger v dx dt \Big| \les \| u \|_{NF^\pm(\kappa)} \| v \|_{[NF^\pm]^*(\kappa)} \end{equation}
and consequently, by a duality argument, we have the following counterpart to Lemma \ref{lem - PW NF L2 duality}.

\begin{lemma}\label{lem - PW NF* duality}
  Let $0<\alpha, \beta \ll 1$. Assume $\kappa \in \mc{C}_\alpha$ and $\bar{\kappa} \in \mc{C}_\beta$ with $\theta( \kappa,  \bar{\kappa}) \g 5 \max\{\alpha, \beta\}$. Let $u, v$ take values in $\CC^2$. Then
        $$ \| u^\dagger v \|_{L^2_{t, x}} \lesa \| u \|_{[NF^\pm]^*(\bar{\kappa})} \| v \|_{PW^\pm(\kappa)} .$$
\end{lemma}

\begin{remark}
 In the original work of Tataru \cite{Tataru2001}, the null frame spaces were defined similarly but without the added complications of the projections $\Pi_{\pm\omega}$. The addition of the projections $\Pi_{\pm\omega}$ is needed to exploit the vector valued nature of the Dirac equation, and is motivated by the fact that if $\supp \widehat{f} \subset A_\lambda(\kappa)$, then from (\ref{eqn - hom - solution as average of free wave}) we can write the homogeneous solution $\mc{U}_-(t)f$ in the form
            $$ \mc{U}_-(t) f = \int_{\kappa} \Pi_\omega f_\omega( \sqrt{2} t_\omega) \, d \sph(\omega) .$$
 Thus the projections $\Pi_\omega$ appear naturally when we write the solution as an average of traveling waves. Furthermore, morally speaking, as $\Pi_\omega f_\omega( \sqrt{2} t_\omega)$ is a multiple of a $PW^\pm(\kappa)$ atom, we should have the bound
    $$  \| \rho(t) \mc{U}_-(t) f \|_{PW^+(\kappa)} \les \int_\kappa \| f_\omega \|_{L^2(\RR)} d\sph(\omega), \qquad \| \rho(t) \mc{U}_-(t) f \|_{PW^-(\kappa)} \les \alpha^{-1} \int_{\kappa} \| f_{\omega}\|_{L^2(\RR)} d \sph(\omega)$$
 where $\rho \in C^\infty_0(\RR)$ is a cutoff in time\footnote{It is unclear to the authors if this is true without the cutoff $\rho$.}  (see Corollary \ref{cor - null frame bounds hom case} below). In particular, as $\alpha \ll 1$, $\mc{U}_-(t) f$ obeys much better bounds in $PW^+(\kappa)$ than $PW^-(\kappa)$. Without the projections $\Pi_{\pm \omega}$ built into the spaces $PW^\pm(\kappa)$, this observation would be much harder to exploit. Finally, we note that the additional regularity given by placing $\mc{U}_-(t) f \in PW^+(\kappa)$, is a manifestation  of the \emph{null structure} of the Dirac equation, and plays a crucial role in the proof of Theorem \ref{thm - main thm with u, v}.
\end{remark}

%------------------------------------------------------------------------------------------------------------------------------%
\subsection{The space $\mc{N}^\pm_\lambda$.}
%------------------------------------------------------------------------------------------------------------------------------%

The space defined to hold the nonlinearity at scale $\lambda$, is made up of three components, a $L^1_t L^2_x$ component, an $X^{s, b}$ component, and a null frame $L^1_{t_\omega} L^2_{x_\omega}$ component. As previously, the definition is an atomic one, however, unlike the definition of $NF^\pm(\kappa)$ and $PW^\pm(\kappa)$, we require 3 different types of atoms.\\

\begin{enumerate}
\item We say $F$ is a $NF^\pm_\lambda$ \emph{atom} if there exists (a dyadic) $0<\alpha \ll 1$  and a decomposition  $F= \sum_{\kappa \in \mc{C}_\alpha} F_\kappa$ such that each $F_\kappa \in NF^\pm(\kappa)$ with
             \begin{equation}\label{eqn - supp prop of NF atoms} \supp \widehat{\Pi_+ F_\kappa} \subset A^\pm_{\alpha, \lambda}(\kappa), \qquad \supp \widehat{\Pi_- F_{\kappa}} \subset A^\mp_{\alpha, \lambda}(\kappa)\end{equation}
and we have the angular square function estimate
                        $$ \Big( \sum_{\kappa \in \mc{C}_\alpha}  \| F_\kappa \|_{NF^\pm(\kappa)}^2 \Big)^{\frac{1}{2}} \les 1.$$

\item We say that $F$ is a $P_\lambda(L^1_tL^2_x)$ atom, or \emph{energy atom}, if $\supp \widehat{F} \subset \{ |\xi| \approx \lambda\}$ and
                $$ \| F \|_{L^1_t L^2_x} \les 1.$$

\item We say that $F$ is a $\dot{\mc{X}}^{-\frac{1}{2}, 1}_\pm$ \emph{atom} if $$\supp \widetilde{\Pi_+ F} \subset \big\{ \,|\xi| \approx \lambda, \,\, \big| \tau \pm |\xi| \big| \approx d \, \big\}, \qquad  \supp \widetilde{\Pi_- F} \subset \big\{ \,|\xi| \approx \lambda, \,\, \big| \tau \mp |\xi| \big| \approx d \, \big\} $$
and
    $$ \|  F \|_{L^2_{t, x}} \les d^{\frac{1}{2}}.$$
\end{enumerate}
We now define
  $$ \mc{N}^\pm_\lambda = \Big\{ \sum_j c_j F_j \,\, \Big| \,\, (c_j) \in \ell^1(\NN), \,\,  F_j \text{ is either a $N^\pm_\lambda$ atom, an energy atom, or a $\dot{\mc{X}}^{-\frac{1}{2}, 1}_\pm$ atom } \,\, \Big\}$$
with the obvious norm given by (\ref{eqn - defn atomic banach space norm}). It is not so difficult to check that the condition (\ref{eqn - gen cond on atoms}) is satisfied, thus the space $\mc{N}^+_\lambda$ is a well-defined atomic Banach space. \\

In the proof of Theorem \ref{thm - main thm intro}, our aim will be to place the nonlinearity in $\mc{N}^\pm_\lambda$. Thus we shall frequently be aiming to estimate terms of the form $\| P_\lambda F \|_{\mc{N}^\pm_\lambda}$. To this end, we note that if $P_\lambda F \in L^1_t L^2_x$, then $P_\lambda F$ is multiple of an energy atom. Hence $P_\lambda F \in \mc{N}^\pm_\lambda$ and we have the immediate bound
            \begin{equation}\label{eqn - N controlled by L1L2} \| P_\lambda F \|_{\mc{N}^\pm_\lambda} \les \| P_\lambda F \|_{L^1_t L^2_x}.\end{equation}
Similarly, if we can write $P_\lambda F = \sum_{d \in 2^{\ZZ}} P_\lambda \mathfrak{C}^\pm_d F$, then as each $P_\lambda \mathfrak{C}^\pm_d F$ is a multiple of a $\dot{\mc{X}}^{-\frac{1}{2}, 1}_\pm$ atom, we have $P_\lambda F \in \mc{N}^\pm_\lambda$ and
            \begin{equation}\label{eqn - N controlled by X}\| P_\lambda F \|_{\mc{N}^\pm_\lambda} \les \| P_\lambda  F \|_{\dot{\mc{X}}^{-\frac{1}{2}, 1}_\pm}.\end{equation}
The general strategy to put $F \in \mc{N}^\pm_\lambda$ will be to decompose $F$ into certain frequency regions, and then make use of the previous bounds. Of course we will be unable to always place the nonlinearity in as nice a space as $L^1_t L^2_x$ (or $\dot{\mc{X}}^{-\frac{1}{2}, 1}_\pm$) and in certain frequency regions (notable when everything is close to the cone) we have to use the additional flexibility given by the $NF^\pm_\lambda$ type atoms.\\

\begin{remark} Let $F$ be a $NF^\pm_\lambda$ atom, and let $F = \sum_{\kappa \in \mc{C}_\alpha} F_\kappa$ be the corresponding decomposition into atoms. When we come to prove estimates for the $F_\kappa$, to use the fact that $F_\kappa \in NF^\pm(\kappa)$, we will be forced to decompose $F_\kappa = \sum_j F^{(j)}_\kappa$ into $NF^\pm(\kappa)$ atoms $F^{(j)}_\kappa$. Unfortunately this means that we may lose the support properties (\ref{eqn - supp prop of NF atoms}), as there is no guarantee that the $F^{(j)}_\kappa$ retain the same Fourier support as $F_\kappa$.  However, as we can write
            $$ \Pi_+ F_\kappa = {^{\natural}P^{\pm, \alpha}_{\lambda, \kappa}} \Pi_+ F_{\kappa} = \sum_j c_j {^{\natural}P^{\pm, \alpha}_{\lambda, \kappa}} \Pi_+ F^{(j)}_\kappa$$
  then as we have the bound
    \begin{align*}\big\| {^{\natural}P^{\pm, \alpha}_{\lambda, \kappa}} \Pi_{\pm \omega} \Pi_+ F^{(j)}_\kappa \big\|_{L^1_{t_\omega} L^2_{x_\omega}}  + \theta(\omega, \kappa)^{-1} \big\| &{^{\natural}P^{\pm, \alpha}_{\lambda, \kappa}} \Pi_{\mp \omega} \Pi_+ F^{(j)}_\kappa \big\|_{L^1_{t_\omega} L^2_{x_\omega}}\\
        &\lesa \big\|\Pi_{\pm \omega} F^{(j)}_\kappa \big\|_{L^1_{t_\omega} L^2_{x_\omega}}+ \theta(\omega, \kappa)^{-1} \big\| \Pi_{\mp \omega} F^{(j)}_\kappa \big\|_{L^1_{t_\omega} L^2_{x_\omega}}\end{align*}
(see Lemma \ref{lem - mult are disposable} below) the function ${^{\natural}P^{\pm, \alpha}_{\lambda, \kappa}} \Pi_+ F^{(j)}_\kappa$ is again a, perhaps slightly larger, $NF^\pm(\kappa)$ atom. Thus we may always assume that the functions $F^{(j)}_\kappa$ satisfy the slightly larger support properties
         $$ \supp \widehat{\Pi_+ F_\kappa}^{(j)} \subset {^{\natural}A}^\pm_{\alpha, \lambda}(\kappa), \qquad \supp \widehat{\Pi_- F_{\kappa}}^{(j)} \subset {^{\natural}A}^\mp_{\alpha, \lambda}(\kappa).$$
  This observation is frequently used without mention in the remainder of the article.
\end{remark}

When we come to prove estimates using the $\mc{N}^+_\lambda$ spaces, we often have to estimate a $NF^\pm_\lambda$ atom in $L^2_{t, x}$. The following lemma is very useful in this regard.

\begin{lemma}\label{lem - L2 bound on NF atom}
  Let $0<\alpha \ll 1$ and assume $F= \sum_{\kappa \in \mc{C}_\alpha} F_\kappa$ is a $NF^\pm_\lambda$ atom. Then
        $$ \|\mathfrak{C}^\pm_d F \|_{L^2_{t,x}} \lesa \big( \min\{ d, \alpha^2 \lambda \} \big)^\frac{1}{2}.$$
  \begin{proof}
   We only prove the $\pm=+$ case, the $-$ case is similar. Let $F= \sum_{\kappa \in \mc{C}_\alpha} F_\kappa$. By orthogonality in $L^2_{t, x}$, and the observation that $\mathfrak{C}^+_d F_\kappa = 0$ for $d > \alpha^2 \lambda$, it is enough to show that
            $$ \| \mathfrak{C}^+_d F_\kappa \|_{L^2_{t, x}} \lesa d^\frac{1}{2} \| F_{\kappa} \|_{NF^+(\kappa)}$$
   for $d \lesa \alpha^2 \lambda$.  Furthermore, by decomposing $F_\kappa$ into $NF^+(\kappa)$ atoms, we reduce to proving that for $\omega \not \in 2 \kappa$ we have
       \begin{equation}\label{eqn - lem L2 bound on NF atom - L2 bounded by L1L2 in null coord}\| C^\pm_{ \lesa d} {^\natural P^{\pm, \alpha}_{\lambda, \kappa}} \Pi_\pm G \|_{L^2_{t, x}} \lesa d^\frac{1}{2} \Big( \| \Pi_\omega G \|_{L^1_{t_\omega} L^2_{x_\omega}} + \theta(\omega, \kappa)^{-1} \| \Pi_{-\omega} G \|_{L^1_{t_\omega} L^2_{x_\omega}} \Big).\end{equation}
Note that if $(\tau, \xi) \in {^\natural A^\pm_{\lambda, \alpha}}(\kappa)$ and  $|\tau \pm |\xi| | \lesa d$, then from (\ref{eqn - estimate on dual coordinates on the sets A with omega restriction}), we have $|\xi^1_\omega| \approx \lambda \theta(\omega, \kappa)^2$ and hence
    $$ \Big| \tau_\omega - \frac{|\xi^\bot_\omega|}{2 \xi^1_\omega} \Big| = \frac{ \big| |\tau|^2- |\xi|^2 \big|}{2|\xi^1_\omega|} \lesa \frac{ d}{\theta(\omega, \kappa)^2}.$$
Thus, for fixed $\xi_\omega$,  $\tau_\omega$ varies in a set of size $\frac{d }{\theta(\omega, \kappa)^2}$. Therefore, by an application of Bernstein together with the null form estimate $|\Pi_{\pm \frac{\xi}{|\xi|}} \Pi_\omega| \lesa \theta(\omega, \mp \xi) \approx \theta(\omega, \kappa)$, we have
    \begin{align*}
      \| C^\pm_{ d} {^\natural P^{\pm, \alpha}_{\lambda, \kappa}} \Pi_\pm G \|_{L^2_{t, x}} &\lesa \| C^\pm_{ d} {^\natural P^{\pm, \alpha}_{\lambda, \kappa}} \Pi_\pm \Pi_\omega G \|_{L^2_{t, x}} + \| C^\pm_{d} {^\natural P^{\pm, \alpha}_{\lambda, \kappa}} \Pi_\pm \Pi_{-\omega} G \|_{L^2_{t, x}} \\
      &\lesa \theta(\omega, \kappa) \times \frac{ d^{\frac{1}{2}}}{\theta(\omega, \kappa)} \big\| \widetilde{\Pi_\omega G} \big\|_{L^2_{\xi_\omega} L^\infty_{\tau_\omega}} + \frac{ d^{\frac{1}{2}}}{\theta(\omega, \kappa)} \big\| \widetilde{\Pi_{-\omega} G} \big\|_{L^2_{\xi_\omega} L^\infty_{\tau_\omega}} \\
      &\lesa d^\frac{1}{2} \Big( \| \Pi_\omega G \|_{L^1_{t_\omega} L^2_{x_\omega}} + \theta(\omega, \kappa)^{-1} \| \Pi_{-\omega} G \|_{L^1_{t_\omega} L^2_{x_\omega}} \Big)
    \end{align*}
as required.
  \end{proof}
\end{lemma}

%-------------------------------------------------------------------------------------%
\subsection{Iteration Space}
%-------------------------------------------------------------------------------------%

We now have the basic building blocks of the Banach space with which to prove Theorem \ref{thm - main thm intro}. Define
        $$ F^\pm_\lambda = \big\{ \,\, u \in L^\infty_t L^2_x\,\, \big| \,\, \supp \widehat{u} \subset \{ |\xi| \approx \lambda\}, \,\, (\p_t \pm \sigma \cdot \nabla) u \in \mc{N}^\pm_\lambda \, \big\}$$
with the associated norm
        $$ \| u \|_{F^\pm_\lambda} = \| u \|_{L^\infty_t L^2_x} + \big\| ( \p_t \pm \sigma \cdot \nabla) u \big\|_{\mc{N}^\pm_\lambda}.$$
We now sum up over frequencies to define
        $$ \| u \|_{F^{s, \pm}} = \Big( \sum_{\lambda \in 2^\ZZ} \lambda^{2s} \| P_\lambda u \|_{F_\lambda^\pm}^2 \Big)^\frac{1}{2}$$
and let
        $$ F^{s, \pm} = \big\{ u \in L^\infty_t \dot{H}^s_x \,\, \big| \,\, P_\lambda u \in F^\pm_\lambda, \,\, \| u \|_{F^{s, \pm}} < \infty\,\, \big\}.$$
The space $F^\pm_\lambda$ is essentially enough to prove the multi-linear estimates that we require, and via bilinear estimates of the form in Lemma \ref{lem - hom bilinear est}, it is possible to complete the proof of Theorem \ref{thm - main thm intro} with small data in the slightly smaller space $\dot{B}^\frac{n-1}{2}_{2, 1}$ (i.e. with an $\ell^1$ sum over frequencies instead of a $\ell^2$ sum). To get the more general $\dot{H}^\frac{n-1}{2}$ result, we need some additional gain away from the light cone. To this end, motivated by the recent work\footnote{In the work of Bejenaru-Herr, they also needed some additional integrability in time of functions supported away from the light cone. To accomplish this, they made use of the norm (in the notation used in the current paper)
    $$ \sup_{d} d \big\| \mathfrak{C}^\pm_d u \big\|_{L^\frac{4}{3}_t L^2_x(\RR^{1 + 3})}.$$
In the current paper, this norm is to strong, and we need to use the slightly weaker $\mc{Y}^\pm_\lambda$ norm (note that $\frac{4n}{3n-1} = \frac{3}{2}$ if $n=3$, thus we need \emph{less} integrability in time). } of Bejenaru-Herr \cite{Bejenaru2013}, we define an additional semi-norm  $\| \cdot \|_{\mc{Y}^\pm}$ as
 $$ \| u \|_{\mc{Y}^\pm} =  \sup_d   d \big\| \mathfrak{C}^\pm_d u \big\|_{L^\frac{4n}{3n-1}_t L^2_x}.$$
Then we take $G^\pm_\lambda$ as
    $$ G^\pm_\lambda = \big\{\,\,  u \in F^\pm_\lambda \,\, \big| \,\, \big\| u \big\|_{\mc{Y}^\pm}< \infty \,\big\}$$
with the norm
    $$ \| u \|_{G^\pm_\lambda} = \| u \|_{F^\pm_\lambda} + \lambda^{ -  \frac{n +1}{4n}} \|  u \|_{\mc{Y}^\pm}$$
where the $\lambda^{ - \frac{n+1}{4n}}$ term is to ensure that both components of the $G^\pm_\lambda$ norm scale the same way.
We now define
        $$ G^{s, \pm} = \big\{\,\, u \in L^\infty_t \dot{H}^s_x \,\, \big| \,\, P_\lambda u \in G^\pm_\lambda, \,\, \| u \|_{G^{s, \pm}}< \infty \,\, \}$$
where
        $$ \| u \|_{G^{s, \pm}} = \Big( \sum_{\lambda \in 2^\ZZ} \lambda^{2s} \big\| P_\lambda u \big\|_{G^\pm_\lambda}^2 \Big)^\frac{1}{2}.$$
Corresponding to the function spaces $F^{s, \pm}$ and $G^{s, \pm}$, we aim to put the nonlinearity in the summed up versions of the $\mc{N}^\pm_\lambda$ and $L^\frac{4n}{3n-1}_t L^2_x$ spaces. Namely, we define
    $$ \| F \|_{\mc{N}^{s, \pm} } = \bigg( \sum_{\lambda \in 2^\ZZ} \lambda^{2s} \big\| P_\lambda F \big\|_{\mc{N}^\pm_\lambda}^2 \bigg)^\frac{1}{2}$$
and
    $$ \| F \|_{(\mc{N}\cap\mc{Y})^{s, \pm} } =   \bigg( \sum_{\lambda \in 2^\ZZ} \lambda^{2s} \big\| P_\lambda F \big\|_{\mc{N}^\pm_\lambda}^2  + \lambda^{2(s - \frac{n+1}{4n})} \big\| P_\lambda F \big\|_{L^\frac{4n}{3n-1}_t L^2_x}^2 \bigg)^\frac{1}{2}.$$
These spaces satisfy the following important properties.

\begin{theorem}\label{thm - energy inequality + invariance under cutoffs}
\leavevmode
\begin{enumerate}
    \item \emph{(Energy inequality.)} Let $s \g 0$. Then $F^{s, \pm}$ is a Banach space, and moreover we have the energy inequality
        $$ \| u \|_{F^{s, \pm}} \les \| u(0) \|_{\dot{H}^s} + C\big\| (\p_t \pm \sigma \cdot \nabla) u \big\|_{\mc{N}^{s, \pm}}.$$
    Similarly we have
        $$ \| u \|_{G^{s, \pm}} \les \| u(0) \|_{\dot{H}^s} + C\big\| (\p_t \pm \sigma \cdot \nabla) u \big\|_{(\mc{N}\cap\mc{Y})^{s, \pm}}$$
    (here $C$ is some constant independent of $u$).\\

    \item \emph{(Stability with respect to time cutoffs.)} Let $\rho \in C^\infty_0(\RR)$ and $T>0$. Then
            $$ \big\| \rho(\tfrac{t}{T}) u \big\|_{F^{s, \pm}} \lesa \| u \|_{F^{s,
            \pm}}, \qquad \qquad \big\| \ind_{(-T, T)}(t) F \big\|_{\mc{N}^{s,
            \pm}} \lesa \| F \|_{\mc{N}^{s, \pm}}$$
    where the implied constants are independent of $T$. Similarly, if $\lambda \in 2^\ZZ$ and $T \g \lambda^{-1}$, we have the bound
            $$\big\| \rho(\tfrac{t}{T}) u \big\|_{G^\pm_\lambda}\lesa \| u \|_{G^\pm_\lambda}$$
    where the implied constant is again independent of $T$.\\

    \item \emph{(Scattering.)} Let $\rho \in C^\infty_0(\RR)$ with $\rho(t) = 1$ on $[-1, 1]$ and assume $\sup_{T>0} \| \rho( \frac{t}{T}) u \|_{F^{s, \pm}}< \infty$. Then there exists $f_{-\infty}, f_{+\infty} \in \dot{H}^{s}$ such that
            $$ \lim_{t \rightarrow \infty} \Big( \big\| u(t) - \mc{U}_\pm(t) f_{+\infty} \big\|_{\dot{H}^{s}} + \big\| u(-t) - \mc{U}_\pm(-t) f_{-\infty} \big\|_{\dot{H}^{s}}\Big) = 0.$$
\end{enumerate}
\begin{proof}
We leave the proof to Section \ref{sec - energy inequality}.
\end{proof}
\end{theorem}

\begin{remark}
  Note that the previous theorem implies that we have the bound
        $$ \| \phi \|_{F^{s, \pm}} \les \| \phi(0) \|_{\dot{H}^s} + \| ( \p_t \pm \sigma \cdot \nabla) \phi \|_{L^1_t \dot{H}^s_x}.$$
  In particular, $\| \phi \|_{F^{s, \pm}} < \infty$ for every $\phi \in \s$. A similar comment applies in the $G^{s, \pm}$ case.
\end{remark}

\begin{remark}\label{rem - + ---> - via reflection}
Our eventual aim will be to construct a solution in $G^{\frac{n-1}{2}, \pm}$, although this will require a significant amount of work. To alleviate this somewhat, we note that if $u \in F^\pm_\lambda$, then letting $v(t, x) = u(t, -x)$, a computation shows that\footnote{Essentially this boils down to showing that reflecting a $\mc{N}^\pm_\lambda$ atom in $x$, gives a $\mc{N}^\mp_\lambda$ atom, which is not to difficult to show. }  $v \in F^\mp_\lambda$. Similarly we can check that if $u \in G^\pm_\lambda$ then $v \in G^\mp_\lambda$. On the other hand, if we reflect in both $t$ \emph{and} $x$, i.e. we let $w(t, x) = u(-t, -x)$, then a similar calculation shows that $\| u \|_{G^\pm_\lambda} = \| w \|_{G^\pm_\lambda}$ and $\| u \|_{F^\pm_\lambda} = \| w \|_{F^\pm_\lambda}$ while $\big(\Pi_\pm u\big)(t, x) = \big(\Pi_\mp w \big)(-t, -x)$. Together these observations often allow us to reduce to considering just the $+$ case, rather than both $+$ and $-$ cases.

In a similar vein, we observe that in the $n=2$ case a computation using (\ref{eqn - beta flips sign of linear eqn}) shows that we have $\| \beta u \|_{G^{s, \pm}} \approx \| u \|_{G^{s, \mp}}$. As in the homogeneous case, this will allow us to deduce estimates for $u^\dagger \beta u$ from estimates of the form $u^\dagger v$.
\end{remark}

The norm $F^\pm_\lambda$ is fairly complicated due to its atomic structure. However it can be compared to the more standard $\dot{\mc{X}}^{\frac{1}{2}, q}_{ \pm}$ spaces by the following useful estimate.

\begin{lemma}\label{lem - F controls Xsb infty}
Let $u \in F_\lambda^\pm$. Then
        \begin{equation}\label{eqn - lem F controls Xsb infty} \| u \|_{\dot{\mc{X}}^{\frac{1}{2},\infty}_{\pm}} \lesa \| u \|_{F^\pm_\lambda}. \end{equation}
\begin{proof}
 By a reflection, we may assume that $\pm = +$. The estimate
            $$ \| u \|_{\dot{\mc{X}}^{\frac{1}{2},\infty}_{+}} \approx \| (\p_t + \sigma \cdot \nabla) u \|_{\dot{\mc{X}}^{-\frac{1}{2},\infty}_{+}}$$
 shows that is enough to prove that $ \| F \|_{\dot{\mc{X}}^{-\frac{1}{2},\infty}_{+}} \lesa \| F \|_{\mc{N}^+_\lambda}$. The atomic definition of $\mc{N}^+_\lambda$, implies that we need to consider three cases, $F$ is a $\dot{\mc{X}}^{-\frac{1}{2},\infty}_{+}$ atom, $F$ is an energy atom, and $F$ is a $NF^+_\lambda$ atom. The first case is obvious due to the embedding $\dot{\mc{X}}^{-\frac{1}{2},1}_{+} \subset \dot{\mc{X}}^{-\frac{1}{2},\infty}_{+}$. On the other hand, if $F \in L^1_t L^2_x$, then as $\supp \widetilde{ \mathfrak{C}^+_d F} \subset \big\{ \big||\tau| - |\xi| \big| \approx d\big\}$, we see that for each fixed $d$
    \begin{align*}
      d^{-\frac{1}{2}} \| \mathfrak{C}^+_d F \|_{L^2_{t, x}} \lesa  \big\| \widetilde{F} \big\|_{L^2_{\xi} L^\infty_{\tau}} \lesa \| F \|_{L^1_t L^2_x}.
    \end{align*}
 Taking the sup over $d$ then gives the $F \in L^1_t L^2_x$ case. Finally, if $F$ is a $NF^+_\lambda$ atom, then by Lemma \ref{lem - L2 bound on NF atom} we obtain $\| \mathfrak{C}^+_d F \|_{L^2_{t, x}}\lesa d^{\frac{1}{2}} $ and so we clearly have $\| F \|_{\mc{X}^{-\frac{1}{2}, \infty}_{ +}} \lesa 1$ as required.

\end{proof}
\end{lemma}

\begin{remark}\label{rem - Xsb controls F_lambda}
Note that, if $\supp u \subset \{ |\xi| \approx \lambda\}$ and $u \in L^2_{t, x}(\RR^{1+n})$, then by decomposing $u = \sum_{d \in 2^{\ZZ}} \mathfrak{C}^\pm_d u$ (which is possible as $u \in L^2_{t, x}$), by definition of $\mc{N}^\pm_\lambda$ together with (\ref{eqn - Xsb controls strichartz}) we have
   \begin{equation}\label{eqn - Xsb controls F lambda} \| u \|_{F^\pm_\lambda} \les  \| u \|_{L^\infty_t L^2_x} + \big\| (\p_t \pm \sigma \cdot \nabla) u \big\|_{\mc{N}^\pm_\lambda} \lesa \| u \|_{\dot{\mc{X}}^{\frac{1}{2}, 1}_\pm} + \sum_{d \in 2^{\ZZ}} d^{- \frac{1}{2}} \big\| (\p_t \pm \sigma \cdot \nabla) \mathfrak{C}^\pm_d u \big\|_{L^2_{t, x}} \lesa \| u \|_{\dot{\mc{X}}^{\frac{1}{2}, 1}_\pm}. \end{equation}
In particular, we have the bounds
    $$ \| u \|_{\dot{\mc{X}}^{\frac{1}{2}, \infty}_\pm} \lesa \| u \|_{F^\pm_\lambda} \lesa \| u \|_{\dot{\mc{X}}^{\frac{1}{2}, 1}_\pm},$$
thus $F^\pm_\lambda$ is within a log factor of an $X^{s, b}$ spaces.

\end{remark}

As mentioned previously, if we had access to a $L^2_t L^\infty_x$ Strichartz estimate, then the proof of GWP would follow by an application of H\"older's inequality. However, the $L^2_t L^\infty_x$ Strichartz estimate barely fails in $n=3$, and is far from true in the $n=2$ case. Despite this, provided we are away from the cone, we \emph{can} control the $L^2_t L^\infty_x$ by a simple application of Bernstein together with the previous lemma. More precisely, if $u \in F_\lambda^\pm$, then by Lemma \ref{lem - F controls Xsb infty}
    \begin{equation}\label{eqn - L2 control away from null cone}
        \| \mathfrak{C}^\pm_{\gtrsim \delta} u \|_{L^2_{t, x}} \lesa \sum_{d \gtrsim \delta} \| \mathfrak{C}^\pm_{d} u \|_{L^2_{t, x}} \lesa  \| u \|_{\dot{\mc{X}}^{\frac{1}{2}, \infty}_{ \pm} } \sum_{ d \gtrsim \delta} d^{-\frac{1}{2}} \lesa  \delta^{-\frac{1}{2}} \| u \|_{F^\pm_\lambda}
    \end{equation}
and consequently
    $$ \|\mathfrak{C}^\pm_{\gtrsim \delta} u \|_{L^2_t L^\infty_x} \lesa \lambda^{\frac{n}{2}} \| \mathfrak{C}^\pm_{\gtrsim \delta}u \|_{L^2_{t, x}} \lesa \lambda^{\frac{n}{2}} \delta^{-\frac{1}{2}} \| u \|_{F^\pm_\lambda}.$$
This estimate, as well as the important $L^2_{t, x}$ bound (\ref{eqn - L2 control away from null cone}), is used frequently in the remainder of this article as it essentially allows us to deal with the the region away from the light cone\footnote{This is true in the bilinear case. In the proof of the trilinear estimates, Lemma \ref{lem - F controls Xsb infty} is not enough to deal with the far cone regions and we require the addition decay in time provided by the $\mc{Y}^\pm$ norms.}. The remaining close cone interaction is much more complicated, and requires the the full strength of the norms defined above.

%------------------------------------------------------------------------------------------------------------------------------%
\subsection{Disposable Multipliers}
%------------------------------------------------------------------------------------------------------------------------------%

We use some notation originally due to Tao \cite{Tao2001a}. We say a Fourier multiplier  $\mc{M}$ is \emph{disposable} on a Banach space $X$, if we have
        $$ \| \mc{M} F \|_{X } \lesa \| F \|_{X}.$$
Clearly any Fourier multiplier with bounded symbol is disposable on $L^2_{t, x}$ by Plancheral. More generally we have the following.

\begin{lemma}[Multipliers are disposable]\label{lem - mult are disposable}
 Let $\alpha, \beta \ll 1$ and $\kappa \in \mc{C}_\alpha$, $\bar{\kappa} \in \mc{C}_\beta$.

\begin{enumerate}
    \item Let $\pm_1$ and $\pm_2$ be independent choices of signs. Then $P^{\pm_1, \alpha}_{\lambda, \kappa}$  is given by a convolution with an $L^1_{t, x}(\RR^{1+n})$ kernel. In particular, $P^{\pm_1, \alpha}_{\lambda, \kappa}$ is disposable on $NF^{\pm_2}(\bar{\kappa})$, $PW^{\pm_2}(\bar{\kappa})$, and $[NF^{\pm_2}]^*(\bar{\kappa})$.\\

    \item Assume $\alpha \lesa \beta$ and $\kappa \cap \bar{\kappa} \not = \varnothing$.  Then $P^{\pm, \alpha}_{\lambda, \kappa}\Pi_{\pm}$ is disposable on $NF^+(\bar{\kappa})$, $PW^+(\bar{\kappa})$, and  $[NF^+]^*(\bar{\kappa})$.   Similarly $P^{\mp, \alpha}_{\lambda, \kappa}\Pi_{\pm}$ is disposable on $NF^-(\bar{\kappa})$, $PW^-(\bar{\kappa})$, and $[NF^-]^*(\bar{\kappa})$.\\

    \item The multipliers $C_d$, $C_{\lesa d}$, $C^\pm_d$, $C^\pm_{\lesa d}$ are disposable on $L^q_t L^2_x$ for $1\les q \les \infty$.\\

    \item Let $d \gtrsim \lambda$. Then $P_\lambda C_d$, $P_\lambda C_{\lesa d}$, and $P_\lambda C_{\gtrsim d}$ are disposable on $L^q_t L^r_x$ for any $1\les q, r \les \infty$.
\end{enumerate}
\begin{proof}
\textbf{(i) and (ii):}  We only show that $P^{\pm, \alpha}_{\lambda, \kappa} \Pi_\pm$ is disposable as the remaining case is similar (but easier). So assume that $\alpha \lesa \beta$ and $\kappa \cap \bar{\kappa} \not = \varnothing$. The general idea is to show that the kernel of $P^{\pm, \alpha}_{\lambda, \kappa}\Pi_{\pm}$ belongs to $L^1_{t, x}$, and then apply Holder. There is a slight complication however, as the definition of the null frame spaces use the projections $\Pi_{\pm \omega}$ which do not commute with the $\Pi_{\pm}$. Thus showing that the kernel is in $L^1_{t, x}$ would not suffice and we need to prove a stronger estimate exploiting the null form estimate (\ref{eqn - null structure estimate}).

Let $ \widetilde{\rho}(\tau, \xi) = \Phi( \tfrac{|\xi|}{\lambda}) \Phi_\kappa\big(\mp \tfrac{\xi}{|\xi|}\big) \Phi_0( \tfrac{|\tau \pm |\xi||}{c \alpha^2 \lambda})$ where $c$ is the small constant used in the definition of $A^\pm_{\lambda, \alpha}(\kappa)$, thus $P^{\pm, \alpha}_{\lambda, \kappa} u = \rho* u$. Fix any $\omega \in \sph^{n-1}$. The key is to prove that $\| \Pi_\pm \rho\|_{L^1_{t, x}} \lesa 1$ as well as the stronger estimate
        \begin{equation}\label{eqn - lem mult dispos - key L1 bound for (i)}
                \big\| \big(\Pi_\pm - \Pi_{-\omega}\big) \rho \big\|_{L^1_{t, x}}\lesa \max\{ \theta(\omega, \bar{\kappa}), \beta\}.
            \end{equation}
Since assuming we have (\ref{eqn - lem mult dispos - key L1 bound for (i)}) and using the identity $\Pi_\pm \Pi_\omega = ( \Pi_\pm - \Pi_{-\omega} ) \Pi_\omega$ we deduce that
               $$ \big\| P^{\pm, \alpha}_{\lambda, \kappa} \Pi_{\pm} \Pi_\omega u \big\|_{L^q_{t_\omega} L^r_{x_\omega}} =\big\|  \big[ \big( \Pi_\pm - \Pi_{-\omega}\big) \rho\big] * \big(\Pi_\omega u\big) \big\|_{L^q_{t_\omega} L^r_{x_\omega}} \lesa \max\{ \theta(\omega, \bar{\kappa}), \beta\} \| \Pi_\omega u \|_{L^q_{t_\omega} L^r_{x_\omega}}. $$
Similarly the $L^1_{t, x}$ bound gives
            $$ \big\| P^{\pm, \alpha}_{\lambda, \kappa} \Pi_{\pm} \Pi_{-\omega} u \big\|_{L^q_{t_\omega} L^r_{x_\omega}} = \big\| (\Pi_\pm\rho) * \big(\Pi_{-\omega} u\big) \big\|_{L^q_{t_\omega} L^r_{x_\omega}}  \lesa  \| \Pi_{-\omega} u \|_{L^q_{t_\omega} L^r_{x_\omega}}.$$
Applying these bounds to the relevant atoms, we obtain the boundedness of $P^{\pm, \alpha}_{\lambda, \kappa} \Pi_\pm$ on $NF^+(\bar{\kappa})$, $PW^+(\bar{\kappa})$, and $[NF^+]^*(\bar{\kappa})$.

 We now prove (\ref{eqn - lem mult dispos - key L1 bound for (i)}). Let $x=(x_1, x') \in \RR\times \RR^{n-1}$ and $\xi=(\xi_1, \xi') \in \RR \times \RR^{n-1}$. By rotating the $\xi$ coordinates (and a reflection if needed) we may assume that $\pm = +$ and  $\kappa$ is centered around $(-1, 0, ..., 0)$, thus $\widetilde{\rho}$ is supported in the set $\{ \xi_1 \sim \lambda,\,\,\, |\xi'| \lesa \lambda \alpha\}$. A computation gives
      \begin{align*}  2\big(\Pi_{+} - \Pi_{-\omega} \big) &\rho\big( t, x_1 + t, x'\big)\\
                &= \frac{1}{(2\pi)^{n+1}}\int_{\RR^n}\int_\RR \Big( \omega + \frac{\xi}{|\xi|} \Big)\cdot \sigma\widetilde{\rho}(\tau, \xi) e^{ i x \cdot \xi} e^{ i t (\tau + \xi_1)} d\tau d\xi\\
                &= \frac{1}{(2\pi)^{n+1}}\int_{\RR^n} \int_\RR \Big( \omega + \frac{\xi}{|\xi|} \Big)\cdot \sigma\widetilde{\rho}\big(\tau - \xi_1, \xi\big) e^{ i x \cdot \xi} e^{ i t \tau} d\tau d\xi\\
                &= \frac{(\alpha \lambda)^{n+1}}{(2\pi)^{n+1}} \int_{\RR^n} \int_\RR \Big( \omega + \frac{(\xi_1, \alpha \xi')}{|(\xi_1, \alpha\xi')|} \Big)\cdot \sigma\widetilde{\rho}\big(\alpha^2 \lambda \tau - \lambda \xi_1, \lambda \xi_1, \alpha \lambda \xi'\big) e^{ i \lambda(\alpha^2 t,  x_1, \alpha x') \cdot(\tau,  \xi) }  d\tau d\xi.
      \end{align*}
 If we now rescale the $(t,x)$ variables (which leaves the $L^1_{t, x}$ norm unchanged), it is enough to prove that
        $$ \p_\tau^{N_1} \p_{\xi_1}^{N_2} \nabla^{N_3}_{\xi'} \left( \Big( \omega + \frac{(\xi_1, \alpha \xi')}{|(\xi_1, \alpha\xi')|} \Big)\widetilde{\rho}\big(\alpha^2 \lambda \tau - \lambda \xi_1, \lambda \xi_1, \alpha \lambda \xi'\big) \right) \lesa \max\{ \theta(\omega, \bar{\kappa}), \beta \}$$
where $(\tau, \xi_1, \xi') \in \{ |\tau| \lesa 1, \,\, \xi_1 \approx 1, \, \, |\xi'| \lesa 1 \}$. If we note that $ \tau - \xi_1 + |\xi| = \tau  + \frac{|\xi'|^2}{\xi_1 + |\xi|}$ we can write
    $$ \widetilde{\rho}\big(\alpha^2 \lambda \tau - \lambda \xi_1, \lambda \xi_1, \alpha \lambda \xi'\big) = \Phi\big( |(\xi_1, \alpha \xi')|\big) \Phi_\kappa\Big( - \frac{( \xi, \alpha \xi')}{|(\xi, \alpha \xi')|}\Big) \Phi_0\Big( c^{-1} \tau + \frac{c^{-1} |\xi'|^2}{ \xi_1 + |(\xi_1, \alpha \xi')|}\Big)$$
and thus whenever a derivative hits $\widetilde{\rho}$, by (\ref{eqn - derivative bounds for multipliers}), we at worst pick up a factor of $\alpha \lesa \beta \ll 1$. Thus it remains to show that
            $$ \left| \p^{N_2}_{\xi_1} \nabla_{\xi'}^{N_3} \left( \omega + \frac{(\xi_1, \alpha \xi')}{|(\xi_1, \alpha  \xi')|}\right) \right| \lesa \max\{\theta(\omega, \bar{\kappa}), \beta\}.$$
Suppose $N_2 = N_3 =0$. Let $\eta = (\xi_1, \alpha \xi')$ and $\xi^* \in \kappa \cap \bar{\kappa} (\not =\varnothing)$. Then $- \frac{\eta}{|\eta|} \in \kappa$ and moreover
    $$ \left| \omega + \frac{(\xi_1, \alpha \xi')}{|(\xi_1, \alpha  \xi')|}\right| \lesa \theta(\omega, - \eta) \lesa \theta(\omega, \xi^*) + \theta(- \eta, \xi^*) \lesa \max\{ \theta(\omega, \bar{\kappa}), \beta\}$$
since $\bar{\kappa} \in \mc{C}_\beta$. On the other hand for $N_1, N_2 \not = 0$, we simple note that derivatives of $\xi'$ only add multiples of $\alpha$ (which is acceptable as $\alpha \lesa \beta$), while
            $$ \left| \p_{\xi_1} \left( \omega - \frac{(\xi_1, \alpha \xi')}{|(\xi_1, \alpha  \xi')|}\right)\right|  = \left| \frac{ (\alpha^2 |\xi'|^2, - \alpha \xi_1 |\xi'|)}{|(\xi_1, \alpha \xi')|^3}\right| \lesa \alpha$$
which again is clearly acceptable. Thus we obtain (\ref{eqn - lem mult dispos - key L1 bound for (i)}), and clearly the same argument shows that $\| \Pi_+ \rho \|_{L^1} \lesa 1$ as required.

\textbf{(ii) and (iii):} These are both well known, see for instance \cite[Lemma 3]{Tao2001a}.

\end{proof}
\end{lemma}

\begin{remark}\label{rem - C_d mult disposable}
  It is clear from the proof that multipliers of the form ${^\natural P}_\lambda R_\kappa^\pm C^\pm_{\ll d}$ (and other similar combinations) also satisfy the properties  $(i)$ and $(ii)$ in the previous lemma provided $d \gtrsim \alpha^2 \lambda$. In particular, if $\supp \widehat{u} \subset \{ |\xi| \approx \lambda\}$, $\kappa \in \mc{C}_\alpha$, and $d \gtrsim \alpha^2 \lambda$, then we can write $C^\pm_{\les d} R^\pm_\kappa u = \rho * R^\pm_\kappa u$ with $\rho \in L^1_{t, x}(\RR^{1+n})$. Thus
        $$ \| C^\pm_{\les d} R^\pm_\kappa u \|_{PW^{\pm'}(\kappa)} \lesa \|  R^\pm_\kappa u \|_{PW^{\pm'}(\kappa)}$$
  for any choice of signs $\pm$, $\pm'$.
\end{remark}

%------------------------------------------------------------------------------------------------------------------------------%
  %------------------------------------------------------------------------------------------------------------------------------%
%------------------------------------------------------------------------------------------------------------------------------%
\section{Linear Estimates}\label{sec - linear est}
%------------------------------------------------------------------------------------------------------------------------------%
%------------------------------------------------------------------------------------------------------------------------------%

In this section we introduce the key linear estimate that we require. As similar versions of these estimates are known, at least for the related function spaces used in the waves maps case \cite{Tataru2001, Sterbenz2010, Krieger2012},  we leave the proofs till Sections \ref{sec - proof of null frame bounds} and \ref{sec - proof of stricharz est}.\\

 In the subcritical setting, Strichartz estimates have proven to be a key tool in the local and global well-posedness theory for the Dirac equation, especially in the $n=3$ case, see for instance \cite{Escobedo1997, Machihara2005a}. We would like to show that our iteration norm $F^+_\lambda$ controls the Strichartz type norms $L^q_t L^r_x$. For the $L^1_t L^2_x$ and $\dot{\mc{X}}^{-\frac{1}{2}, 1}_\pm$ components of our norms, we have a transference type principle, and hence, roughly speaking, any estimate satisfied by homogeneous solutions immediately holds general functions in spaces of the form $(\p_t \pm \sigma \cdot \nabla)^{-1} \big(L^1_t L^2_x + \dot{\mc{X}}^{-\frac{1}{2}, 1}_{\pm}\big)$, see for instance Section 4 in \cite{Sterbenz2004}. On the other hand, it is much more difficult to show that null frame component of our norms controls the Strichartz norms. In fact, in the case of the related function spaces used in the wave maps problem, initially only the ``off the line'' Strichartz estimates ($\frac{1}{q} + \frac{n-1}{2r} < \frac{n-1}{4}$ ) were known, see for instance \cite{Tao2001a, Krieger2003a}. However, recently, it was observed by Sterbenz-Tataru \cite{Sterbenz2010a} that the ``on the line'' Strichartz estimates also hold. In the current article, by adapting the argument used in \cite{Sterbenz2010a}, we can show that the space $F^+_\lambda$ also controls the Strichartz type norms.  Note that, in the homogeneous case,  the following estimates are immediate from the classical Strichartz estimates together with the $L^2_x$ orthogonality of the angular projections $P_{\lambda, \kappa}$.

\begin{theorem}[$F^\pm_\lambda$ controls Strichartz]\label{thm - F controls strichartz and angular sum}
Let $2 \les q , r \les \infty$ with $q>2$ and $\frac{1}{q} + \frac{n-1}{2r} \les \frac{n-1}{4}$. Suppose $u \in F^\pm_\lambda$. Then we have the estimate
        $$ \| u \|_{L^q_t L^r_x} \lesa \lambda^{n (\frac{1}{2} - \frac{1}{r}) - \frac{1}{q}} \| u \|_{F^{\pm}_\lambda}.$$
More generally, let $d \in 2^\ZZ$ and suppose that  $\mc{M}$ be Fourier multiplier with matrix valued symbol $m(\xi)$ such that $|m(\xi)|\lesa \delta$ for every $\xi \in \supp \widehat{u}$. Then with $(q, r)$ as above
        $$ \big\| \mc{M} C^{\pm'}_{\les d} u \big\|_{L^q_t L^r_x} \lesa \delta \lambda^{ n(\frac{1}{2} - \frac{1}{r}) - \frac{1}{q}} \| u \|_{F^{\pm}_{\lambda}}$$
where $\pm$ and $\pm'$ are independent choices of signs.
\begin{proof}
  See Subsection \ref{subsec - proof of Strichartz bounds} below.
\end{proof}
\end{theorem}

\begin{remark}
 If $\supp \widehat{u}$ is contained in a ball of radius $\mu \les \lambda$ in the annulus of size $\lambda$, then we can replace $\lambda^{ n ( \frac{1}{2} - \frac{1}{r}) - \frac{1}{q}}$ with the smaller $(\tfrac{\mu}{\lambda})^{ n ( \frac{1}{2} - \frac{1}{r}) - \frac{2}{q}} \lambda^{ n ( \frac{1}{2} - \frac{1}{r}) - \frac{1}{q}}$, see Remark \ref{rem - refined strichartz implies better bound for Vpm est} below. This small scale improvement follows from the refined Strichartz estimates of Klainerman-Tataru in \cite{Klainerman1999} and can be very useful in proving bilinear estimates, particularly in the high-high frequency interaction.  In the current article the high-high interaction is not particularly hard to deal with, and so this small scale refinement is not needed.
\end{remark}

The next set of linear estimates we require are bounds involve the null frame type norms $PW^\pm(\kappa)$ and $[NF^\pm]^*(\kappa)$.

\begin{theorem}[Null frame bounds]\label{thm - null frame bounds}  Let $ \alpha \ll 1$,  $\lambda \in 2^\ZZ$, $T>0$, and $ \rho \in C^\infty_0(\RR)$. Suppose $u \in F^\pm_\lambda$.  Then we have the estimates
    \begin{equation}\label{eqn - thm null frame bounds - NF bound} \Bigg(\sum_{\kappa \in \mc{C}_\alpha} \big\| R^\pm_{\kappa, \, \alpha^2 \lambda}  \Pi_+ u \big\|_{[NF^\pm]^*(\kappa)}^2  + \big\|   R^\mp_{\kappa, \, \alpha^2 \lambda}  \Pi_- u \big\|_{[NF^\pm]^*(\kappa)}^2\Bigg)^\frac{1}{2} \lesa \| u \|_{F^\pm_\lambda} \end{equation}
and
    \begin{equation}\label{eqn - thm null frame bounds - PW bound}
        \Bigg(\sum_{\kappa \in \mc{C}_\alpha} \big\| R^\pm_{\kappa, \, \alpha^2 \lambda}  \Pi_+\big[ \rho( \tfrac{t}{T})  u \big] \big\|_{PW^\mp(\kappa)}^2 + \big\| R^\mp_{\kappa, \, \alpha^2 \lambda}  \Pi_-\big[ \rho( \tfrac{t}{T})  u \big] \big\|_{PW^\mp(\kappa)}^2\Bigg)^{\frac{1}{2}} \lesa (\alpha \lambda)^{\frac{n-1}{2}} \| u \|_{F^\pm_\lambda}
    \end{equation}
where the implied constant is independent of $T$.
\begin{proof}
  See Subsection \ref{subsec - proof of null frame bounds} below.
\end{proof}
\end{theorem}

\begin{remark}
  Theorem \ref{thm - null frame bounds} is in some sense a null form estimate, as it depends on certain cancelations involving the projections $\Pi_\pm$. In particular, if we tried to replace the $PW^\mp(\kappa)$ norm with $PW^\pm(\kappa)$, then we would only obtain (\ref{eqn - thm null frame bounds - PW bound}) with the factor $(\alpha \lambda)^{ \frac{n-1}{2}}$ replaced with the much larger $\alpha^{\frac{n-3}{2}} \lambda^{\frac{n-1}{2}}$.
\end{remark}

\begin{remark}
  We may replace the multiplies $R^\pm_{\kappa, \alpha^2 \lambda}$ in Theorem \ref{thm - null frame bounds} with $R^\pm_\kappa C^\pm_{\lesa \alpha^2 \lambda}$. In particular, it is not necessary that $|\tau \pm |\xi| | \ll \alpha^2 \lambda$, it is enough to be localised to the larger region $|\tau \pm |\xi| | \lesa \alpha^2 \lambda$. This follows by noting that we can always reduce the later condition to the former by using the $\dot{\mc{X}}^{\frac{1}{2}, \infty}_\pm$ spaces. See for instance the proof of Corollary \ref{cor - F controls dual of N} below.
\end{remark}

The final result in this section is a duality type estimate that helps to reduce the number of bilinear estimates we need to prove.

\begin{corollary}[$F^\pm_\lambda$ controls dual of $\mc{N}^\pm_\lambda$]\label{cor - F controls dual of N}
Let $\lambda, \mu \in 2^{\ZZ}$. Assume $u \in F^\pm_\lambda$ and $v \in \mc{N}^\pm_\mu$. Then
        $$ \Big| \int_{\RR^{n+1}} u^\dagger v dx dt \Big| \lesa \| u \|_{F^\pm_{\mu}} \| v \|_{\mc{N}^\pm_{\lambda}}  .$$
\begin{proof}
After a reflection, we may assume that $\pm = +$. The atomic definition of $\mc{N}^+_\lambda$ implies that it suffices to consider the case where $v$ is an atom. If $v$ is an energy or $\dot{\mc{X}}^{-\frac{1}{2}, 1}_+$ atom, then the estimate follows easily by duality together with the estimate
    $$ \|  u \|_{L^\infty_t L^2_x} + \| u \|_{\dot{\mc{X}}^{\frac{1}{2}, \infty}_{ +}}  \lesa \| u \|_{F^+_{\mu}}$$
which follows from Lemma \ref{lem - F controls Xsb infty}. Thus it only remains to consider the case where $v$ is a $NF^+_{\lambda}$ atom. By definition, there exists $\alpha>0$ such that we have a decomposition $ v = \sum_\kappa F_\kappa$ with $\supp \widetilde{\Pi_\pm F}_\kappa \subset A^\pm_{\lambda, \alpha}(\kappa)$ and
        $$ \sum_{\kappa \in \mc{C}_\alpha} \| F_{\kappa} \|_{NF^+(\kappa)}^2 \les 1.$$
Note that if $\kappa \in \mc{C}_\alpha$ and $\kappa' \in \mc{C}_{\frac{\alpha}{4}}$ with $\kappa \cap \kappa' \not = \varnothing$, then for every $\omega \in \sph^{n-1}$,  $\omega \not \in 2\kappa$ implies $\omega \not \in 2\kappa'$ and consequently $\| u \|_{[NF^+]^*(\kappa)} \les \| u \|_{[NF^+]^*(\kappa')}$. Hence an application of the duality estimate (\ref{eqn - NF, NF* duality}) gives
    \begin{align*} \Big|\int u^\dagger v \big|
            &\les \sum_{\pm} \sum_{ \substack{ \kappa \in \mc{C}_\alpha, \kappa' \in \mc{C}_{\frac{\alpha}{4}} \\ \kappa \cap \kappa' \not = \varnothing}} \Big|\int \big(R^\pm_{\kappa'} C^\pm_{\les \alpha^2 \lambda} \Pi_\pm u\big)^\dagger F_\kappa dx dt\Big| \\
            &\les  \sum_{ \substack{ \kappa \in \mc{C}_\alpha, \kappa' \in \mc{C}_{\frac{\alpha}{4}} \\ \kappa \cap \kappa' \not = \varnothing}}  \|R^\pm_{\kappa'} \Pi_\pm u \|_{[NF^+]^*(\kappa)}
                    \| F_{\kappa} \|_{NF^+(\kappa)} \\
            &\les  \bigg( \sum_{\kappa' \in \mc{C}_{\frac{\alpha}{4}}}\| R^\pm_{\kappa'} C^\pm_{\les \alpha^2 \lambda} \Pi_\pm  u \|_{[NF^+]^*(\kappa')}^2 \bigg)^\frac{1}{2}.
    \end{align*}
We now decompose $R^\pm_{\kappa'} C^\pm_{\les \alpha^2 \lambda} \Pi_\pm  u = R^\pm_{\kappa', \frac{\alpha^2}{16} \lambda} \Pi_\pm u + \sum_{ d \approx \alpha^2 \lambda} R^\pm_\kappa C^\pm_d \Pi_\pm u$. The first term we can directly estimate by using Theorem \ref{thm - null frame bounds}. For the second term, we can not directly apply Theorem \ref{thm - null frame bounds}, as the support is not sufficiently close to the light cone. Instead, we use Lemma \ref{lem - Xsb decomp into free waves} to decompose into an average of free waves, and note that as $ \| \mc{U}_+(t) f \|_{F^+_\lambda} \lesa \| f \|_{L^2_x}$ (for $f$ with Fourier support in $|\xi| \approx \lambda$), we can apply Theorem \ref{thm - null frame bounds} to deduce that
   \begin{align*}  \sum_{d \approx \alpha^2 \lambda} \bigg( \sum_{\kappa' \in \mc{C}_{\frac{\alpha}{4}}}\| R^\pm_{\kappa'} C^\pm_{d} \Pi_\pm  u \|_{[NF^+]^*(\kappa')}^2 \bigg)^\frac{1}{2} &\lesa  \int_{|\tau| \approx \alpha^2 \lambda}  \bigg( \sum_{\kappa' \in \mc{C}_{\frac{\alpha}{4}}}\| R^\pm_{\kappa'} \Pi_\pm  \mc{U}_+(t)[f_\tau] \|_{[NF^+]^*(\kappa')}^2 \bigg)^\frac{1}{2}\, d\tau\\
   &\lesa \int_{|\tau| \approx \alpha^2 \lambda} \| f_\tau \|_{L^2} d\tau \lesa \| u \|_{\dot{\mc{X}}^{\frac{1}{2}, \infty}_+} \lesa \| u \|_{F^+_\lambda}
   \end{align*}
where we used Lemma \ref{lem - F controls Xsb infty} to estimate the $\dot{X}^{\frac{1}{2}, \infty}_+$ norm. Thus result follows.
\end{proof}
\end{corollary}

%------------------------------------------------------------------------------------------------------------------------------%
%------------------------------------------------------------------------------------------------------------------------------%
\section{Bilinear Estimates}\label{sec - bilinear est}
%------------------------------------------------------------------------------------------------------------------------------%
%------------------------------------------------------------------------------------------------------------------------------%

The key estimate we prove in this section is the following bilinear null form estimate.

\begin{theorem}[Bilinear estimate in $\mc{N}^\pm_\lambda$ - close cone case]\label{thm - bilinear null frame est close cone}
 Let $\delta \lesa \min\{ \lambda, N_1\} $. Assume $ v \in F^\mp_{N_1}$ has compact support in time. Suppose $F$ is scalar valued with $\supp \widetilde{F} \subset \big\{ |\xi| \lesa \lambda, \,\, \big| |\tau| -|\xi| \big| \lesa \delta\big\}$. Then
             \begin{equation}\label{eqn - thm bilinear null frame est close clone - est in statement of thm}\big\| S^\pm_{N_0, \ll \delta} \big(  F  \mathfrak{C}^\mp_{\ll \delta} v\big) \big\|_{\mc{N}^\pm_{N_0}} \lesa \big( \delta \min\{ \lambda, N_1\} \big)^\frac{n-1}{4} \| F \|_{L^2_{t, x}} \| v \|_{F^\mp_{N_1}}.\end{equation}
\end{theorem}

\begin{remark}
  We should emphasis that the implied constant in (\ref{eqn - thm bilinear null frame est close clone - est in statement of thm}) is \emph{independent} of $v$. In particular, although the Theorem \ref{thm - bilinear null frame est close cone} requires $v$ to have compact support in time, the implied constant \emph{does not} depend on the size of the support. This is due to fact that the only place the compact support assumption is needed, is to control the $PW^\pm(\kappa)$ type norms. By Theorem \ref{thm - null frame bounds}, this is possible provided we can write $ v = \rho( \tfrac{t}{T}) v$ for some $\rho \in C^\infty_0(\RR)$, with the implied constant being independent of $T$ and $v$, and consequently, independent of the size of the support. A similar comment applies to the bilinear estimates appearing in Corollary \ref{cor - bilinear null frame est far cone} and Corollary \ref{cor - bilinear L2 estimates} below.
\end{remark}

Theorem \ref{thm - bilinear null frame est close cone} contains the main multi-linear estimate contained in this article. In fact all other bilinear and trilinear estimates essentially follow by using Lemma \ref{lem - F controls Xsb infty} and (\ref{eqn - L2 control away from null cone}) to control the region away from the cone, and Corollary \ref{cor - F controls dual of N} to deduce bilinear estimates in $L^2_{t, x}$ by duality.

\begin{proof}[Proof of Theorem \ref{thm - bilinear null frame est close cone}]
After a reflection in $x$, we may assume that  $\pm = +$. Note that as $v \in F^-_{N_0}$, we have $\supp \widehat{v} \subset\{ |\xi| \approx N_0\}$. Thus the lefthand side of (\ref{eqn - thm bilinear null frame est close clone - est in statement of thm}) vanishes unless $\max\{ \lambda, N_0, N_1\} \approx \text{med}\{ \lambda, N_0, N_1\}$, where $\text{med}\{ a, b, c\}$ is the median of $a, b, c \in \RR$,  we make use of this simple observation later. Let $\mu = \min\{ \lambda, N_0, N_1\}$. We claim that it is enough to consider the case $\delta \les \mu$. To prove the claim, note that if $\delta \g \mu$, then using Lemma \ref{lem - F controls Xsb infty} we have the estimates
    \begin{equation} \label{eqn - thm bilinear null frame est - far cone case v}
 \big\| P_{N_0} \big(  F \mathfrak{C}^-_{ \gtrsim \mu} v \big) \big\|_{L^1_t L^2_x} \lesa  \mu^{\frac{n}{2}} \| F \|_{L^2_{t, x}} \| \mathfrak{C}^-_{ \gtrsim \mu} v \|_{L^2_{t, x}} \lesa \mu^{\frac{n-1}{2}} \| F \|_{L^2_{t, x}} \| v \|_{F^-_{N_1}}
    \end{equation}
and
        \begin{equation}\label{eqn - thm bilinear null frame est - far cone case output}
                \big\| S^+_{N_0, \gtrsim \mu} \big( F v \big) \big\|_{\dot{\mc{X}}^{-\frac{1}{2}, 1}_+} \les \mu^{-\frac{1}{2}} \big\| P_{N_0} (F v)\big\|_{L^2_{t, x}} \lesa \mu^{ \frac{n-1}{2}} \| F \|_{L^2_{t, x}} \| v \|_{L^\infty_t L^2_x} \lesa \mu^{ \frac{n-1}{2}} \| F \|_{L^2_{t, x}} \| v \|_{F^-_{N_1}}
         \end{equation}
together with the obvious bounds (\ref{eqn - N controlled by L1L2}) and (\ref{eqn - N controlled by X}), reduce the problem to estimating
$S^+_{N_0, \ll \mu} ( F  \mathfrak{C}^-_{\ll \mu} v)$, this is almost the case $\delta = \mu$, but we need to restrict the support of $F$ further. To this end, we observe the identity\footnote{ This follows by noting that if $(\tau+ \tau', \xi + \xi'), (\tau', \xi') \in \big\{ \big| |\tau| - |\xi| \big| \ll \mu\}$ and $|\xi + \xi'| \approx N_0$, $|\xi'| \approx N_1$
then the inequality
        $$\big| |\tau| - |\xi| \big| \les \big| |\tau + \tau'| - |\xi + \xi'| \big| + \big| |\tau'| - |\xi'| \big| + 2\min\{ |\xi + \xi'|, |\xi'| \}$$
shows that $\big| |\tau| - |\xi| \big| \lesa \mu + \min\{ |\xi + \xi'|, |\xi'|\}$ which gives the claimed identity. This method of deducing close cone information on $F$, from close cone information on the output $Fv$ and $v$, occurs frequently in what follows. }
    $$  S^+_{N_0, \ll \mu} ( F  \mathfrak{C}^-_{\ll \mu} v) = S^+_{N_0, \ll \mu} ( C_{\lesa a} F  \mathfrak{C}^-_{\ll \mu} v)$$
where $a = \mu + \min\{ N_0, N_1\}$. Then as we already have $\supp \widetilde{F} \subset \big\{ \big| |\tau| - |\xi| \big| \lesa \delta \,\}$ with $\delta \lesa \lambda$, we deduce that $\supp \widetilde{ C_{\lesa a} F } \subset \big\{ \big| |\tau| - |\xi| \big| \lesa \mu \,\}$. Thus we can reduce the case $\delta \g \mu$ to $\delta \les \mu$, hence it is enough to consider the case $\delta \les \mu$ as claimed.\\

  It remains to consider the case $\delta \les \mu$. To this end we decompose $S^+_{N_0, \ll \delta}$ into terms of the form $C^\pm_{\ll \delta} P_{N_0} \Pi_\pm$ and note that, after a reflection (in $t$ and $x$), it is enough to prove
            $$ \big\| C^+_{\ll \delta}P_{N_0} \Pi_+ \big( F C^\pm_{\ll \delta} v^\mp \big)\big\|_{\mc{N}^+_{N_0}} \lesa \big( \delta \min\{ \lambda, N_1\}\big)^{\frac{n-1}{4}} \| F \|_{L^2_{t, x}} \| v \|_{F^-_{N_1}}$$
 where we let $v^\mp =  \Pi_\mp v$.  We now decompose the distance to the cone into three main interactions
     \begin{align*} C^+_{\ll \delta} \big[  \big( F  C^\pm_{\ll \delta} v^\mp \big]
   &= \sum_{d \lesa \delta} C^+_d \big[   C^+_{\lesa d} F C^\pm_{ \lesa d} v^\mp \big] +  \sum_{d\lesa \delta} C^+_{\lesa d}\big[ C^+_{\lesa d} F C^\pm_{ d} v^\mp \big]   +  \sum_{ d \lesa \delta} C^+_{\ll d}\big[   C^+_{ d} F  C^\pm_{ \ll d} v^\mp \big]\\
   &= A_I + A_{II} + A_{III}.\end{align*}
  Roughly the strategy  is to use the $\dot{\mc{X}}^{\frac{1}{2}, 1}_\pm$ type estimates whenever the output $Fv$ or $v$ is away from the light cone (the interactions $A_I$ and $A_{II}$), and for the more delicate interaction, $A_{III}$, apply the null frame bounds in Theorem \ref{thm - null frame bounds}.\\

  \textbf{Case 1: $A_I$.} The main idea is to use the close cone condition to limit the possible angular interactions. The key tool to accomplish this will be the elementary angle estimate
        \begin{equation}\label{eqn - thm bilinear null frame est - key angle est}  \theta(\xi + \xi', \pm \xi')^2 \approx \frac{ \big| |\xi + \xi'| \mp |\xi'| - |\xi| \big|  \times \big| |\xi + \xi'| \mp |\xi'| + |\xi| \big|}{|\xi' + \xi| |\xi'| }.\end{equation}
This estimate is used as follows. Suppose that
            $$ (\tau, \xi) \in \supp \widetilde{ C_{\lesa d} F }, \qquad (\tau', \xi') \in \supp \widetilde{C^\pm_{\lesa d} v^\mp}, \qquad (\tau + \tau', \xi + \xi') \in \big\{ |\xi| \approx N_0, \,\, \big|\tau + |\xi| \big| \approx d\}.$$
  Then using the assumption $\delta \lesa \lambda$ we have
    $$ \big| |\xi + \xi'| \mp |\xi'| - \sgn(\tau) |\xi| \big| \les \big| \tau + \tau' + |\xi + \xi'| \big| + \big| \tau' \pm |\xi'| \big| + |\tau| + |\xi| \lesa \lambda$$
  and
    $$\big||\xi + \xi'|  \mp |\xi'|  + \sgn(\tau) |\xi| \big| \lesa \big| \tau + \tau' + |\xi + \xi'| \big| + \big| \tau' \pm |\xi'| \big| + \big||\tau| - |\xi|\big| \lesa d.$$
  Therefore (\ref{eqn - thm bilinear null frame est - key angle est}) shows that $\theta(\xi + \xi', \pm\xi') \lesa \sqrt{ \frac{ \lambda d }{N_0 N_1 }}$. This suggests that we should localise $v$ and the product $Fv$ to caps of radius $\sqrt{ \frac{ \lambda d }{N_0 N_1 }}$, as if the Fourier support of $v$ was contained in a cap of radius $\sqrt{ \frac{ \lambda d }{N_0 N_1 }}$, then the Fourier support of $Fv$ must be contained in a similar cap. More precisely, letting $\alpha = \sqrt{\frac{\lambda d}{N_0 N_1}}$, the angle estimate implies the decomposition
    $$ P_{N_0} C^+_{ d}\big[ C_{\lesa d} F C^\pm_{ \lesa d} v^\mp \big] = \sum_{\substack{\kappa, \bar{\kappa} \in \mc{C}_\alpha \\ \theta(\kappa, \bar{\kappa}) \lesa \alpha}} P^+_{N_0, \bar{\kappa}} C^+_{d}\big[ C_{\lesa d} F \,R^\pm_{\kappa} C^\pm_{ \lesa d} v^\mp \big]. $$
  Consequently, by orthogonality in $L^2_x$, together with an application of Bernstein's inequality, and the null structure estimate $|\Pi_{\omega} \Pi_{\omega'} |\lesa \theta(\omega, -\omega')$ we obtain
  \begin{align*} \big\|  P_{N_0} C^+_{ d}\Pi_+\big[ C_{\lesa d} F(t) C^\pm_{ \lesa d} v^\mp(t) \big] \big\|_{L^2_x} &\lesa  \bigg( \sum_{\substack{\kappa, \bar{\kappa} \in \mc{C}_\alpha \\ \theta(\kappa, \bar{\kappa}) \lesa \alpha}}  \Big\| \int_{\RR^n} |C_{\lesa d} \widehat{F}(t, \xi - \eta)| \,\Big| \Pi_{\frac{\xi}{|\xi|}} \Pi_{\mp\frac{\eta}{|\eta}}\widehat{R^\pm_{\kappa} C^\pm_{\lesa d} v}(t, \eta) \Big| d\eta \Big\|_{L^2_{\xi}[A^+_{N_0}(\bar{\kappa})]}^2 \bigg)^\frac{1}{2}\\
                    &\lesa \alpha  \big( \alpha \min\{ N_0, N_1\}\big)^\frac{n-1}{2} \big(\min\{ N_0, N_1\}\big)^{\frac{1}{2}}\| C_{\lesa d} F(t) \|_{L^2_x} \Big( \sum_{\kappa \in \mc{C}_\alpha}  \big\| R^\pm_\kappa C^\pm_{\lesa d} v(t) \big\|_{L^2_x}^2 \Big)^\frac{1}{2} \\
                    &\lesa d^{ \frac{n+1}{4}} \big( \min\{ \lambda, N_0, N_1 \} \big)^\frac{n-1}{4} \| C_{\lesa d} F(t) \|_{L^2_x} \| C^\pm_{\lesa d} v(t) \|_{L^2_x}.
    \end{align*}
  Thus taking the $L^2_t$ norm of both sides, we that $P_{N_0} C^+_{ d}\Pi_+\big[ C_{\lesa d} F C^\pm_{ \lesa d} v^\mp \big] $ is a multiple of a $\dot{\mc{X}}^{-\frac{1}{2}, 1}_+$ atom, in other words, the $A_I$ term is a sum of $\dot{\mc{X}}^{-\frac{1}{2}, 1}_+$ atoms. Therefore, the atomic definition of $\mc{N}^+_{N_0}$ gives
    \begin{align*}
      \Big\| \sum_{d \lesa \delta}  C^+_d P_{N_0} \Pi_+\big[   C_{\lesa d} F C^\pm_{ \lesa d} v^\mp \big] \Big\|_{\mc{N}^+_{N_0}}
            &\les \sum_{d \lesa \delta} d^{-\frac{1}{2}} \Big\| C^+_d P_{N_0} \Pi_+\big[   C_{\lesa d} F C^\pm_{ \lesa d} v^\mp_{N_1} \big] \Big\|_{L^2_{t, x}}\\
            &\lesa  \big( \min\{ N_0, N_1 \} \big)^\frac{n-1}{4}  \| F \|_{L^2_{t, x}} \sum_{d \lesa \delta}  d^{ \frac{n+1}{4}- \frac{1}{2}} \| C^\pm_{\lesa d} v \|_{L^\infty_t L^2_x}\\
      &\lesa \big( \delta \min\{ N_0, N_1\} \big)^{\frac{n-1}{4}} \| F \|_{L^2_{t, x}} \| v \|_{F^-_{N_1}}
    \end{align*}
where we used an application of (iii) in Lemma \ref{lem - mult are disposable} to dispose of the $C^\pm_{\les d}$ multiplier, and the fact that $\frac{n-1}{4} > 0$ to control the sum over $d$.\\

\textbf{ Case 2: $A_{II}$.} We follow a similar argument to that used to control $A_I$. Let $\alpha = \sqrt{\frac{\lambda d}{N_0 N_1}}$. A moments thought shows that, as in the $A_I$ case, we have the angle estimate $\theta(\xi+ \xi', \pm \xi') \lesa \alpha$. Consequently we have the decomposition
    $$ P_{N_0} C^+_{ \lesa d}\big[ C_{\lesa d} F C^\pm_{  d} v^\mp \big] = \sum_{\substack{\kappa, \bar{\kappa} \in \mc{C}_\alpha \\ \theta(\kappa, \bar{\kappa}) \lesa \alpha}} P^+_{N_0, \bar{\kappa}} C^+_{ \lesa d}\big[ C_{\lesa d} F \,R^\pm_\kappa C^\pm_{d} v^\mp \big]. $$
Moreover, for any caps $\kappa, \bar{\kappa} \in \mc{C}_\alpha$ with $\theta(\kappa, \bar{\kappa}) \lesa \alpha$ we have by Bernstein and the null form estimate
    \begin{align*} \big\| P^+_{N_0, \bar{\kappa}}\Pi_+\big[ F \,R^\pm_\kappa v^\mp \big]\big\|_{L^2_x} &\lesa \Big\| \int_{\RR^n} | \widehat{F}(\xi - \eta)| \big| \Pi_{\frac{\xi}{|\xi|} }\Pi_{\mp\frac{\eta}{|\eta|}} \widehat{R^\pm_{\kappa} v}(\eta) \big| d\eta \Big\|_{L^2_\xi[ A^+_{N_0}(\bar{\kappa})]} \\
            &\lesa \alpha \big(\alpha \min\{ N_0, N_1\}\big)^{\frac{n-1}{2}} \big(\min\{ N_0, N_1\} \big)^\frac{1}{2} \| F \|_{L^2_x} \big\| R^\pm_\kappa v \big\|_{L^2_x} \\
            &\lesa d^{\frac{n+1}{4}} \big( \min\{N_0, N_1 \} \big)^\frac{n-1}{4} \| F \|_{L^2_x} \big\| R^\pm_\kappa v \big\|_{L^2_x}
    \end{align*}
Therefore, by the $L^2_x$ orthogonality of the  projections $R^\pm_\kappa$ we have
    \begin{align*}
      \Big\| P_{N_0} C^+_{ \lesa d}\Pi_+\big[ C_{\lesa d} F \, C^\pm_{  d} v^\mp \big] \Big\|_{\mc{N}^+_{N_0}} &\les  \big\| P_{N_0} C^+_{ \lesa d}\Pi_+\big[ C_{\lesa d} F C^\pm_{  d} v^\mp \big] \big\|_{L^1_t L^2_x} \\
      &\lesa  \Big\| \Big(\sum_{ \substack{\kappa, \bar{\kappa} \in \mc{C}_\alpha\\ \theta(\kappa, \bar{\kappa}) \lesa \alpha}} \big\| P^+_{N_0, \bar{\kappa}} \Pi_+ \big[ C_{\lesa d} F  C^\pm_{  d} R^\pm_\kappa v^\mp \big]\big\|_{L^2_x}^2 \Big)^\frac{1}{2} \Big\|_{L^1_t} \\
      &\lesa d^{\frac{n+1}{4}} \big( \min\{ N_0, N_1 \} \big)^\frac{n-1}{4} \Big\| \| C_{\lesa d } F \|_{L^2_x} \Big(\sum_{ \kappa \in \mc{C}_\alpha} \big\|   C^\pm_{  d}  R^\pm_\kappa v^\mp \big\|_{L^2_x}^2 \Big)^\frac{1}{2} \Big\|_{L^1_t} \\
      &\lesa d^{\frac{n+1}{4}} \big( \min\{ N_0, N_1 \} \big)^\frac{n-1}{4}  \| F \|_{L^2_{t, x}}    \big\|   C^\pm_d v^\mp  \big\|_{L^2_{t, x}} \\
      &\lesa d^{\frac{n+1}{4}- \frac{1}{2} } \big( \min\{  N_0, N_1 \} \big)^\frac{n-1}{4} \| F \|_{L^2_{t, x}} \| v \|_{F^-_{N_1}}
    \end{align*}
where we made use of the $L^2_{t, x}$ bound away from the cone (\ref{eqn - L2 control away from null cone}). If we now apply the triangle inequality and sum up in $d\lesa \delta$, we obtain the required inequality. \\

\textbf{Case 3: $A_{III}$.} The remaining case $A_{III}$ is more difficult as it includes interactions where both the output and $v$ can be concentrated very close to the cone (so it is hard to use $X^{s, b}$ type norms), while $F$ can be (relatively) far from the light cone. Since $F$ is just an arbitrary $L^2$ function, it is not easy to use the fact that $F$ lies away from the light cone, in particular we have no weights of the form $|\tau + |\xi| |$ to exploit (which is what makes $X^{s, b}$ type norms so useful away from the light cone). The key observation is that it is only possible for $F$ and $v$ to produce interactions close to the light cone, if the spatial Fourier supports are at an angle of  $\sqrt{ \frac{\lambda d}{N_0 N_1}}$. The difference to the previous cases is that we will have a bound on the angle from above \emph{and} below. This angular separation allows us to make use of the null frame type spaces, in particular it means we can use the simple bilinear estimate in Lemma \ref{lem - PW NF L2 duality}.

Another way to view this, is that we can no longer put the output $Fv$, or the function $v$, into the $L^2_{t, x}$ based $\dot{\mc{X}}^{-\frac{1}{2}, 1}$ type spaces. Thus we are essentially forced to put $Fv$ in either $L^1_t L^2_x$ or the null frame version $L^1_{t_\omega} L^2_{x_\omega}$. Clearly we can not put $Fv \in L^1_t L^2_x$ as this would require a $L^2_t L^\infty_x$ bound on $v$ which is out of reach. The only remaining option is to put $Fv$ into the null frame type norms $L^1_{t_\omega} L^2_{x_\omega}$, and then $v \in L^2_{t_\omega} L^\infty_{x_\omega}$. This is possible, but requires that we spend powers of $N_1$. Thus in the case where $N_1$ is very large, we need to ensure that $v$ is localised to small caps to get the right constants. \\

We start by making the following observation. Suppose that
     \begin{equation}\label{eqn - thm bilinear null frame est - A_III supp assumps} (\tau, \xi) \in \supp \widetilde{ C_{ d} F }, \qquad (\tau', \xi') \in \supp \widetilde{C^\pm_{\ll d} v^\mp}, \qquad (\tau + \tau', \xi + \xi') \in \big\{ |\xi| \approx N_0, \,\, \big|\tau + |\xi| \big| \ll d\}.\end{equation}
Then since $|\tau| + |\xi| \approx \max\{ \big| |\tau| - |\xi| \big|, |\xi| \} \approx \max\{ d, |\xi| \}$ we have
    $$ \big| |\xi + \xi'| \mp |\xi'| - \sgn(\tau) |\xi| \big| = \big| (\tau + \tau' + |\xi + \xi'|) - ( \tau' \pm |\xi'|) + \sgn(\tau) ( |\tau| + |\xi|) \big| \approx \max\{ d, |\xi|\}.$$
Moreover, an application of (\ref{eqn - thm bilinear null frame est - key angle est}),  together with the observation that $\big| |\xi + \xi'| \mp |\xi'| + \sgn(\tau)|\xi| \big| \approx d$ gives
    \begin{equation}\label{eqn - thm bilinear null frame est - key angle est A_III} C \sqrt{\frac{d \max\{ d, |\xi|\} }{N_0 N_1}} \les \theta(\xi + \xi', \pm\xi') \lesa \sqrt{\frac{d \max\{ d, |\xi| \} }{N_0 N_1}}.\end{equation}
Note that we have a bound on the angle from above and below. This estimate suggests that we should decompose the output $Fv$, and $v$, into caps of size $\sqrt{\frac{|\xi| d}{N_0 N_1}}$. On the other hand, if we are at frequency $\mu$, and modulation $\ll d$, then geometrically, the natural size of the caps should be $\sqrt{\frac{d}{\mu}}$ as this matches the correct Fourier support properties used in the definition of the null frame atoms. The difficulty is that these natural sizes may not always match up, and particularly in the case $N_0 \approx N_1$, an additional decomposition is required. Consequently, unlike in the cases $A_{II}$ and $A_{III}$, we cannot consider all the various frequency interactions simultaneously. Thus we separate the argument into the cases  $N_0 \gg N_1$, $N_0 \ll N_1$, and $N_0 \approx N_1$. \\

\textbf{Case 3a: $A_{III}$ and $N_0 \gg N_1$.} Note that if we have $(\tau, \xi)$ and $(\tau', \xi')$ as in (\ref{eqn - thm bilinear null frame est - A_III supp assumps}), then $|\xi| \approx N_0$. Hence the angle estimate (\ref{eqn - thm bilinear null frame est - key angle est A_III}) implies that
         $$C \sqrt{\frac{d }{N_1}} \les \theta(\xi + \xi', \xi') \lesa \sqrt{\frac{d}{N_1}}.$$
Let $\alpha = \frac{C}{100}\sqrt{\frac{d}{N_1}}$ and $\beta =  \frac{C}{100} \sqrt{\frac{d}{N_0}}$, note that $\alpha \g \beta$.  By decomposing $v$ into caps of radius $\alpha$, and the output $Fv$ into caps of size $\beta$,  we have the identity
    $$ P_{N_0} C^+_{\ll d}\big[   C_{ d}  F  C^\pm_{ \ll d} v^\mp \big] =  \sum_{\bar{\kappa} \in \mc{C}_\beta} \sum_{\substack{ \kappa \in \mc{C}_\alpha \\ 5 \alpha \les \theta(\kappa, \bar{\kappa} ) \lesa \alpha }} P^{+, \beta}_{N_0, \bar{\kappa}} \big[ C_d F \, R^\pm_{\kappa, \, \alpha^2 N_1} v^\mp \big].$$
Since $\beta\les \alpha$, the angle condition implies that fixing a cap $\bar{\kappa} \in \mc{C}_\beta$ essentially fixes the cap $\kappa \in \mc{C}_\alpha$.   Therefore, by the orthogonality estimate in Lemma \ref{lem - PW NF L2 duality} and the $PW^\pm(\kappa)$ estimate in Theorem \ref{thm - null frame bounds} (note that $v$ has compact support in time), we have
\begin{align*}
      \bigg( \sum_{\bar{\kappa} \in \mc{C}_\beta} \bigg\| \sum_{\substack{ \kappa \in \mc{C}_\alpha \\ 5 \alpha \les \theta(\kappa, \bar{\kappa} ) \lesa \alpha }} P^{+, \beta}_{N_0, \bar{\kappa}}\Pi_+ &\big[ C_d F \, R^\pm_{\kappa, \, \alpha^2 \lambda} v^\mp \big] \bigg\|_{NF^+(\bar{\kappa})}^2 \bigg)^\frac{1}{2}\\
                                        &\lesa \Big( \sum_{\kappa \in \mc{C}_\alpha} \sum_{\substack{ \bar{\kappa} \in \mc{C}_\beta \\ 5 \alpha \les \theta(\kappa, \bar{\kappa} ) \lesa \alpha }} \big\| P^{+, \beta}_{N_0, \bar{\kappa}}\Pi_+ \big[ C^+_d F \,  R^\pm_{\kappa, \, \alpha^2N_1} v^\mp \big] \big\|_{NF^+(\bar{\kappa})}^2 \Big)^\frac{1}{2}\\
                                        &\lesa \| C_d F \|_{L^2_{t, x}} \Big( \sum_{\kappa \in \mc{C}_\alpha}  \big\| R^\pm_{\kappa, \, \alpha^2 N_1}    v^\mp \big\|_{PW^+(\kappa)}^2 \Big)^\frac{1}{2}\\
                                        &\lesa (\alpha N_1)^{\frac{n-1}{2}} \| C_d F \|_{L^2_{t, x}} \| v \|_{F^-_{N_1}}
                                        \lesa  ( d N_1 )^{\frac{n-1}{4}} \| C_d F \|_{L^2_{t, x}} \| v \|_{F^-_{N_1}}
    \end{align*}
Consequently $P_{N_0} C^+_{\ll d}\big[   C_{ d}  F  C^\pm_{ \ll d} v^\mp \big]$ is a multiple of a $NF^+_{N_0}$ atom, and therefore
$$ \Big\| \sum_{d \lesa \delta} P_{N_0} C^+_{\ll d}\Pi_+\big[   C_{ d}  F  C^\pm_{ \ll d} v^\mp \big] \Big\|_{\mc{N}^+_{N_0}}\lesa \| F \|_{L^2_{t, x}} \| v \|_{F^-_{N_1}} \sum_{d \lesa \delta} ( d N_1 )^{\frac{n-1}{4}} \lesa ( \delta N_1 )^{\frac{n-1}{4}}\| F \|_{L^2_{t, x}} \| v \|_{F^-_{N_1}}  $$
which is acceptable as $N_1 = \min\{ \lambda, N_0, N_1\}$. \\

\textbf{Case 3b: $A_{III}$ and $N_0 \ll N_1$.} We start by observing that if we have $(\tau, \xi)$, $(\tau', \xi')$ as in (\ref{eqn - thm bilinear null frame est - A_III supp assumps}), then $|\xi| \approx N_1$. Consequently, from (\ref{eqn - thm bilinear null frame est - key angle est A_III}) we deduce that
        $$C \sqrt{\frac{d }{N_0}} \les \theta(\xi + \xi', \xi') \lesa \sqrt{\frac{d}{N_0}}.$$
If we let $\alpha = \frac{C}{100} \sqrt{\frac{d}{N_1}}$ and $\beta = \frac{C}{100} \sqrt{\frac{d}{N_0}}$, then the angle estimate implies the decomposition
         $$ P_{N_0} C^+_{\ll d}\big[   C_{ d} F  C^\pm_{ \ll d} v^\mp_{N_1} \big] =  \sum_{\bar{\kappa} \in \mc{C}_\beta} \sum_{\substack{ \kappa \in \mc{C}_\alpha \\ 4 \beta \les \theta(\kappa, \bar{\kappa} ) \lesa \beta }} P^{+, \beta}_{N_0, \bar{\kappa}} \big[ C_d F \, R^{\pm}_{\kappa, \, \alpha^2 N_1} v^\mp \big].$$
The key difference to the previous case, is that since we now have $\alpha \les  \beta$, if we fix a cap $\bar{\kappa} \in \mc{C}_\beta$, the condition $\theta(\kappa, \bar{\kappa}) \approx \beta$ no longer fixes the cap $\kappa \in \mc{C}_\alpha$. Thus to regain an $\ell^2$ sum over the caps $\kappa \in \mc{C}_\alpha$, we need to exploit an additional orthogonality property. The key point is the identity
        $$ \theta\big(\sgn(\tau)\xi, \mp \xi'\big)^2 \approx  \frac{ \big| |\xi + \xi'| + \sgn(\tau)|\xi| \mp |\xi'| \big| \times \big| |\xi + \xi'| - \sgn(\tau)|\xi| \pm |\xi'| \big|}{|\xi| |\xi'|}.$$
As in the derivation of (\ref{eqn - thm bilinear null frame est - key angle est A_III}), for frequencies localised as in (\ref{eqn - thm bilinear null frame est - A_III supp assumps}), since $N_0 \ll N_1 \approx |\xi|$ we have the crude estimate $ \big| |\xi + \xi'| - \sgn(\tau)|\xi| \pm |\xi'| \big| \lesa N_1$
and hence
    $$ \theta\big(\sgn(\tau)\xi, \mp \xi'\big) \lesa \sqrt{\frac{d}{N_1}} \approx \alpha.$$
Therefore, if we also decompose $F$ into caps of radius $\alpha$,  we can refine our decomposition and obtain
    $$ P_{N_0} C^+_{\ll d}\big[   C_{ d} F  C^\pm_{ \ll d} v^\mp \big] = \sum_{\substack{ \kappa, \kappa' \in \mc{C}_\alpha\\ \theta(\kappa, \kappa') \lesa \alpha}}  \sum_{\substack{ \bar{\kappa} \in \mc{C}_\beta \\ 4 \beta \les \theta(\kappa, \bar{\kappa} ) \lesa \beta }} P^{+, \beta}_{N_0, \bar{\kappa}} \big[ C_d R_{\kappa'} F \, R^\pm_{\kappa, \, \alpha^2 N_1} v^\mp \big]$$
which gives the required orthogonality in the $\kappa$ sum. If we now apply Lemma  \ref{lem - PW NF L2 duality} and Theorem \ref{thm - null frame bounds} we deduce that
\begin{align*}
  \bigg( \sum_{\bar{\kappa} \in \mc{C}_\beta} \big\| P^{+, \beta}_{N_0, \bar{\kappa}} \Pi_+ \big[   C_{ d} F  C^\pm_{ \ll d} v^\mp \big] \big\|_{NF^+(\bar{\kappa})}^2 \bigg)^\frac{1}{2}
        &\les \sum_{\substack{ \kappa, \kappa' \in \mc{C}_\alpha\\ \theta(\kappa, \kappa') \lesa \alpha}} \bigg(  \sum_{\substack{ \bar{\kappa} \in \mc{C}_\beta \\ 4 \beta \les \theta(\kappa, \bar{\kappa} ) \lesa \beta }} \big\| P^{+, \beta}_{N_0, \bar{\kappa}} \big[ C_d R_{\kappa'} F \, R^\pm_{\kappa, \, \alpha^2 N_1} v^\mp \big] \big\|_{NF^+(\bar{\kappa})}^2 \bigg)^\frac{1}{2}\\
        &\lesa \sum_{\substack{ \kappa, \kappa' \in \mc{C}_\alpha\\ \theta(\kappa, \kappa') \lesa \alpha}} \| R_{\kappa'} F \|_{L^2_{t, x}} \big\| R^\pm_{\kappa, \, \alpha^2 N_1} v^\mp \big\|_{PW^+(\kappa)} \\
        &\lesa (\alpha N_1)^{\frac{n-1}{2}} \| F \|_{L^2_{t, x}} \| v \|_{F^-_{N_0}}.
\end{align*}
Thus, as in the previous case, $P_{N_0} C^+_{\ll d}\big[   C_{ d} F  C^\pm_{ \ll d} v^\mp \big]$ is a multiple of a $NF^+_{N_0}$ atom, and consequently by summing up over $d \lesa \delta$ we obtain  a constant of the size $( \delta N_1)^{\frac{n-1}{4}} \approx (\delta \min\{ \lambda, N_1 \} )^{\frac{n-1}{4}}$ as required. \\

\textbf{Case 3c: $A_{III}$ and $N_0 \approx N_1$}. Unlike the previous cases, we no longer have an estimate on $|\xi|$ from below. Thus to exploit the angle estimate (\ref{eqn - thm bilinear null frame est - key angle est A_III}), we need to dyadically decompose $F$ further into $F = P_{\ll d} F + \sum_{d \lesa \lambda' \lesa \lambda} P_{\lambda'} F$. After an application of the triangle inequality, we reduce to estimating
        \begin{equation}\label{eqn - thm bilinear null frame est - AIII case N_0 = N_1 F decomp} \sum_{d \lesa \delta} \big\| P_{N_0} C^+_{\ll d}\Pi_+ \big( C_d P_{\ll d} F C^\pm_{\ll d} v^\mp \big) \big\|_{\mc{N}^+_{N_0}} + \sum_{d \lesa \delta} \sum_{d \lesa \lambda' \lesa \lambda} \big\| P_{N_0} C^+_{\ll d}\Pi_+ \big( C_d P_{\lambda'} F C^\pm_{\ll d} v^\mp\big) \big\|_{\mc{N}^+_{N_0}}. \end{equation}
We start with the first term.  Assume that
 \begin{equation}\label{eqn - thm bilinear null frame est - AIII case supp assump} (\tau + \tau', \xi + \xi') \in  \big\{ |\xi| \approx N_0, \,\, \big| \tau + |\xi| \big| \ll d \}, \qquad (\tau, \xi ) \in \big\{ \big| |\tau| - |\xi| \big| \approx d \big\}, \qquad (\tau', \xi') \in \supp \widetilde{ C^\pm_{\ll d} v }\end{equation}
and $|\xi| \ll d$. Clearly, we have $|\tau| \approx d$. Moreover, the close cone condition on $(\tau + \tau', \xi + \xi')$ and $(\tau', \xi')$ implies that
    $$ d \approx |\tau| \approx \big| \tau - \big( \tau + \tau'  + |\xi + \xi'| \big)  + \big( \tau' \pm |\xi'| \big) \big| = \big| |\xi + \xi'| \mp |\xi'| \big|.$$
However, as $|\xi| \ll d$, this is only possible if $\pm = -$ and $d \approx N_1$.  In particular,  the first term in (\ref{eqn - thm bilinear null frame est - AIII case N_0 = N_1 F decomp}) is only nonzero if we have $\pm = -$ and $ d \approx N_1$. Consequently, from (\ref{eqn - thm bilinear null frame est - key angle est A_III}) we have $\theta(\xi + \xi', -\xi') \approx 1$ and so letting $\alpha = \frac{1}{100}$ we deduce the  identity
    $$ \sum_{d \lesa \delta} C^+_{\ll d} P_{N_0} \big( C_d P_{\ll d} F C^-_{\ll d} v^+\big) = \sum_{d \approx N_1} \sum_{\substack{ \kappa, \bar{\kappa} \in \mc{C}_\alpha \\ \alpha \ll \theta(\kappa, \bar{\kappa}) \approx \alpha}} P^{+, \alpha}_{N_0, \bar{\kappa}} \big( C_d P_{\ll d} F\,  R^-_{\kappa, \, \alpha^2 N_1} v^+\big).$$
Therefore
\begin{align*}
  \sum_{d \approx N_1} \big\| P_{N_0} C^+_{\ll d} \Pi_+ \big( C_d P_{\ll d} F C^\pm_{\ll d} v^\mp\big) \big\|_{\mc{N}^+_{N_0}} &\les \sum_{d \approx N_1} \Big( \sum_{\substack{ \kappa, \bar{\kappa} \in \mc{C}_\alpha \\ \alpha \ll \theta(\kappa, \bar{\kappa}) \approx \alpha}} \big\|  P^{+, \alpha}_{N_0, \bar{\kappa}} \big( C_d P_{\ll d} F  R^-_{\kappa, \, \alpha^2 N_1} v^+\big)\big\|_{NF^+(\bar{\kappa})}^2 \Big)^\frac{1}{2} \\
  &\lesa \| F \|_{L^2_{t, x}} \Big( \sum_{\kappa \in \mc{C}_\alpha} \| R^-_{\kappa, \, \alpha^2 N_1}  v^+ \|_{PW^+(\kappa)}^2 \Big)^\frac{1}{2} \\
  &\lesa \big(N_1\big)^\frac{n-1}{2} \| F \|_{L^2_{t, x}} \| v \|_{F^-_{N_1}} \approx \big( \delta \min\{ \lambda, N_1\} \big)^{\frac{n-1}{4}} \| F \|_{L^2_{t, x}} \| v \|_{F^-_{N_1}}
\end{align*}
where we used the fact that $N_1 \approx d \lesa \delta$ and the assumption $d \lesa \delta \lesa \lambda \lesa N_1$,  together with Lemma \ref{lem - PW NF L2 duality} and Theorem \ref{thm - null frame bounds}.

It remains to consider the more interesting second term in (\ref{eqn - thm bilinear null frame est - AIII case N_0 = N_1 F decomp}). Suppose that $(\tau + \tau', \xi + \xi')$ and $(\tau', \xi')$ are as in (\ref{eqn - thm bilinear null frame est - AIII case supp assump}). If $|\xi| \approx \lambda' \gtrsim d$, then the angle estimate (\ref{eqn - thm bilinear null frame est - key angle est A_III}) implies that
    \begin{equation}\label{eqn - thm bilinear null frame est - key angle est AIII with lambda'}
            C \sqrt{\frac{d \lambda' }{N_0 N_1}} \les \theta(\xi + \xi', \pm \xi') \lesa \sqrt{\frac{d \lambda' }{N_0 N_1}}
    \end{equation}
Let $\alpha = \frac{C}{100} \sqrt{\frac{ d \lambda' }{N_0 N_1}}$. The estimate on the angle (\ref{eqn - thm bilinear null frame est - key angle est AIII with lambda'}) implies that we should be decomposing the output $Fv$ and $v$ into caps of radius $\alpha$. However, as $\alpha^2 N_0 \approx \frac{\lambda'}{N_1} d$,  if we want to decompose $Fv$ into $NF^+_{N_0}$ atoms, the output $Fv$ should be at a distance $\frac{\lambda'}{N_1} d $ from the cone. Thus we need to first deal with the case where the distance to the cone is in the region $ \frac{\lambda'}{N_1} d \lesa \bullet \ll d$. To this end, by exploiting the null structure as in Case 1 above and using the decomposition
             $$ P_{N_0} C^+_{\frac{\lambda'}{N_1} d \lesa \bullet \ll  d}\big[   C_{ d}P_{\lambda'} F  C^\pm_{ \ll d} v^\mp\big] =  \sum_{\substack{ \kappa, \bar{\kappa} \in \mc{C}_\alpha \\ \theta(\kappa, \bar{\kappa} ) \approx \alpha }} P^+_{N_0, \bar{\kappa}} C^+_{\frac{\lambda'}{N_1} d \lesa \bullet \ll  d} \big[ C_d P_{\lambda'} F \,  C^\pm_{\ll d} R^\pm_{\kappa} v^\mp \big]$$
(which follows from (\ref{eqn - thm bilinear null frame est - key angle est AIII with lambda'}) ) we obtain  by an application of Corollary \ref{lem - mult are disposable}
    \begin{align*}
        \big\| P_{N_0} C^+_{\frac{\lambda'}{N_1} d \lesa \bullet \ll  d} \Pi_+ \big[   C_{ d} P_{\lambda'}F  C^\pm_{ \ll d} v^\mp \big] \big\|_{\mc{N}^+_{N_0}} &\lesa \Big( \frac{ d \lambda'}{N_1} \Big)^{-\frac{1}{2}} \Big\| \sum_{\substack{ \kappa, \bar{\kappa} \in \mc{C}_\alpha \\ \theta(\kappa, \bar{\kappa} ) \approx \alpha }} P^+_{N_0, \bar{\kappa}}\Pi_+ \big[ C_d P_{\lambda'}F \, C^\pm_{\ll d} R^\pm_{\kappa} v^\mp \big] \Big\|_{L^2_{t, x}}\\
        &\lesa \Big( \frac{ d \lambda'}{N_1}\Big)^{-\frac{1}{2}} \alpha \| F \|_{L^2_{t, x}}\Big\| \Big( \sum_{\kappa \in \mc{C}_\alpha} \| \widehat{ C^\pm_{\ll d} R^{\pm}_{\kappa} v} \|_{L^1_{\xi}}^2 \Big)^\frac{1}{2} \Big\|_{L^\infty_t} \\
        &\lesa \Big( \frac{ d \lambda'}{N_1}\Big)^{-\frac{1}{2}} \alpha \big( \alpha N_1 \big)^{\frac{n-1}{2}} (N_1)^{\frac{1}{2}} \| F \|_{L^2_{t, x}} \| C^\pm_{\ll d} v \|_{L^\infty_t L^2_x} \\
        &\lesa (\lambda' d)^{\frac{n-1}{4}} \| F \|_{L^2_{t,x }} \| v \|_{F^-_{N_1}}.
     \end{align*}
Hence summing up in $d$ and $\lambda'$ gives the required estimate. Similarly, if we fix $v$ to have modulation in $\frac{\lambda'}{N_0} d \lesa \bullet \ll d$, then by following a similar argument to that used in Case 2  we obtain
    \begin{align*}
        \big\| P_{N_0} C^+_{\ll  d}\Pi_+\big[   C_{ d} P_{\lambda'} F  C^\pm_{\frac{\lambda'}{N_0} d \lesa \bullet \ll d} v^\mp_{N_1} \big] \big\|_{\mc{N}^+_{N_0}}
                &\lesa  \Big\| \sum_{\substack{ \kappa, \bar{\kappa} \in \mc{C}_\alpha \\ \theta(\kappa, \bar{\kappa} ) \approx \alpha }} P^+_{N_0, \bar{\kappa}} \Pi_+ \big[ C_d P_{\lambda'} F \, P^{\pm}_{N_1, \kappa} C^\pm_{\frac{\lambda'}{N_0} d \lesa \bullet \ll d} v^\mp \big] \Big\|_{L^1_t L^2_{x}}\\
                &\lesa \alpha \big( \alpha N_1 \big)^\frac{n-1}{2} (N_1)^{\frac{1}{2}} \| F \|_{L^2} \| C^\pm_{\frac{\lambda'}{N_0} d \lesa \bullet \lesa d} v^\pm \|_{L^2_{t, x}} \\
                &\lesa ( \lambda' d )^{\frac{n-1}{4}} \| F \|_{L^2_{t, x}} \| v \|_{F^-_{N_1}}
     \end{align*}
which again is acceptable.

Thus it remains to deal with the term $C^+_{\ll  \frac{\lambda'}{N_1} d}\big[   C_{ d} P_{\lambda'} F  C^\pm_{\ll \frac{\lambda'}{N_0} d } v^\mp \big]$. By exploiting the angle estimate (\ref{eqn - thm bilinear null frame est - key angle est AIII with lambda'}), we have the decomposition
    $$P_{N_0}  C^+_{\ll  \frac{\lambda'}{N_1} d}\big[   C_{ d} P_{\lambda'} F  C^\pm_{\ll d } v^\mp \big] = \sum_{\substack{\kappa, \bar{\kappa} \in \mc{C}_\alpha \\ 5 \alpha \les \theta(\kappa, \bar{\kappa} ) \lesa \alpha }}  P^{+, \alpha}_{N_0, \bar{\kappa}} \big[ C_d P_{\lambda'} F \,R^\pm_{\kappa, \alpha^2 N_1}  v^\mp \big].$$
Therefore, as $P_{N_0}  C^+_{\ll  \frac{\lambda'}{N_1} d}\big[   C_{ d} P_{\lambda'} F  C^\pm_{\ll d } v^\mp \big]$ is now a multiple of a $NF^+_{N_0}$ atom, by again using Lemma \ref{lem - PW NF L2 duality} together with Theorem \ref{thm - null frame bounds} we obtain
    \begin{align*}
      \big\| C^+_{\ll  \frac{\lambda'}{N_1} d}\Pi_+ \big[   C_{ d} P_{\lambda'} F  C^\pm_{\ll d } v^\mp \big]\big\|_{\mc{N}^+_{N_0}} &\lesa \Big( \sum_{\substack{\kappa, \bar{\kappa} \in \mc{C}_\alpha \\ 5 \alpha \les \theta(\kappa, \bar{\kappa} ) \lesa \alpha }}  \big\| P^{+, \alpha}_{N_0, \bar{\kappa}}\Pi_+  \big[ C_d P_{\lambda'} F \,R^\pm_{\kappa, \alpha^2 N_1} v^\mp \big]\big\|_{NF^+(\bar{\kappa})}^2 \Big)^\frac{1}{2} \\
      &\lesa \| F \|_{L^2_{t, x}} \Big( \sum_{\kappa \in \mc{C}_\alpha} \| R^\pm_{\kappa, \alpha^2 N_1} v^\mp \big\|_{PW^+(\kappa)}^2 \Big)^\frac{1}{2} \\
      &\lesa (\alpha N_1 )^{\frac{n-1}{2}} \| F \|_{L^2_{t, x}} \| v \|_{F^-_{N_1}} \approx (\lambda' d)^{\frac{n-1}{4}}  \| F \|_{L^2_{t, x}} \| v \|_{F^-_{N_1}}
    \end{align*}
and hence by summing up in $d$ and $\lambda'$ we obtain the required estimate.
\end{proof}

\begin{remark}
  It is possible to improve the factors on the righthand side, for instance we can replace $\min\{ \lambda, N_1\}$ with $\min\{ \lambda, N_0, N_1\}$ by a minor additional argument. Other improvements are also possible, particularly in the high-high case $N_0 \approx N_1$. However as we have no need for any further refinements here, we leave this as an exercise to the interested reader.
\end{remark}

We now present a number of useful bilinear estimates that follow from Theorem \ref{thm - bilinear null frame est close cone}. The first is a ``far cone'' version of Theorem \ref{thm - bilinear null frame est close cone}.

\begin{corollary}[Bilinear estimates in $\mc{N}^\pm_N$ - Far cone case]\label{cor - bilinear null frame est far cone} Let $ v \in F^\mp_{N_1}$ have compact support in time.
\begin{enumerate}
 \item Let $1\les a \les 2$ and $\delta \gtrsim N_1$. Assume $F \in L^a_t L^2_x$ is scalar valued. Then
   $$\big\| P_{N_0} \big(F v\big) \big\|_{\mc{N}^\pm_{N_0}} + \big\| S^\pm_{N_0, \lesa \delta}\big( F \mathfrak{C}^\mp_{\lesa \delta} v\big) \big\|_{\mc{N}^\pm_{N_0}} \lesa \big( N_1\big)^{\frac{n-1}{2} + \frac{1}{a} - \frac{1}{2} } \| F \|_{L^a_t L^2_x} \| v \|_{F^\mp_{N_1}}. $$

 \item Let $1 \les a <2$, $ b \g 2$, and $\frac{1}{a} + \frac{n-1}{2b} \g 1$.  Assume $\lambda \ll N_1$ and $F \in L^2_{t, x} $ is scalar valued with $\supp \widehat{F} \subset \{ |\xi| \lesa \lambda\}$. Then
    $$ \big\| P_{N_0}\big(F v\big) \big\|_{\mc{N}^\pm_{N_0}} \lesa \Big(\lambda^{\frac{n-1}{2}} \| F \|_{L^2_{t, x}} + \big(N_1\big)^{\frac{n}{b} + \frac{1}{a} -1} \| C_{\gtrsim N_1} F \|_{L^a_t L^b_x} \Big) \| v \|_{F^\mp_{N_1}}.$$

\end{enumerate}
\begin{proof}
 \textbf{(i):} As usual, we may assume $\pm = +$. By interpolation, it is enough to considering the cases $a=1$ and $a=2$. The former case is simply an application of Sobolev embedding together with (\ref{eqn - N controlled by L1L2}). On the other hand, the $a=2$ case can be reduce to Theorem \ref{thm - bilinear null frame est close cone} by using $\dot{\mc{X}}^{b, q}_\pm$ spaces to deal with the region away from the cone. More precisely, if we let $\lambda \approx \max\{ N_0, N_1\}$, then from the estimates (\ref{eqn - thm bilinear null frame est - far cone case v}) and (\ref{eqn - thm bilinear null frame est - far cone case output}),  we reduce to estimating $ S^+_{N_0, \ll N_1} ( F \mathfrak{C}^-_{\ll N_1} v)$. Noting the identity\footnote{As in the proof of the $\delta \g \min\{ N_0, \lambda, N_2\}$ case in Theorem \ref{thm - bilinear null frame est close cone}, the identity follows from the inequality
    $$ \big| |\tau| - |\xi| \big| \les \big| |\tau + \tau'| - |\xi + \xi'| \big| + \big| |\tau'| - |\xi'| \big| + 2 \min\{ |\xi + \xi'|, |\xi'|\}. $$}
        $$S^+_{N_0, \ll N_1}( F \mathfrak{C}^-_{\ll N_1} v) = S^+_{N_0, \ll N_1} ( P_{\lesa \lambda} C_{\lesa N_1} F \mathfrak{C}^-_{\ll N_1} v)$$
the required estimate now follows from Theorem \ref{thm - bilinear null frame est close cone}.\\

 \textbf{(ii):} As in the proof of (i), we can use (\ref{eqn - thm bilinear null frame est - far cone case v}) and (\ref{eqn - thm bilinear null frame est - far cone case output}), to reduce to estimating the close cone term $S^+_{N_0, \ll \lambda}( F \mathfrak{C}^-_{\ll\lambda} v )$. We would like to now deduce that $F$ must also be $\lambda$ from the cone, and thus simply apply Theorem \ref{thm - bilinear null frame est close cone}. This is true for when we have the output $Fv$ and $v$ localised close to a $\pm$ cone, but can fail when we $Fv$ is near a $\pm$ cone, and $v$ is close to a $\mp$ cone. However in the latter case, we have the redeeming feature that $F$ is in fact $N_1 \gg \lambda$ away from the cone, which gives the second term in (i). The details are as follows.

 Write
        $$S^+_{N_0, \ll \lambda}( F \mathfrak{C}^-_{\ll\lambda} v ) = S^+_{N_0, \ll \lambda}( C_{\lesa \lambda}F \,\mathfrak{C}^-_{\ll\lambda} v ) + S^+_{N_0, \ll \lambda}( C_{ \gg \lambda} F\, \mathfrak{C}^-_{\ll\lambda} v ).$$
 For the first term we can simply use Theorem \ref{thm - bilinear null frame est close cone} with $\delta = \lambda$. On the other hand, for the second term, we observe  that if
 $$ (\tau, \xi) \in \supp \widetilde{ C_{ \gg \lambda} F}, \qquad (\tau', \xi') \in \supp \widetilde{\mathfrak{C}^-_{\ll\lambda} v}, \qquad (\tau + \tau', \xi + \xi')\in \{ |\xi| \approx N_0, \,\,  \big| |\tau| - |\xi| \big| \ll \lambda\}, $$
then we have
    \begin{align} \big| |\tau| - |\xi| \big| &\approx \big| \big( \tau - \sgn(\tau)|\xi|\big) + \big( \tau' - \sgn(\tau') |\xi'| \big) - \big( \tau + \tau' - \sgn(\tau+ \tau') |\xi + \xi'| \big) \big| \notag\\
            &=\big| \sgn(\tau + \tau') |\xi + \xi'| - \sgn(\tau') |\xi'| - \sgn(\tau) |\xi| \big| \label{eqn - cor bilinear null frame far cone - imply F far cone}\end{align}
 Therefore, as $N_1 \gg \lambda$, and the left hand side of (\ref{eqn - cor bilinear null frame far cone - imply F far cone}) is $\gg \lambda$, we must have $ \big| |\tau| - |\xi| \big| \approx |\xi'| $ and consequently we have the identity
    $$ S^+_{N_0, \ll \lambda} \big( C_{\gg \lambda} F \,\mathfrak{C}^-_{\ll \lambda} v \big) = S^+_{N_0, \ll \lambda} \big( C_{\gtrsim N_1} F \,\mathfrak{C}^-_{\ll \lambda} v  \big) .$$
 Note that the conditions on $(a, b)$ imply that the pair $(\frac{a}{a-1}, \frac{2b}{b - 2} )$ is Strichartz admissible. Thus, by an application of Holder's inequality together with Theorem \ref{thm - F controls strichartz and angular sum}, we deduce that
    \begin{align*}
      \big\| S^+_{N_0, \ll \lambda} \big( C_{\gtrsim N_1} F \,\mathfrak{C}^-_{\ll \lambda} v  \big) \big\|_{\mc{N}^+_{N_0}} &\les \big\| S^+_{N_0, \ll \lambda} \big( C_{\gtrsim N_1} F \,\mathfrak{C}^-_{\ll \lambda}  v \big) \big\|_{L^1_t L^2_x} \\
      &\lesa \| C_{\gtrsim N_1} F \|_{L^a_t L^b_x} \| \mathfrak{C}^-_{\ll \lambda}  v \|_{L^{\frac{a}{a-1}}_t L^{\frac{2b}{b-2}}_x} \\
      &\lesa  \big( N_1 \big)^{ n(\frac{1}{2} - \frac{b-2}{2b}) - \frac{a-1}{a}}\| C_{\gtrsim N_1} F \|_{L^a_t L^b_x} \| v \|_{F^-_{N_1}}\\
       &= \big( N_1 \big)^{ \frac{n}{b} + \frac{1}{a} - 1}\| C_{\gtrsim N_1} F \|_{L^a_t L^b_x} \| v \|_{F^-_{N_1}}
    \end{align*}
as required. \\

\end{proof}
\end{corollary}

By duality, Theorem \ref{thm - bilinear null frame est close cone} also implies the following bilinear estimates in $L^2_{t, x}$.

\begin{corollary}[$L^2$ Bilinear Estimates]\label{cor - bilinear L2 estimates} Let $u \in F^\pm_{N_1}$ and $v \in F^\mp_{N_2}$ have compact support in time.
\begin{enumerate}
   \item Let $\delta \les \min\{ N_1, N_2\}$. Then
            $$ \big\| C_{\lesa \delta} \big[ ( \mathfrak{C}^\pm_{\ll \delta} u)^\dagger \mathfrak{C}^\mp_{\ll \delta} v \big] \big\|_{L^2_{t, x}} \lesa \big(\delta \min\{ N_1, N_2\} \big)^{\frac{n-1}{4}} \| u \|_{F^\pm_{N_1}} \| v \|_{F^\mp_{N_2}}.$$

    \item Let $a, b \g 2$ and $\delta \g \min\{ N_1, N_2\}$. Then
            \begin{align*}  \big\| u^\dagger v\big\|_{L^a_t L^b_x} &+ \big\| ( \mathfrak{C}^\pm_{\lesa \delta} u)^\dagger \mathfrak{C}^\mp_{\lesa \delta} v \big\|_{L^a_t L^b_x} \\
            &\lesa \big( \min\{ N_1, N_2\}\big)^{\frac{n}{2} - \frac{1}{a}(\frac{1}{2} + \frac{1}{b})} \big( \max\{ N_1, N_2\}\big)^{ (n - \frac{1}{a})(\frac{1}{2} - \frac{1}{b})} \| u \|_{F^\pm_{N_1}} \| v \|_{F^\mp_{N_2}}.\end{align*}
\end{enumerate}

 \begin{proof}
 \textbf{(i):} After a reflection, it is enough to consider the case $N_1 \g N_2$ and $\pm = +$. Let $\lambda = 4 N_1$ and note that $ u^\dagger v = P_{\les \lambda}( u^\dagger v)$. An application of the the duality estimate in Corollary \ref{cor - F controls dual of N} together with Theorem \ref{thm - bilinear null frame est close cone} gives
    \begin{align*}
      \big\| C_{\lesa \delta} \big[ ( \mathfrak{C}^+_{\ll \delta} u)^\dagger \mathfrak{C}^-_{\ll \delta} v \big]  \big\|_{L^2_{t, x}}
            &= \sup_{\| F \|_{L^2_{t, x}} \les 1} \bigg| \int_{\RR^{n+1}} C_{\lesa \delta} P_{\les \lambda} F \,\big[ ( \mathfrak{C}^+_{\ll \delta} u)^\dagger \mathfrak{C}^-_{\ll \delta} v \big]  dt dx \bigg| \\
            &\les \sum_{N_0 \approx N_1} \sup_{\| F \|_{L^2_{t, x}} \les 1} \bigg| \int_{\RR^{n+1}} u^\dagger  S^+_{N_0, \ll \delta} \big[C_{\lesa \delta} P_{\les \lambda} F  \, \mathfrak{C}^-_{\ll \delta} v \big] dt dx \bigg|\\
            &\lesa \| u \|_{F^+_{N_1}} \sum_{N_0 \approx N_1} \sup_{\| F \|_{L^2_{t, x}} \les 1} \big\| S^+_{N_0, \ll \delta} \big[C_{\lesa \delta} P_{\les \lambda} F  \mathfrak{C}^-_{\ll \delta} v \big] \big\|_{\mc{N}^+_{N_0}} \\
            &\lesa \big( \delta \min\{N_1,  N_2\} \big)^{\frac{n-1}{4}} \| u \|_{F^+_{N_1}} \| v \|_{F^-_{N_2}}
    \end{align*}
 as required.

  \textbf{(ii):} The cases $(a, b) = (\infty, \infty), (\infty, 2), (2, \infty)$ follow by Sobolev embedding and the $L^4_t L^\infty_x$ Strichartz estimate in Theorem \ref{thm - F controls strichartz and angular sum}. Thus by interpolation we reduce to the case $(a, b) = (2, 2)$. Without loss of generality, we may assume $N_1 \g N_2$.  To deal with the far cone case, we use (\ref{eqn - L2 control away from null cone}) to obtain the inequalities
    $$ \big\| \big( \mathfrak{C}^+_{\gtrsim N_2} u \big)^\dagger v  \big\|_{L^2_{t, x}} \lesa \| \mathfrak{C}^+_{\gtrsim N_2} u \|_{L^2_{t, x}} \| v \|_{L^\infty_{t, x}} \lesa \big( N_2 \big)^{\frac{n-1}{2}} \| u \|_{F^+_{N_1}} \| v \|_{F^-_{N_2}}$$
  and
    $$ \big\| u^\dagger \mathfrak{C}^-_{\gtrsim N_2} v \big\|_{L^2_{t, x}} \les \| u \|_{L^\infty_t L^2_x} \| \mathfrak{C}^-_{\gtrsim N_2} v \|_{L^2_t L^\infty_x} \lesa \big( N_2 \big)^{\frac{n-1}{2}} \| u \|_{F^+_{N_1}} \| v \|_{F^-_{N_2}}.$$
  Therefore we reduce to estimating $(\mathfrak{C}^+_{\ll N_2}u )^\dagger \mathfrak{C}^-_{\ll N_2}v = C_{\lesa N_2} \big[(\mathfrak{C}^+_{\ll N_2}u )^\dagger \mathfrak{C}^-_{\ll N_2}v\big]$ in $L^2_{t, x}$, but this follows from $(i)$ by taking $\delta = N_2$.

 \end{proof}
\end{corollary}

%------------------------------------------------------------------------------------------------------------------------------%
%------------------------------------------------------------------------------------------------------------------------------%
\section{Cubic Estimates}\label{sec - cubic est}
%------------------------------------------------------------------------------------------------------------------------------%
%------------------------------------------------------------------------------------------------------------------------------%

We now come to main trilinear estimate we require. In this case, although the bilinear estimates only required $u \in F^\pm_N$, we are forced to make use of the stronger $G^\pm_N$ spaces. Essentially this is due to the fact that away from the light cone, we require the additional integrability in $t$ given by the $\mc{Y}^\pm$ norm. \\

 Let $N_{min}$, $N_{med}$, and $N_{max}$ denote, respectively, the minimum, the median, and the maximum, of the set $\{ N_1, N_2, N_3\}$. Our aim is to prove the following.

\begin{theorem}\label{thm - trilinear estimate}
Let $T>0$ and assume that $\pm$ and $\pm'$ are independent choices of signs. There exists $\epsilon>0$ such that if $u_1 \in G^{\pm'}_{N_1}$, $u_2 \in G^{\mp'}_{N_2}$, $u_3 \in G^\mp_{N_3}$ then
    $$ \big\| \ind_{(-T, T)}(t) P_{N_0} \big[\big(u_1^\dagger u_2\big) u_3 \big]\big\|_{\mc{N}^\pm_{N_0}} \lesa \big(N_{min} N_{med}\big)^\frac{n-1}{2} \Big(\frac{N_{min}}{N_{med}}\Big)^\epsilon \| u_1 \|_{G^{\pm'}_{N_1}} \| u_2\|_{G^{\mp'}_{N_2}} \| u_3 \|_{G^\mp_{N_3}}$$
where the implied constant is independent of $T$.
\end{theorem}

To deal with the close cone region, the following lemma will prove crucial.

\begin{lemma}\label{lem - trilinear supp identity}
  Let $\delta \g N_{min}$ and $\lambda = \min\{ \max\{ N_1, N_2\}, \max\{ N_0, N_1\} \}$. Assume $\supp \widetilde{u}_j \subset \big\{ \big| |\tau| - |\xi| \big| \les \delta, \,\, |\xi| \approx N_j \big\}$. Then
    $$ C_{\les \delta} P_{N_0} \big[ (u_1^\dagger u_2) u_3 \big]  = C_{\les \delta} P_{N_0} \Big( C_{\lesa \delta} P_{\lesa \lambda} \big[ u_1^\dagger u_2 \big] u_3 \Big).$$
  \begin{proof}
    The inequality
        \begin{equation}\label{eqn - lem trilinear supp ident - key ineq} \big| |\tau + \tau'| - |\xi + \xi'|\big| \les \big| |\tau| - |\xi| \big| + \big| |\tau'| - |\xi'| \big| + 2 \min\{ |\xi|, |\xi'| \}\end{equation}
    implies that if  $(\tau, \xi), (\tau', \xi') \in \big\{ \big| |\tau| - |\xi| \big| \les \delta \big\}$ and $\delta \g \min\{ N_1, N_2\}$, then $\big| |\tau + \tau'| - |\xi + \xi'| \big| \lesa \delta$. Consequently we have the identity $u^\dagger_1 u_2 = C_{\lesa \delta} P_{\lesa\max\{ N_1, N_1\}}\big[u^\dagger_1 u_2\big]$. On the other hand, if $\delta \g \min\{ N_0, N_1\}$, can again use  (\ref{eqn - lem trilinear supp ident - key ineq}) to deduce the identity $C_{\les \delta} P_{N_0} \big( F v_3 \big) = C_{\les \delta} P_{N_0} \big( C_{\lesa \delta} P_{\lesa \max\{ N_0, N_1\} }F v_3 \big)$. Thus lemma follows.
  \end{proof}
\end{lemma}

\begin{proof}[Proof of Theorem \ref{thm - trilinear estimate}]
After a reflection, we may assume $\pm = +$. We also fix $\pm'=+$, as the $\pm'=-$ argument is identical. Let $\rho \in C^\infty_0(\RR)$ with $\rho = 1$ on $[-1, 1]$ and fix $T^* \gg \max\{ T, N_{max}\}$. An application of (ii) in Theorem \ref{thm - energy inequality + invariance under cutoffs}, together with the identity
    $$ \ind_{(-T, T)}(t) P_{N_0} \big[\big(u_1^\dagger u_2\big) u_3 \big] =  \ind_{(-T, T)}(t) P_{N_0} \big[ \big( (\rho( \tfrac{t}{T^*})u_1)^\dagger (\rho( \tfrac{t}{T^*})u_2)\big) \rho( \tfrac{t}{T^*})u_3 \big]$$
 shows that it is enough to prove
   $$ \big\| P_{N_0} \big[\big(u_1^\dagger u_2\big) u_3 \big]\big\|_{\mc{N}^+_{N_0}} \lesa \big(N_{min} N_{med}\big)^\frac{n-1}{2} \Big(\frac{N_{min}}{N_{med}}\Big)^\epsilon \| u_1 \|_{G^+_{N_1}} \| u_2\|_{G^-_{N_2}} \| u_3 \|_{G^-_{N_3}}$$
under the additional assumption that each $u_j$ has compact support in time (thus, in particular, we can make use of the bilinear estimates in the previous section). Let $\frac{n+1}{4n}<a<\frac{1}{2}$ and take
        $$ \delta = \big( N_{min} \big)^{\frac{1}{2} - a} \big( N_{med} \big)^{\frac{1}{2} + a}.$$
The strategy is roughly to decompose into regions $\delta$ away from the light cone, and regions within $\delta$ of the light cone. In the close cone region, we can essentially just apply the bilinear estimate Theorem \ref{thm - bilinear null frame est close cone}. On the other hand, in the region away from the light cone, the argument is more involved and we need to exploit the bilinear estimates in Corollaries \ref{cor - bilinear L2 estimates} and \ref{cor - bilinear null frame est far cone}, together with the additional integrability in $t$ given by the $\mc{Y}^\pm$ norms. \\

We break the proof into three main cases, $N_3 = N_{max} \gg N_{med}$, $N_3 \approx N_{med}$, and $N_3 \approx N_{min}$. \\

\textbf{Case 1: $N_3 = N_{max} \gg N_{med}$.} We begin by decomposing
   \begin{equation}\label{eqn - thm trilinear est - case 1 initial decomp} (u_1^\dagger u_2) u_3 = C_{\lesa \delta}(u_1^\dagger u_2) u_3 + C_{\gg \delta}(u_1^\dagger u_2) u_3.\end{equation}
Note that by the $L^4_t L^\infty_x$ Strichartz estimate in Theorem \ref{thm - F controls strichartz and angular sum} we have the inequalities
    \begin{align} \big\| P_{N_0} \big( \big[ u_1^\dagger u_2 \big] \mathfrak{C}^-_{\gtrsim \delta} u_3 \big) \big\|_{\mc{N}^+_{N_0}} &\lesa \| u_1 \|_{L^4_t L^\infty_x} \| u_2 \|_{L^4_t L^\infty_x} \| \mathfrak{C}^-_{\gtrsim \delta} u_3 \|_{L^2_{t, x}}\notag \\
     &\lesa \big( N_{min} N_{med} \big)^{\frac{n-1}{2} + \frac{1}{4}} \delta^{-\frac{1}{2}} \| u_1 \|_{F^+_{N_1}} \| u_2 \|_{F^-_{N_2}} \| u_3 \|_{F^-_{N_3}} \label{eqn - thm trilinear est - output far cone}
     \end{align}
and
    \begin{align} \big\| \mathfrak{C}^+_{\gtrsim \delta}  P_{N_0} \big(\big[ u_1^\dagger u_2 \big]  u_3 \big) \big\|_{\mc{N}^+_{N_0}} &\lesa \delta^{-\frac{1}{2}} \| \big[ u_1^\dagger u_2\big]  u_3 \|_{L^2_{t, x}}\notag\\
                &\lesa \delta^{-\frac{1}{2}} \| u_1 \|_{L^4_t L^\infty_x} \| u_2 \|_{L^4_t L^\infty_x} \| u_3 \|_{L^\infty_t L^2_x}\notag\\
                 &\lesa \big( N_{min} N_{med} \big)^{\frac{n-1}{2} + \frac{1}{4}} \delta^{-\frac{1}{2}} \| u_1 \|_{F^+_{N_1}} \| u_2 \|_{F^-_{N_2}} \| u_3 \|_{F^-_{N_3}}. \label{eqn - thm trilinear est - max frequency far cone}
    \end{align}
Since $ \big( N_{min} N_{med} \big)^\frac{1}{4} \delta^{-\frac{1}{2}} = \Big( \frac{N_{min}}{N_{med}}\Big)^\frac{a}{2}$ both estimates are acceptable. Observe that neither (\ref{eqn - thm trilinear est - output far cone}) nor (\ref{eqn - thm trilinear est - max frequency far cone}) made any use of the structure of the product. Thus we can always control the case where the output, $(u_1^\dagger u_2) u_3$, or the $N_{max}$ term are $\delta$ from the cone, by putting the low frequency terms in $L^4_t L^\infty_x$. Moreover, as the estimates (\ref{eqn - thm trilinear est - output far cone}) and (\ref{eqn - thm trilinear est - max frequency far cone}) only made use of $L^2_x$ based spaces, together with Theorem \ref{thm - F controls strichartz and angular sum}, we can freely add $\mathfrak{C}^\pm_\delta$ multipliers to the left hand side without affecting the validity of (\ref{eqn - thm trilinear est - output far cone}) and (\ref{eqn - thm trilinear est - max frequency far cone}). These observations are also used in the case $N_3 \approx N_{med}$ and $N_3 \approx N_{min}$.

The estimate for the first term in (\ref{eqn - thm trilinear est - case 1 initial decomp}) is now straightforward. By (\ref{eqn - thm trilinear est - output far cone}) and (\ref{eqn - thm trilinear est - max frequency far cone}), we reduce to considering the term $\mathfrak{C}^+_{\ll \delta} \big( C_{\lesa \delta}\big[ u_1^\dagger u_2\big] \mathfrak{C}^-_{\ll \delta} u_3 \big)$. An application  of  Theorem \ref{thm - bilinear null frame est close cone} and Corollary \ref{cor - bilinear L2 estimates} then gives
    \begin{align*}
     \big\| \mathfrak{C}^+_{\ll\delta}  P_{N_0} \big(C_{\lesa \delta} \big[ u_1^\dagger u_2 \big]  \mathfrak{C}^-_{\ll \delta} u_3 \big) \big\|_{\mc{N}^+_{N_0}}
                    &\lesa \big( \delta N_{med} \big)^{\frac{n-1}{4}} \| u^\dagger_1 u_2 \|_{L^2_{t, x}}  \| u_3 \|_{F^-_{N_3}} \\
                    &\lesa \big( N_{min} \big)^\frac{n-1}{2} \big( \delta N_{med} \big)^{\frac{n-1}{4}} \| u_1\|_{F^+_{N_1}} \| u_2 \|_{F^-_{N_2}}  \| u_3 \|_{F^-_{N_3}}
    \end{align*}
which is acceptable since $\big( \tfrac{\delta}{N_{med}} \big)^{\frac{n-1}{4}} = \big( \tfrac{N_{min}}{N_{med}}\big)^{ ( \frac{1}{2} - a) \frac{n-1}{4}}$
and $\frac{n+1}{4n} < a < \frac{1}{2}$.

On the other hand, to deal with the second term in (\ref{eqn - thm trilinear est - case 1 initial decomp}), assume for the moment that we have the inequalities
\begin{equation}\label{eqn - thm trilinear est - case 1 bilinear delta from cone}
     \big\| C_{\gg \delta} ( u^\dagger_1 u_2) \big\|_{L^2_{t, x}} \lesa \big( N_{min} \big)^\frac{n}{2} \delta^{-\frac{1}{2}} \| u_1 \|_{F^+_{N_1}} \| u_2 \|_{F^-_{N_2}}
   \end{equation}
and
 \begin{equation}\label{eqn - thm trilinear est - case 1 bilinear N3 from cone} \big\| C_{\gtrsim N_3 }( u_1^\dagger u_2) \big\|_{L^\frac{8n}{5n+3}_t L^\frac{4n}{3}_x} \lesa \big( N_{min} N_{med} \big)^\frac{n-1}{2} \Big( \frac{N_{min} }{N_{med}} \Big)^{\frac{1}{4n}} \big(N_3 \big)^{- \frac{3(n+1)}{8n}} \| u_1 \|_{G^+_{N_1}} \| u_2 \|_{G^-_{N_2}}. \end{equation}
Then by (ii) in Corollary \ref{cor - bilinear null frame est far cone} with $(a, b) = (\frac{8n}{5n+3},\frac{4n}{3})$ (note that this pair is admissible),  we have
    \begin{align*}
        \big\| P_{N_0} \big[ C_{\gg \delta} ( u^\dagger_1 u_2) u_3 \big]\big\|_{\mc{N}^+_{N_0}} &\lesa  \Big[\big(N_{med}\big)^{\frac{n-1}{2}} \| C_{\gg \delta} ( u^\dagger_1 u_2) \|_{L^2_{t, x}} + \big(N_3\big)^{\frac{3(n +1)}{8n}} \| C_{\gtrsim N_3} ( u^\dagger_1 u_2) \|_{L^\frac{8n}{5n+3}_t L^\frac{4n}{3}_x} \Big]\| u_3 \|_{F^-_{N_3}}\\
        &\lesa  \Big[ \big( N_{min}\big)^\frac{n}{2} \big(N_{med}\big)^{\frac{n-1}{2}} \delta^{-\frac{1}{2}}  + \big( N_{min} N_{med} \big)^\frac{n-1}{2} \Big( \frac{N_{min} }{N_{med}} \Big)^{\frac{1}{4n}}  \Big] \| u_1 \|_{G^+_{N_1}} \| u_2 \|_{G^-_{N_2}}  \| u_3 \|_{F^-_{N_3}} \\
        &\lesa \big( N_{min} N_{med}\big)^{\frac{n-1}{4}} \Big( \frac{N_{min}}{N_{med}}\Big)^{\min\{ \frac{1}{4n}, \frac{a}{2}\} }  \| u_1 \|_{G^+_{N_1}} \| u_2 \|_{G^-_{N_2}} \| u_3 \|_{F^-_{N_3}}.
    \end{align*}
Thus, to complete the proof of the case $N_3 \gg N_{med}$, it only remains to deduce the inequalities (\ref{eqn - thm trilinear est - case 1 bilinear delta from cone}) and (\ref{eqn - thm trilinear est - case 1 bilinear N3 from cone}). We start with the more difficult (\ref{eqn - thm trilinear est - case 1 bilinear N3 from cone}). To this end, note that the inequality  (\ref{eqn - lem trilinear supp ident - key ineq}) and the assumption $N_3 \gg N_{med}$ implies the decomposition
    $$ C_{\gtrsim N_3 } \big[ u_1^\dagger u_2\big] = C_{\gtrsim N_3 } \big[ \big( \mathfrak{C}^+_{\gtrsim N_3} u_1 \big)^\dagger u_2\big] + C_{\gtrsim N_3 } \big[ \big( \mathfrak{C}^-_{\ll N_3} u_1 \big)^\dagger \mathfrak{C}^-_{\gtrsim N_3} u_2\big].$$
Essentially the point is that if the output of $u^\dagger_1 u_2$ is far from the cone, then it is not possible for both $u_1$ and $u_2$ to have Fourier support close the cone. If we now apply Lemma \ref{lem - mult are disposable} to dispose of the $C_{\gtrsim N_3}$ multiplier, followed by  the $L^\frac{8n}{5-n}_t L^\infty_x$ Strichartz estimate (note that $\frac{8n}{5-n} \g 4$), by the definition of the $\mc{Y}^\pm$ norm we deduce that
\begin{align*}
      \big\| C_{\gtrsim N_3 } \big[ \big( \mathfrak{C}^+_{\gtrsim N_3} u_1 \big)^\dagger u_2\big] \big\|_{L^\frac{8n}{5n+3}_t L^\frac{4n}{3}_x}
                &\lesa \big( N_1 \big)^{\frac{n}{2} - \frac{3}{4}} \big\|  \mathfrak{C}^+_{ \gtrsim N_3} u_1 \big\|_{L^\frac{4n}{3n-1}_t L^2_x} \| u_2 \|_{L^\frac{8n}{5-n}_t L^\infty_x} \\
                &\lesa \big( N_1 \big)^{\frac{n}{2} - \frac{3}{4}} \big( N_2 \big)^{ \frac{n}{2} - \frac{ 5-n}{8n} } \big( N_3 \big)^{-1} \big\|  u_1 \big\|_{\mc{Y}^+} \| u_2\|_{F^-_{N_2}} \\
                &\lesa \big( N_1 \big)^{\frac{n-1}{2}  + \frac{1}{4n}} \big( N_2 \big)^{ \frac{n-1}{2} + \frac{5(n-1)}{8n}} \big( N_3 \big)^{-1} \| u_1 \|_{G^+_{N_1}} \| v_2 \|_{F^-_{N_2}}
    \end{align*}
which is acceptable since, using the fact that $N_3 \g N_1, N_2$,
    $$ \big( N_1 \big)^{\frac{n-1}{2}  + \frac{1}{4n}} \big( N_2 \big)^{ \frac{n-1}{2} + \frac{5(n-1)}{8n}} \big( N_3 \big)^{-1}\les ( N_{min} N_{med} \big)^{\frac{n-1}{2}} \big( N_3\big)^{ - \frac{3(n+1)}{8n}} \Big( \frac{N_{min}}{N_{med}}\Big)^{\frac{1}{4n}}.$$
A similar argument handles the $C_{\gtrsim N_3 } \big[ \big( \mathfrak{C}^+_{\ll N_3} u_1 \big)^\dagger \mathfrak{C}^-_{\gtrsim N_3} u_2\big]$ term, we just put $\mathfrak{C}^+_{\ll N_3} u_1 \in L^\frac{8n}{5-n}_t L^\infty_x$ and $\mathfrak{C}^-_{\gtrsim N_3} u_2 \in \mc{Y}^-$. Thus we obtain (\ref{eqn - thm trilinear est - case 1 bilinear N3 from cone}). The proof of (\ref{eqn - thm trilinear est - case 1 bilinear delta from cone}) is similar, we just note that again the inequality (\ref{eqn - lem trilinear supp ident - key ineq}) implies that
    $$ C_{\gg \delta } \big[ u_1^\dagger u_2\big] = C_{\gg \delta } \big[ \big( \mathfrak{C}^+_{\gtrsim \delta} u_1 \big)^\dagger u_2\big] + C_{\gg \delta } \big[ \big( \mathfrak{C}^-_{\ll \delta} u_1 \big)^\dagger \mathfrak{C}^-_{\gtrsim \delta} u_2\big].$$
Hence putting the far cone terms in $L^2_{t, x}$ we obtain
    \begin{align*}
      \big\| C_{\gg \delta } \big[ u_1^\dagger u_2\big] \big\|_{L^2_{t, x}} &\lesa \big( N_{min} \big)^{\frac{n}{2}} \Big( \| \mathfrak{C}^+_{\gtrsim \delta} u_1 \|_{L^2_{t, x}} \| u_2 \|_{L^\infty_t L^2_x} + \| u_1 \|_{L^\infty_t L^2_x} \| \mathfrak{C}^-_{\gtrsim \delta} u_2 \|_{L^2_{t, x}} \Big) \\
      &\lesa \big( N_{min} \big)^{\frac{n}{2}} \delta^{-\frac{1}{2}} \| u_1 \|_{F^+_{N_1}} \| u_2 \|_{F^-_{N_2}}.
    \end{align*}
Therefore we obtain (\ref{eqn - thm trilinear est - case 1 bilinear delta from cone}) and so the case $N_3 = N_{max} \gg N_{med}$ follows. \\

\textbf{Case 2: $N_3 \approx N_{med}$.} Note that $N_3 \approx N_{med}$ implies that $N_{min} \approx \min\{ N_1, N_2\}$. If $u_1$ is $\delta$ away from the cone, then by $(i)$ in Corollary \ref{cor - bilinear null frame est far cone} we have
    \begin{align*}
      \big\| P_{N_0} \big( \big[ \big( \mathfrak{C}^+_{ \gtrsim \delta}  u_1\big)^\dagger u_2 \big] u_3 \big) \big\|_{\mc{N}^+_{N_0}} &\lesa \big( N_{med} \big)^\frac{n-1}{2} \big\| \big( \mathfrak{C}^+_{ \gtrsim \delta} u_1\big)^\dagger u_2 \big\|_{L^2_{t, x}} \| u_3 \|_{F^-_{N_3}} \\
                    &\lesa \big( N_{min}\big)^{\frac{n}{2}} \big(N_{med}\big)^\frac{n-1}{2} \big\| \mathfrak{C}^+_{ \gtrsim \delta} u_1 \big\|_{L^2_{t, x}} \| u_2 \|_{L^\infty_t L^2_x} \| u_3 \|_{F^-_{N_3}} \\
                    &\lesa \big( N_{min} N_{med} \big)^{\frac{n-1}{2} } \Big( \frac{N_{min}}{\delta}\Big)^\frac{1}{2} \| u_1 \|_{F^+_{N_1}} \| u_2 \|_{F^-_{N_2}} \| u_3 \|_{F^-_{N_3}}
    \end{align*}
which is acceptable as $\big( \tfrac{N_{min}}{\delta}\big)^\frac{1}{2} = \big( \frac{N_{min}}{N_{med}}\big)^{\frac{1}{4} + \frac{a}{2}}$. A similar argument handles the case where $u_2$ is $\delta$ away from the cone.

On the other hand, when $u_3$ is $\delta$ away from the cone, the argument is more involved as we need to make use of the $\mc{Y}^\pm$ norms to gain the correct factors $N_{min}$, $N_{med}$. An application of Holder together with $(ii)$ in  Corollary \ref{cor - bilinear L2 estimates} gives
    \begin{align*}
      \big\| P_{N_0} \big(   \big[ \big( \mathfrak{C}^+_{ \ll \delta} u_1 \big)^\dagger  \mathfrak{C}^-_{ \ll \delta} u_2 \big] \mathfrak{C}^-_{ \gtrsim \delta} u_3 \big)\big\|_{\mc{N}^+_{N_0}}&\les \big\| P_{N_0} \big(   \big[ \big( \mathfrak{C}^+_{ \ll \delta} u_1 \big)^\dagger  \mathfrak{C}^-_{ \ll \delta} u_2 \big] \mathfrak{C}^-_{ \gtrsim \delta} u_3 \big)\big\|_{L^1_t L^2_x} \\
      &\lesa \big\| \big( \mathfrak{C}^+_{ \ll \delta} u_1 \big)^\dagger  \mathfrak{C}^-_{ \ll \delta} u_2 \big\|_{L_t^\frac{4n}{n+1} L^2_x} \| \mathfrak{C}^-_{ \gtrsim  \delta} u_3 \|_{L^\frac{4n}{3n-1}_t L^\infty_x} \\
      &\lesa \big( N_{min} \big)^{\frac{n}{2} - \frac{n+1}{4n} } \big( N_{med} \big)^{\frac{n}{2}} \delta^{-1} \| u_1 \|_{F^+_{N_1}} \| u_2 \|_{F^-_{N_2}} \| u_3 \|_{\mc{Y}^-} \\
      &\lesa \big( N_{min} \big)^{\frac{n}{2} - \frac{n+1}{4n} } \big( N_{med} \big)^{\frac{n}{2} + \frac{n+1}{4n}  } \delta^{-1} \| u_1 \|_{F^+_{N_1}} \| u_2 \|_{F^-_{N_2}} \| u_3 \|_{G^-_{N_3}}
    \end{align*}
which is acceptable since
        $$ \big( N_{min} \big)^{\frac{n}{2} - \frac{n+1}{4n} } \big( N_{med} \big)^{\frac{n}{2} + \frac{n+1}{4n}} \delta^{-1} = \big( N_{min} N_{med} \big)^{\frac{n-1}{2}} \Big( \frac{N_{min}}{N_{med}}\Big)^{a - \frac{n+1}{4n}}$$
and $\frac{n+1}{4n} < a < \frac{1}{2}$. The final far cone case is when the output is $\delta$ from the cone. However here we can simply argue as in (\ref{eqn - thm trilinear est - output far cone}) but put the low frequency terms in $L^4_tL^\infty_x$, and the $N_{max}$ term in $L^\infty_t L^2_x$.

It remains to deal with the close cone case $S^+_{N_0, \ll \delta}\big(\big[ \big( \mathfrak{C}^+_{ \ll \delta} u_1 \big)^\dagger  \mathfrak{C}^-_{ \ll \delta} u_2 \big] \mathfrak{C}^-_{ \ll \delta} u_3 \big)$. To this end, as in the $N_3 \approx N_{max}$ case, we apply Lemma \ref{lem - trilinear supp identity}, Theorem \ref{thm - bilinear null frame est close cone}, and Corollary \ref{cor - bilinear L2 estimates} to obtain
    \begin{align*}
      \big\| S^+_{N_0, \ll \delta}\big(\big[ \big( \mathfrak{C}^+_{ \ll \delta} u_1 \big)^\dagger  \mathfrak{C}^-_{ \ll \delta} u_2 \big] \mathfrak{C}^-_{ \ll \delta} u_3 \big) \big\|_{\mc{N}^+_{N_0}} &= \big\| S^+_{N_0, \ll \delta}\big(C_{\lesa \delta}\big[  \big( \mathfrak{C}^+_{ \ll \delta} u_1 \big)^\dagger  \mathfrak{C}^-_{ \ll \delta} u_2 \big] \mathfrak{C}^-_{ \ll \delta} u_3 \big) \big\|_{\mc{N}^+_{N_0}}\\
                &\lesa \big( \delta N_{med} \big)^{\frac{n-1}{4}} \big\|  C_{\lesa \delta}\big[  \big( \mathfrak{C}^+_{ \ll \delta} u_1 \big)^\dagger  \mathfrak{C}^-_{ \ll \delta} u_2 \big] \big\|_{L^2_{t, x}} \| u_3 \|_{F^-_{N_3}} \\
                &\lesa \big( N_{min}\big)^{\frac{n-1}{2} }\big( \delta N_{med} \big)^{\frac{n-1}{4}} \| u_1 \|_{F^+_{N_1}} \| u_2 \|_{F^-_{N_2}} \| u_3 \|_{F^-_{N_3}}.
    \end{align*}
which is acceptable as $\big( \tfrac{\delta}{N_{med}}\big)^{\frac{n-1}{4}} = \big( \tfrac{N_{min}}{N_{med}} \big)^{\frac{n-1}{4}( \frac{1}{2} - a)}$. Therefore we obtain the case $N_3 \approx N_{med}$.\\

\textbf{Case 3: $N_3 \approx N_{min}$.} Without loss of generality, we assume that $N_1 \g N_2$, thus $N_1 \approx N_{max}$. The argument to control the case $N_3 \approx N_{min}$ is very similar to the previous case, essentially the only difference is that we need to reverse the order in which we estimate the far cone case to avoid having to estimate the multiplier $\mathfrak{C}^\pm_{ \ll\delta}$ in $F^\pm_N$ with $\delta \ll N$. As before, we start by dealing with the far cone case.  An application of Corollary \ref{cor - bilinear L2 estimates} gives the bound
    \begin{align*} \big\| P_{N_0} \big(\big[ u_1^\dagger u_2 \big] \mathfrak{C}^-_{\gtrsim \delta} u_3\big) \big\|_{\mc{N}^+_{N_0}}\les \big\| \big[ u_1^\dagger u_2 \big] \mathfrak{C}^-_{\gtrsim \delta} u_3 \big\|_{L^1_t L^2_x} &\lesa \big( N_{min} \big)^{\frac{n}{2}} \big\|  u_1^\dagger u_2 \big\|_{L^2_{t, x}} \big\| \mathfrak{C}^-_{ \gtrsim \delta} u_3 \big\|_{L^2_{t, x}} \\
        &\lesa \big( N_{med} \big)^{\frac{n-1}{2}} \big( N_{min} \big)^\frac{n}{2} \delta^{-\frac{1}{2}} \| u_1 \|_{F^+_{N_1}} \| u_2 \|_{F^-_{N_2}} \| u_3 \|_{F^-_{N_3}}
    \end{align*}
which is as before is acceptable. Together with (\ref{eqn - thm trilinear est - output far cone}), (but with the minor difference that we put the low frequency terms $u_2$ and $u_3$ in $L^4_t L^\infty_x$), we may assume that the output is within $\delta$ of the cone. Similarly, when $u_1$ is $\delta$ away from the cone, we follow (\ref{eqn - thm trilinear est - max frequency far cone})  and put $u_2, u_3 \in L^4_t L^\infty_x$ by using Theorem \ref{thm - F controls strichartz and angular sum}. Finally, if $u_2$ is $\delta$ away from the cone, then we use (i) in Corollary \ref{cor - bilinear null frame est far cone} together with the $\mc{Y}^\pm$ norm to deduce that
    \begin{align*}
       \big\| S^+_{N_0, \ll \delta} \big( \big[ \big( \mathfrak{C}^+_{\ll \delta} u_1 \big)^\dagger  \mathfrak{C}^-_{\gtrsim \delta} u_2 \big] \mathfrak{C}^-_{\ll \delta} u_3\big) \big\|_{\mc{N}^+_{N_0}}
            &\lesa \big( N_{min} \big)^{\frac{n}{2} - \frac{n+1}{4n}} \big\|  \big( \mathfrak{C}^+_{\ll \delta} u_1 \big)^\dagger  \mathfrak{C}^-_{\gtrsim \delta} u_2 \big\|_{L^\frac{4n}{3n-1}_t L^2_x} \| u_3 \|_{F^-_{N_3}} \\
            &\lesa  \big( N_{min} \big)^{\frac{n}{2} - \frac{n+1}{4n}} \big( N_{med} \big)^\frac{n}{2} \big\| \mathfrak{C}^+_{\ll \delta} u_1 \big\|_{L^\infty_t L^2_x} \big\|   \mathfrak{C}^-_{\gtrsim \delta} u_2 \big\|_{L^\frac{4n}{3n-1}_t L^2_x} \| u_3 \|_{F^-_{N_3}}\\
            &\lesa  \big( N_{min} \big)^{\frac{n}{2} - \frac{n+1}{4n}} \big( N_{med} \big)^{\frac{n}{2} } \delta^{-1} \| u_1 \|_{F^+_{N_1}} \| u_2 \|_{\mc{Y}^-} \| u_3 \|_{F^-_{N_3}} \\
            &\les \big( N_{min} N_{med} \big)^{\frac{n-1}{2}} \Big( \frac{N_{min}}{N_{med}}\Big)^{ a - \frac{n+1}{4n}} \| u_1 \|_{F^+_{N_1}} \| u_2 \|_{G^-_{N_2}} \| u_3 \|_{F^-_{N_3}}
    \end{align*}
which again is acceptable.

The final case is the close cone term $S^+_{N_0, \ll \delta} \big( \big[ \big( \mathfrak{C}^+_{\ll \delta} u_1 \big)^\dagger  \mathfrak{C}^-_{\ll \delta} u_2 \big] \mathfrak{C}^-_{\ll \delta} u_3\big) $. As previously, by applying Lemma \ref{lem - trilinear supp identity} and Corollaries \ref{cor - bilinear null frame est far cone} and \ref{cor - bilinear L2 estimates},  we obtain
    \begin{align*}
      \big\| S^+_{N_0, \ll \delta} \big( \big[ \big( \mathfrak{C}^+_{\ll \delta} u_1 \big)^\dagger  \mathfrak{C}^-_{\ll \delta} u_2 \big] \mathfrak{C}^-_{\ll \delta} u_3\big) \big\|_{\mc{N}^+_{N_0}} &= \big\| S^+_{N_0, \ll \delta} \big( C_{\lesa \delta} \big[ \big( \mathfrak{C}^+_{\ll \delta} u_1 \big)^\dagger  \mathfrak{C}^-_{\ll \delta} u_2 \big] \mathfrak{C}^-_{\ll \delta} u_3\big) \big\|_{\mc{N}^+_{N_0}}\\
                &\lesa \big( N_{min} \big)^{\frac{n-1}{2}} \big\|  C_{\lesa \delta} \big[ \big( \mathfrak{C}^+_{\ll \delta} u_1 \big)^\dagger  \mathfrak{C}^-_{\ll \delta} u_2 \big] \big\|_{L^2_{t, x}} \| u_3 \|_{F^-_{N_3}} \\
                &\lesa \big( N_{min}\big)^{\frac{n-1}{2} }\big( \delta N_{med} \big)^{\frac{n-1}{4}} \| u_1 \|_{F^+_{N_1}} \| u_2 \|_{F^-_{N_2}} \| u_3 \|_{F^-_{N_3}}
    \end{align*}
which is acceptable. Therefore we obtain the case $N_3 \approx N_{min}$ and hence theorem follows.
\end{proof}

To control the $\mc{Y}^\pm$ component of the $G^\pm_\lambda$ norm, we use the following.

\begin{theorem}\label{thm - control of Y norm}
Let $T>0$ and $\pm$,  $\pm'$ be independent choices of signs. There exists $\epsilon>0$ such that for $u_1 \in F^{\pm'}_{N_1}$, $u_2 \in F^{\mp'}_{N_2}$, and $u_3 \in F^\mp_{N_3}$ we have
    $$ \big\|\ind_{(-T, T)}(t) (u_1^\dagger u_2) u_3 \big\|_{L^\frac{4n}{3n-1}_t L^2_x} \lesa \big( N_{min} N_{med} \big)^{\frac{n-1}{2}} \Big( \frac{N_{min}}{N_{med}} \Big)^\epsilon \big( N_{max} \big)^{\frac{n+1}{4n}} \| u_1 \|_{F^{\pm'}_{N_1}} \| u_2 \|_{F^{\mp'}_{N_2}} \| u_3 \|_{F^\mp_{N_3}}$$
where the implied constant is independent of $T$.
\begin{proof}
As in the proof of Theorem \ref{thm - trilinear estimate}, we only consider the case $\pm = \pm' = +$ as the remaining cases are essentially identical. The required estimate follows  by an application of the Strichartz estimates in Theorem \ref{thm - F controls strichartz and angular sum}, together with the bilinear estimates in Corollary \ref{cor - bilinear L2 estimates}. More precisely, if $N_3 \approx N_{max}$, then as $(q, r) = (\frac{4n}{n-1}, \frac{2n}{n-1})$ is Strichartz admissible, by an application of Holder, followed by Theorem \ref{thm - F controls strichartz and angular sum} and $(ii)$ in Corollary \ref{cor - bilinear L2 estimates} with $(a, b) = (2, 2n)$, we deduce that
   \begin{align*} \| (u_1^\dagger u_2) u_3 \|_{L^\frac{4n}{3n-1}_t L^2_x} &\les \| u^\dagger_{N_1} u_2 \|_{L^2_t L^{2n}_x} \| u_3 \|_{L^{\frac{4n}{n-1}}_t L^{\frac{2n}{n-1}}_x}\\
    &\lesa \big( \min\{N_1, N_2\} \big)^{\frac{n-1}{2} + \frac{n-1}{4n}} \big( \max\{N_1, N_2\} \big)^{\frac{n-1}{2}  - \frac{n-1}{4}} \big( N_3 \big)^{ \frac{n+1}{4n}} \| u \|_{F^+_{N_1}} \| v \|_{F^-_{N_2}} \| v \|_{F^-_{N_3}} \\
    &\lesa \big( N_{min} N_{med} \big)^{\frac{n-1}{2}}  \Big( \frac{ N_{min}}{ N_{med}} \Big)^{\frac{n-1}{4n}} \big( N_{max} \big)^{\frac{n+1}{4n}}\| u \|_{F^+_{N_1}} \| v \|_{F^-_{N_2}} \| v \|_{F^-_{N_3}}
   \end{align*}
as required. On the other hand, if $N_3 \ll N_{max}$,  then we put $u_3 \in L^4_t L^\infty_x$ and again apply $(ii)$ in Corollary \ref{cor - bilinear L2 estimates} with $(a, b) = (\frac{4n}{2n-1}, 2)$ to obtain
    \begin{align*}
        \| (u_1^\dagger u_2) u_3 \|_{L^\frac{4n}{3n-1}_t L^2_x} &\les \| u^\dagger_{N_1} u_2 \|_{L^{\frac{4n}{2n-1}}_t L^2_x} \| u_3 \|_{L^4_t L^\infty_x} \\
         &\lesa \big( \min\{ N_1, N_2\} \big)^{\frac{n-1}{2} + \frac{1}{4n} } \big( N_3 \big)^{\frac{n-1}{2} + \frac{1}{4}} \| u \|_{F^+_{N_1} } \| v \|_{F^-_{N_2}} \| v \|_{F^-_{N_3}} \\
         &\les \big( N_{min} N_{med} \big)^{\frac{n-1}{2}} \Big( \frac{N_{min}}{N_{med}} \Big)^{\frac{1}{4n}} \big( N_{max} \big)^{\frac{n+1}{4}}  \| u \|_{F^+_{N_1} } \| v \|_{F^-_{N_2}} \| v \|_{F^-_{N_3}}.
    \end{align*}
\end{proof}
\end{theorem}

Combining the previous results, we deduce the following corollary.

\begin{corollary}\label{cor - cubic est}
  Let $s \g \frac{n-1}{2}$, $T>0$ and suppose $\pm$ and $\pm'$ are independent choices of signs. Assume $u_1 \in G^{\frac{n-1}{2}, \pm'}$, $u_2\in G^{\frac{n-1}{2}, \mp'}$, $u_3 \in G^{\frac{n-1}{2}, \mp}$. If we let
       $$ \Gamma=\| u_1 \|_{G^{\frac{n-1}{2}, \pm'}} +  \| u_2 \|_{G^{\frac{n-1}{2}, \mp'}} + \| u_3 \|_{G^{\frac{n-1}{2}, \mp}}$$
  then
    \begin{align*}
         \big\| \ind_{(-T, T)}(t) \big( u_1^\dagger u_2\big) u_3 \big\|_{(\mc{N} \cap \mc{Y})^{s, \pm}} \lesa \big( \| u_1 \|_{G^{s, \pm'}} +  \| u_2 \|_{G^{s, \mp'}} + \| u_3 \|_{G^{s, \mp}}\big) \Gamma^2
   \end{align*}
  where the implied constant is independent of $T$.
\begin{proof}
 As previously, we may assume that $\pm' = \pm = +$.  After dyadically decomposing $u_j$, an application of Theorems \ref{thm - trilinear estimate} and \ref{thm - control of Y norm} gives
    \begin{align*}\lambda^{\frac{n+1}{4n}}& \big\| \ind_{(-T, T)}(t) P_\lambda \big[\big( u_1^\dagger u_2\big) u_3 \big]\big\|_{\mc{N}^+_\lambda} + \big\| \ind_{(-T, T)}(t) P_\lambda \big[\big( u_1^\dagger u_2\big) u_3 \big]\big\|_{L^{\frac{4n}{3n-1}}_t L^2_x}\\
     &\lesa \sum_{ N_{max} \approx \max\{ \lambda, N_{med}\}} \big( N_{min} N_{med} \big)^\frac{n-1}{2} \Big( \frac{N_{min}}{N_{med}}\Big)^\epsilon (N_{max})^{\frac{n+1}{4n}}  \big\| P_{N_1} u_1 \big\|_{G^+_{N_1}} \big\| P_{N_2} u_2 \big\|_{G^-_{N_2}}  \big\| P_{N_3} u_3 \big\|_{G^-_{N_3}}.
    \end{align*}
  The required estimate now follows by summing up over $\lambda$, and exploiting the $\big( \tfrac{N_{min}}{N_{med}}\big)^\epsilon$ factor by using the inequality
        $$ \sum_{\substack{\lambda_1, \lambda_2 \in 2^\ZZ \\ \lambda_1 \les \lambda_2}} \big( \tfrac{\lambda_1}{\lambda_2}\big)^\epsilon a_{\lambda_1} b_{\lambda_2}  \lesa \bigg( \sum_{\lambda_1} (a_{\lambda_1})^2 \bigg)^\frac{1}{2} \bigg( \sum_{\lambda_2} (b_{\lambda_2})^2 \bigg)^\frac{1}{2}.$$
 In more detail, as we have now exploited all the structural properties of the product, we may assume that $N_1 \g N_2 \g N_3$. We consider separately the cases $N_1 \gg N_2$, and $N_1 \approx N_2$. In the former case, by summing up in $\lambda$ we obtain
 \begin{align*}
   \bigg(\sum_\lambda \lambda^{2 s} \Big( \big\| &\ind_{(-T, T)}(t) P_\lambda\big[\big( u_1^\dagger u_2\big) u_3 \big]\big\|_{\mc{N}^+_\lambda} + \lambda^{ - \frac{n+1}{4n}} \big\| \ind_{(-T, T)}(t) P_\lambda \big[\big( u_1^\dagger u_2\big) u_3 \big]\big\|_{L^{\frac{4n}{3n-1}}_t L^2_x}\Big)^2 \bigg)^\frac{1}{2} \\
   &\lesa \bigg( \sum_{\lambda} \bigg( \sum_{N_1 \approx \lambda, N_2 \g N_3} \big( N_{2} N_{3} \big)^\frac{n-1}{2} \Big( \frac{N_{3}}{N_{2}}\Big)^\epsilon (N_{1})^s  \big\| P_{N_1} u_1 \big\|_{G^+_{N_1}} \big\| P_{N_2} u_2 \big\|_{G^-_{N_2}}  \big\| P_{N_3} u_3 \big\|_{G^-_{N_3}}\bigg)^2 \bigg)^\frac{1}{2} \\
     &\lesa \sum_{N_2\g N_3}  \big( N_{2} N_{3} \big)^\frac{n-1}{2} \Big( \frac{N_{3}}{N_{2}}\Big)^\epsilon \big\| P_{N_2} u_2 \big\|_{G^-_{N_2}}  \big\| P_{N_3} u_3 \big\|_{G^-_{N_3}} \bigg( \sum_{N_1} (N_1)^{2s} \big\| P_{N_1} u_1 \big\|_{G^+_{N_1}}^2 \bigg)^\frac{1}{2} \\
     &\lesa \| u_1 \|_{G^{s, +}} \| u_2 \|_{G^{\frac{n-1}{2}, -} } \| u_3 \|_{G^{\frac{n-1}{2},-}}.
 \end{align*}
On the other hand, if $N_1 \approx N_2$, then as $s - \frac{n+1}{4n} > 0$, we deduce that
  \begin{align*}
   \bigg(\sum_\lambda \lambda^{2 s} \Big( \big\|& \ind_{(-T, T)}(t) P_\lambda\big[\big( u_1^\dagger u_2\big) u_3 \big]\big\|_{\mc{N}^+_\lambda} + \lambda^{ - \frac{n+1}{4n}} \big\| \ind_{(-T, T)}(t) P_\lambda \big[\big( u_1^\dagger u_2\big) u_3 \big]\big\|_{L^{\frac{4n}{3n-1}}_t L^2_x}\Big)^2 \bigg)^\frac{1}{2} \\
        &\lesa \sum_{ N_1 \approx N_2 \g N_3} \big( N_2 N_3 \big)^\frac{n-1}{2} \Big( \frac{N_3}{N_2}\Big)^\epsilon (N_1)^{\frac{n+1}{4n}}  \big\| P_{N_1} u_1 \big\|_{G^+_{N_1}} \big\| P_{N_2} u_2 \big\|_{G^-_{N_2}}  \big\| P_{N_3} u_3 \big\|_{G^-_{N_3}} \bigg( \sum_{\lambda \les N_1} \lambda^{2(s - \frac{n+1}{4})} \bigg)^\frac{1}{2} \\
        &\lesa \sum_{N_1 \approx N_2} \big( N_1\big)^s \big( N_2 \big)^\frac{n-1}{2} \big\| P_{N_1} u_1 \big\|_{G^+_{N_1}} \big\| P_{N_2} u_2 \big\|_{G^-_{N_2}} \sum_{N_3 \les N_2} \Big( \frac{N_3}{N_2} \Big)^\epsilon (N_3)^\frac{n-1}{2} \| P_{N_3} u_3 \|_{G^-_{N_3}} \\
        &\lesa \| u_1 \|_{G^{s, +}} \| u_2 \|_{G^{\frac{n-1}{2}, -} } \| u_3 \|_{G^{\frac{n-1}{2},-}}.
  \end{align*}
Therefore result follows.
\end{proof}
\end{corollary}

%------------------------------------------------------------------------------------------------------------------------------%

%------------------------------------------------------------------------------------------------------------------------------%
 %------------------------------------------------------------------------------------------------------------------------------%
%------------------------------------------------------------------------------------------------------------------------------%
\section{Proof of Global Well-posedness}\label{sec - proof of GWP}
%------------------------------------------------------------------------------------------------------------------------------%
%------------------------------------------------------------------------------------------------------------------------------%

The proof of global existence and scattering follows from a more or less standard argument from the energy type inequality in Theorem \ref{thm - energy inequality + invariance under cutoffs}, together with the crucial trilinear estimate in Corollary \ref{cor - cubic est}. We begin by considering the smooth case.

\begin{theorem}\label{thm - gwp for smooth data}
Let $n=2, 3$, $m=0$, and $s\g \frac{n-1}{2}$. Let $\rho \in C^\infty_0(\RR)$.  There exists $\epsilon>0$ such that if $f, g \in C^\infty_0(\RR^n)$ with
        \begin{equation}\label{eqn - thm gwp smooth data - smallness assump} \| f \|_{\dot{H}^{\frac{n-1}{2}}\cap \dot{H}^s} + \| g \|_{\dot{H}^\frac{n-1}{2}\cap \dot{H}^s} < \epsilon \end{equation}
then we have a global solution $(u, v) \in C^\infty( \RR^{1+n})$ to (\ref{eqn - general form of eqn}) such that $(u, v)(0) = (f, g)$ and
        $$ \| u \|_{L^\infty_t \dot{H}^s} + \| v \|_{L^\infty_t \dot{H}^s} \lesa \| f \|_{\dot{H}^s} + \| g \|_{\dot{H}^s}.$$
Moreover, if $f', g' \in C^\infty_0(\RR^n)$ also satisfies (\ref{eqn - thm gwp smooth data - smallness assump}) and $(u', v')$ denotes the corresponding solution to (\ref{eqn - general form of eqn}) with $(u', v')(0) = (f', g')$, then we have the Lipschitz bound
        \begin{align*} \sup_{T>0} \Big(\big\| \rho(\tfrac{t}{T}) \big(  u - u'\big)  \big\|_{F^{s, +}} + \big\| \rho(\tfrac{t}{T}) \big( v - v'\big) \big\|_{F^{s, -}} \Big) \lesa  \| f - f'\|_{\dot{H}^\frac{n-1}{2} \cap \dot{H}^s} + \| g - g' \|_{\dot{H}^\frac{n-1}{2} \cap \dot{H}^s}. \end{align*}
\begin{proof}
 Let $(f, g) \in C^\infty_0(\RR^n)$ satisfy (\ref{eqn - thm gwp smooth data - smallness assump}).  A standard fixed point argument in\footnote{Alternatively we could appeal to work of Pecher \cite{Pecher2013a, Pecher2015} in the $n=2$ case, and Escobedo-Vega \cite{Escobedo1997} in the $n=3$ case, although these results do not \emph{directly} deal with the system (\ref{eqn - general form of eqn}).} $L^\infty_t H^N$ with $N >\frac{n}{2}$ shows that there exists $T^*>0$ and a smooth solution $(u, v) \in C\big( (-T^*, T^*), H^N (\RR^n) \big)$ with $(u, v)(0) = (f, g)$. Let $T < T^*$ and define $(u_T, v_T)$ as the solution to
         \begin{equation*} \begin{split}
      (\p_t + \sigma \cdot \nabla) u_T &=  \ind_{(-T, T)} \big( B_1(u, v) v + B_2(u, v) \beta u \big)  \\
      (\p_t - \sigma \cdot \nabla) v_T &=  \ind_{(-T, T)} \big( B_3(u, v) u + B_4(u, v) \beta v  \big)
    \end{split}
    \end{equation*}
with $\big( u_T(0), v_T(0) \big) = (f, g)$. Note that $(u_T, v_T)$ is the extension of $(u, v)$ from $(-T, T)\times \RR^n$ to $\RR^{1+n}$ by a linear solution, in particular, we have $(u_T, v_T) = (u, v)$ on $(-T, T) \times \RR^n$. Define
        $$a_s(T) = \| u_T \|_{G^{s, +}} + \| v_T \|_{G^{s, -}}.$$
The bound
    $$ \| F \|_{(\mc{N} \cap \mc{Y})^{s, \pm}} \lesa \big\| F \big\|_{L^1_t \dot{H}^s_x} + \big\| F \big\|_{L^\frac{4n}{3n-1}_t \dot{H}^{s - \frac{n+1}{4n}}_x}$$
together with the equation for $(u_T, v_T)$, implies that for $T, T' \les T^*$
        $$ \big| a_s(T) - a_s(T') \big| \lesa_{u, v, T^*} |T - T'|^{\frac{3n-1}{4n}}.$$
In particular, $a_s(T)$ is a continuous function of $T$. Moreover, an application of Corollary \ref{cor - cubic est} gives
        \begin{equation} \label{eqn - thm gwp for smooth data - a cubic est} a_s(T) \les \| f \|_{\dot{H}^s} + \| g \|_{\dot{H}^s} +  C \big( a_\frac{n-1}{2}(T) \big)^2 a_{s}(T). \end{equation}
Thus as we clearly have $a_s(0) \les \|f \|_{\dot{H}^s} + \| g \|_{\dot{H}^s}$, a continuity argument shows that provided $\epsilon>0$ is sufficiently small (independent of $T$ and $T^*$) we have for every $T< T^*$
        \begin{equation}\label{eqn - thm gwp for smooth data - a est} a_s(T) \les 2 \|f \|_{\dot{H}^s} + \| g \|_{\dot{H}^s}.\end{equation}
Hence we have the bound
        $$ \| u \|_{L^\infty_t \dot{H}^s( (-T^*, T^*) \times \RR^n)} + \| v\|_{L^\infty_t \dot{H}^s( (-T^*, T^*) \times \RR^n)} \les \sup_{T < T^*} a_s(T) \les 2 \big( \| f \|_{\dot{H}^s} + \| g \|_{\dot{H}^s}\big).$$
If we apply this with $s>\frac{n}{2}$, then the classical local existence theory shows that the solution $(u, v)$ exists globally in time, i.e we may take $T^* = \infty$.

It only remains to show the Lipschitz bound. To this end, let $f', g' \in C^\infty_0(\RR^n)$ satisfy (\ref{eqn - thm gwp smooth data - smallness assump}) and let $(u', v')$ denote the corresponding solution. Another application of the cubic estimate in Corollary \ref{cor - cubic est} together with the bound (\ref{eqn - thm gwp for smooth data - a est}) shows that for any $T<\infty$
    $$ \| u_T - u'_T \|_{G^{\frac{n-1}{2}, +}} + \| v_T - v'_T \|_{G^{\frac{n-1}{2}, -}} \les \| f - f'\|_{\dot{H}^\frac{n-1}{2}} + \| g - g' \|_{\dot{H}^{\frac{n-1}{2}}} +  C \epsilon^2 \Big( \| u_T - u'_T \|_{G^{\frac{n-1}{2}, +}} + \| v_T - v'_T \|_{G^{\frac{n-1}{2}, -}}\Big). $$
Hence as $\epsilon>0$ is small, for any $T>0$ we obtain the Lipschitz  bound
    $$ \| u_T - u'_T \|_{G^{\frac{n-1}{2}, +}} + \| v_T - v'_T \|_{G^{\frac{n-1}{2}, -}} \les 2 \Big( \| f - f'\|_{\dot{H}^\frac{n-1}{2}} + \| g - g' \|_{\dot{H}^{\frac{n-1}{2}}} \Big). $$
Similarly, for higher regularities $s> \frac{n-1}{2}$, we can use a similar argument to show that
    \begin{align*}
      \| u_T - u'_T \|_{G^{s, +}} &+ \| v_T - v'_T \|_{G^{s, -}}\\
       &\lesa \| f  - f'\|_{\dot{H}^s} + \| g - g'\|_{\dot{H}^s} + C \epsilon^2 \Big( \| u_T - u'_T \|_{G^{s, +}} + \| v_T - v'_T \|_{G^{s, -}} \\
       & \qquad \qquad \qquad \qquad \qquad \qquad \qquad\qquad \qquad+ \| u_T - u'_T \|_{G^{\frac{n-1}{2}, +}} + \| v_T - v'_T \|_{G^{\frac{n-1}{2}, -}}\Big) \\
      &\lesa \| f  - f'\|_{\dot{H}^\frac{n-1}{2} \cap \dot{H}^s} + \| g - g'\|_{\dot{H}^\frac{n-1}{2} \cap \dot{H}^s} + C \epsilon^2 \Big( \| u_T - u'_T \|_{G^{s, +}} + \| v_T - v'_T \|_{G^{s, -}}\Big)
    \end{align*}
and hence
    \begin{equation}\label{eqn - thm gwp for uv - lipschitz est with Gs}
    \| u_T - u'_T \|_{G^{s, +}} + \| v_T - v'_T \|_{G^{s, -}} \les 2 \Big( \| f - f'\|_{\dot{H}^\frac{n-1}{2} \cap \dot{H}^s} + \| g - g' \|_{\dot{H}^\frac{n-1}{2} \cap \dot{H}^s} \Big).\end{equation}
Let $\rho \in C^\infty_0(\RR)$ and note that provided we choose $\delta$ sufficiently large,
        $$ \rho(\tfrac{t}{T}) (u, v) = \rho(\tfrac{t}{T}) (u_{\delta T}, v_{\delta T}). $$
Therefore by (ii) in Theorem \ref{thm - energy inequality + invariance under cutoffs} we have
   \begin{align*} \big\| \rho(\tfrac{t}{T}) \big( u - u'\big)  \big\|_{F^{s, +}} &+ \big\| \rho(\tfrac{t}{T}) \big( v - v'\big) \big\|_{F^{s, -}}\\
    &= \big\| \rho(\tfrac{t}{T}) \big( u_{\delta T} - u'_{\delta T} \big) \big\|_{F^{s, +}} + \big\| \rho(\tfrac{t}{T}) \big( v_{\delta T} - v'_{\delta T} \big) \big\|_{F^{s, -}} \\
            &\lesa \big\|  u_{\delta T} - u'_{\delta T} \big\|_{F^{s, +}} + \big\|  v_{\delta T} - v'_{\delta T} \big\|_{F^{s, -}}\\
            &\lesa \big\|  u_{\delta T} - u'_{\delta T} \big\|_{G^{s, +}} + \big\|  v_{\delta T} - v'_{\delta T} \big\|_{G^{s, -}}.
   \end{align*}
Thus the required Lipschitz bound follows from (\ref{eqn - thm gwp for uv - lipschitz est with Gs}) and noting that all constants are independent of $T>0$.
\end{proof}
\end{theorem}

The proof of Theorem \ref{thm - main thm with u, v} is now straightforward.

\begin{proof}[Proof of Theorem \ref{thm - main thm with u, v}]
Let $s \g \frac{n-1}{2}$ with $f, g \in \dot{H}^\frac{n-1}{2} \cap \dot{H}^s$ and
        $$ \| f \|_{\dot{H}^\frac{n-1}{2}} + \| g \|_{\dot{H}^\frac{n-1}{2}} < \epsilon$$
where $\epsilon$ is the constant in Theorem \ref{thm - gwp for smooth data}. By rescaling, we may assume that (\ref{eqn - thm gwp smooth data - smallness assump}) holds (if $s= \frac{n-1}{2}$ this is already true, if $s>\frac{n-1}{2}$ then the $\dot{H}^s$ norm is subcritical so we can rescale it to be small without changing the size of the data in $\dot{H}^\frac{n-1}{2}$). Choose a sequence $f_j, g_j \in C^\infty_0(\RR^n)$ satisfying (\ref{eqn - thm gwp smooth data - smallness assump}) such that $(f_j, g_j) \rightarrow (f, g)$ in $ \dot{H}^\frac{n-1}{2}\cap \dot{H}^s$ and let $(u_j, v_j)$ denote the corresponding solution given by Theorem \ref{thm - gwp for smooth data}. Let $\rho \in C^\infty_0(\RR)$ with $\rho = 1 $ on $[-1, 1]$. Then as
        $$ \| \phi \|_{L^\infty_t \dot{H}^b}\les \sup_{ T>0}  \big\| \rho( \tfrac{t}{T}) \phi \big\|_{L^\infty_t H^b_x} \les \sup_T \big\| \rho(\tfrac{t}{T}) \phi \big\|_{F^{b, \pm}}$$
the Lipschitz bound in Theorem \ref{thm - gwp for smooth data} shows that $(u_j, v_j)$ is a Cauchy sequence in $L^\infty_t \dot{H}^\frac{n-1}{2} \cap L^\infty_t \dot{H}^s$ and hence converges to a solution $(u, v) \in C( \RR, \dot{H}^\frac{n-1}{2} \cap \dot{H}^s)$. Moreover, for every $T>0$,  $\rho(\tfrac{t}{T}) (u_j, v_j)$ forms a Cauchy sequence in $F^{b, +} \times F^{b, -}$ for $b= \frac{n-1}{2}, s$. Consequently we must have  $\rho(\tfrac{t}{T}) (u, v) \in F^{b, +} \times F^{b, -}$ with
        $$ \sup_{T > 0} \Big( \big\| \rho(\tfrac{t}{T}) u \big\|_{F^{b, -}} +  \big\| \rho( \tfrac{t}{T}) v \big\|_{F^{b, -}} \Big) \lesa \| f \|_{\dot{H}^b} + \| g \|_{\dot{H}^b}.$$
Therefore, by (iii) in Theorem \ref{thm - energy inequality + invariance under cutoffs}, we see that $(u, v)$ scatters to a homogeneous solution in $\dot{H}^\frac{n-1}{2} \cap \dot{H}^s$ as required. Thus Theorem \ref{thm - gwp for smooth data} follows.
\end{proof}

\begin{remark}\label{rem on positive mass}
  If we have positive mass $m>0$, then we can prove local existence up to times $T \ll m^{-1}$ essentially by just treating the mass term as an additional perturbation. To see this, note that by rescaling, we may assume that $m=1$. Then instead of (\ref{eqn - thm gwp for smooth data - a cubic est}) we would have
    $$ a_s(T) \les \| f \|_{\dot{H}^s} + \| g \|_{\dot{H}^s} + C \Big(  T \big(\| u \|_{L^\infty_t \dot{H}^s} + \| v \|_{L^\infty_t \dot{H}^s}\big) + T^{\frac{3n-1}{4n}} \big( \| u \|_{L^\infty_t \dot{H}^{s - \frac{n+1}{4n}}} + \| v \|_{L^\infty_t \dot{H}^{s - \frac{n+1}{4n}}} \big) +  \epsilon^2 a_s(T) \Big) .$$
  If we now note that
     $$ \| \phi \|_{\dot{H}^{s - \frac{n+1}{4n}}} \lesa \| \phi \|_{L^2_x} + \| \phi \|_{\dot{H}^s} $$
  and use the fact that the charge (i.e. the $L^2_x$ norm) is conserved\footnote{Strictly speaking, the charge is not necessarily conserved for the general system as written in (\ref{eqn - general form of eqn}). However, charge \emph{is} conserved for the original system (\ref{eqn - NLD eqn}). Thus for the versions of (\ref{eqn - general form of eqn}) we are interested in, the charge is certainly conserved. }, then provided $T \ll 1$ we obtain
    $$ a_s(T) \les 2 \big( \| f \|_{\dot{H}^s} + \| g \|_{\dot{H}^s} \big) + \| f \|_{L^2_x} + \| g \|_{L^2_x}.$$
  We can follow a similar minor modification of the remainder of Theorem \ref{thm - gwp for smooth data} to deduce an equivalent result with the restriction $T \ll 1$, which after undoing the scaling, corresponds to $T \ll m^{-1}$.
\end{remark}

%------------------------------------------------------------------------------------------------------------------------------%

%------------------------------------------------------------------------------------------------------------------------------%

%------------------------------------------------------------------------------------------------------------------------------%
%------------------------------------------------------------------------------------------------------------------------------%
\section{Null Frame Bounds}\label{sec - proof of null frame bounds}
%------------------------------------------------------------------------------------------------------------------------------%
%------------------------------------------------------------------------------------------------------------------------------%

The proof of the null frame bounds is based on a transference type argument to reduce to the linear case. For the $L^1_t L^2_x$ and $\dot{\mc{X}}^{ - \frac{1}{2}, 1}$ components of our iteration space, this is not so difficult. On the other hand the $NF^+_\lambda$ case is more challenging and requires some theory on how the Dirac equation behaves in null coordinates. In particular we rely on a version of the Duhamel formula for the Dirac equation in null coordinates. The results in this section are based on related arguments in the work of Tataru \cite{Tataru2001} and  Tao \cite{Tao2001a}.

%------------------------------------------------------------------------------------------------------------------------------%
\subsection{Preliminary Results on Null Frames}\label{subsec - prelim results on null frames}
%------------------------------------------------------------------------------------------------------------------------------%

We start with a number of results on the geometry of null frames. These results are more or less implicit in \cite{Tataru2001, Tao2001a}, but the readers convenience, we include the statements and proofs here. \\

For a set $A \subset \RR^{n+1}$ we define $\mc{P}r_\omega ( A) = A + (1, \omega) \RR$ to be the projection along the null direction $(1, \omega)$. Note that the projected sets $\mc{P}r_\omega (A)$ depend only on the $\xi_\omega$ coordinate. More precisely, since $\ind_{\mc{P}r_\omega(A)}(\tau, \xi) = \ind_{\mc{P}r_\omega(A)}(0, \xi - \tau \omega)$ we have
        \begin{equation}
          \label{eqn - ind functions and null projections} \big( \ind_{\mc{P}r_\omega(A)}\big)^*(\tau_\omega, \xi_\omega) = \ind_{\mc{P}r_\omega(A)}(0, \xi_\omega).
        \end{equation}
Moreover, we have the following.

\begin{lemma}[Geometric Properties of Null Projections]\label{lem - geom prop}
\leavevmode
\begin{enumerate}
\item Let $\alpha, \beta \ll 1$, $\kappa \in \mc{C}_\alpha$, $\bar{\kappa} \in \mc{C}_\beta$, and\footnote{An important point here is that we can take \emph{any} $\bar{\kappa} \in \mc{C}_\beta$, there is no condition needed on $\omega$ with respect to $\bar{\kappa}$.} $\omega \not \in 2 \kappa$. Then
         $$ \{ |\tau| = |\xi| \} \bigcap \mc{P}r_\omega\big[ {^\natural A}_{\alpha, \lambda}(\kappa) \cap {^\natural A}_{\beta, \lambda}(\bar{\kappa})\big] \,\subset\, \Big[ {^\natural A}_\lambda(\tfrac{3}{2} \kappa)\cap {^\natural A}_\lambda(\tfrac{3}{2} \bar{\kappa})\Big].$$
\item Let $\alpha \ll 1$ and $ \omega \in \sph^{n-1}$.  Let $J \subset \mc{C}_\alpha$ be a collection of caps with $\omega \not \in 2 \kappa$ for every $\kappa \in J$. Then  the sets $\Omega_\alpha(\kappa)=\mc{P}r_\omega \big[{^\natural A}_{\alpha, \lambda}(\kappa)\big]$ have finite overlap in the sense that
        $$ \sum_{\kappa \in J} \ind_{\Omega_\alpha(\kappa)}(t, x) \lesa 1.$$
\end{enumerate}

\begin{proof}\textbf{(i):} Let $(\tau, \xi ) \in \{|\tau| = |\xi| \} \cap \mc{P}r_\omega\big[{^\natural A}_{\alpha, \lambda}(\kappa)\cap {^\natural A}_{\beta, \lambda}(\bar{\kappa})\big]$, in other words there exists $a \in \RR$ such that
    $$(\tau, \xi) - (a, a \omega) \in {^\natural A}_{\alpha, \lambda}(\kappa)\cap{^\natural A}_{\beta, \lambda}(\bar{\kappa}), \qquad |\tau| = |\xi|.$$
Let $\xi^\bot$ denote the component of $\xi$ orthogonal to $\omega$. Then since
    $$|\xi^\bot|^2 + (\tau - \xi \cdot \omega)^2 = |\xi|^2 + \tau^2 - 2 \tau \xi \cdot \omega = 2\sgn(\tau) |\xi| ( \tau - \xi \cdot \omega)$$
together with the fact that $|\xi^\bot| = |(\xi - a \omega)^\bot|  \approx \lambda \theta(\omega, \kappa) \theta(\omega, -\kappa)$ and the estimate
    \begin{equation}\label{eqn - lem geom prop - phi est} \tau  - \xi \cdot \omega = \tau - a - \sgn(\tau - a) |\xi - a \omega| + \sgn(\tau - a)|\xi - a \omega| - \omega \cdot (\xi - a \omega) \approx \lambda \theta(\omega, \kappa)^2\end{equation}
we obtain $|\xi| \approx \lambda$.

It remains to show that $\sgn(\tau) \frac{\xi}{|\xi|} \in \tfrac{3}{2}\kappa \cap \tfrac{3}{2} \bar{\kappa}$. To this end,  by noting that $(\tau - a)^2 - |\xi - a\omega|^2  = - 2 a (\tau - \xi\cdot \omega)$ we get
         $$|a| = \frac{ \big| |\tau - a|^2 - |\xi - a \omega|^2 \big|}{ 2|\tau - \xi \cdot \omega|} \approx \frac{\big| |\tau - a | - |\xi - a \omega| \big|}{\theta(\omega, \kappa)^2} $$
Consequently, using the assumption $\tau = \sgn(\tau) |\xi|$, we have
    \begin{align}
       \left| \sgn(\tau) \frac{\xi}{|\xi|} - \sgn(\tau - a \omega) \frac{\xi - a \omega}{|\xi - a \omega}\right|^2 &=  2 -2 \frac{\sgn(\tau) \xi}{|\xi|} \cdot \frac{\sgn(\tau - a \omega) (\xi-a \omega)}{|\xi - a \omega|} \notag\\
      &= \frac{2}{\tau} \big( \tau - a  - \sgn(\tau - a \omega) |\xi - a \omega| \big) + \frac{2a}{\tau} \left( 1 - \omega \cdot \frac{\xi - a \omega}{|\xi - a \omega|}\right) \notag\\
      &\lesa \frac{ \big| |\tau - a| - |\xi - a \omega| \big|}{\lambda} \les \frac{1}{100} (\min\{ \alpha, \beta\})^2 \label{eqn - lem geom prop - difference in angle xi and xi a}
    \end{align}
(provided we choose the close cone constant in the definition of ${^\natural A}_{\alpha, \lambda}(\kappa)$ to be sufficiently small). Let $\omega'$ denote the centre of $\kappa$. Since $\kappa \in \mc{C}_\alpha$, we have $\theta\big( \sgn(\tau-a) (\xi- a \omega), \omega'\big) \les \frac{101}{100}\alpha$. Thus from (\ref{eqn - lem geom prop - difference in angle xi and xi a}) we deduce that

    \begin{align*}
        \Big| \sgn(\tau) \frac{\xi}{|\xi|} - \omega'\Big| &\les \Big| \sgn(\tau) \frac{\xi}{|\xi|} - \sgn(\tau - a \omega) \frac{\xi - a \omega}{|\xi - a \omega}\Big| + \Big| \sgn(\tau - a \omega) \frac{\xi - a \omega}{|\xi - a \omega|}- \omega'\Big| \\
        &\les \frac{1}{10}\min\{ \alpha, \beta\} + \frac{101}{100} \alpha\les \frac{6}{5} \alpha.
    \end{align*}
  Therefore, using (\ref{eqn - sharp angle est}) we obtain
    $$ \theta\big( \sgn(\tau) \xi, \omega'\big) \les \frac{50}{49}  \Big( \frac{6}{5} \alpha\Big) = \frac{60}{49} \alpha < \frac{3}{2} \alpha$$
  and so $ \sgn(\tau) \xi  \in \frac{3}{2} \kappa$.  Similarly, since $(\tau, \xi) - a(1, \omega) \in {^\natural A}_\lambda(\bar{\kappa})$, then using $\bar{\omega}$ to denote the centre of $\bar{\kappa}$, we get
    \begin{align*} \theta\big(\sgn(\tau) \xi, \bar{\omega}\big) &\les \frac{50}{49} \Big| \sgn(\tau) \frac{\xi}{|\xi|} - \sgn(\tau - a \omega) \frac{\xi - a \omega}{|\xi - a \omega}\Big| + \frac{50}{49} \Big| \sgn(\tau - a \omega) \frac{\xi - a \omega}{|\xi - a \omega|}- \bar{\omega}\Big|\\
      &\les \frac{50}{49} \Big( \frac{1}{10} \min\{\alpha, \beta\} + \frac{101}{100} \beta\Big) \les \frac{3}{2} \beta
    \end{align*}
  as required.

\textbf{(ii):} It is enough to show that if $\kappa_1, \kappa_2 \in J$ with $\Omega_\alpha(\kappa_1) \cap \Omega_\alpha(\kappa_2) \not = \varnothing$, then $\theta(\kappa_1, \kappa_2) \lesa \alpha$. As if this holds, then the result follows by using the bounded overlap of the collection $\mc{C}_\alpha$. So suppose $(\tau, \xi ) \in \Omega_\alpha(\kappa_1) \cap \Omega_\alpha(\kappa_2)$. Note that for every $a \in \RR$
        $$(\tau, \xi ) \in \Omega_\alpha(\kappa_1) \cap \Omega_\alpha(\kappa_2) \qquad \Longleftrightarrow \qquad (\tau, \xi ) + a(1, \omega) \in \Omega_\alpha(\kappa_1) \cap \Omega_\alpha(\kappa_2)$$
   In particular, since $\tau - \omega \cdot \xi \not = 0$ by\footnote{Note that the derivation of (\ref{eqn - lem geom prop - phi est}) did not make use of the assumption $|\tau| = |\xi|$.} (\ref{eqn - lem geom prop - phi est}), we may take $a=\frac{ \tau^2 - |\xi|^2}{2 (\tau - \omega \cdot \xi)}$ and consequently
            $$ (\tau, \xi) + a(1, \omega) \in \Omega_\alpha(\kappa_1) \cap \Omega_\alpha(\kappa_2) \cap \{|\tau| = |\xi|\}.$$
   Therefore, by the first half of the lemma, we must have ${^\natural A}_\lambda(\tfrac{3}{2}\kappa_1) \cap {^\natural A}_{\lambda}(\tfrac{3}{2}\kappa_2) \not = \varnothing$, which by the finite overlap of the collection $\mc{C}_\alpha$ implies that $\theta(\kappa_1, \kappa_2) \lesa \alpha$ as required.

\end{proof}
\end{lemma}

The previous geometric lemma implies the following important orthogonality properties.

\begin{corollary}[Orthogonality in Null frames]\label{cor - orthog in null frame} Let $\alpha, \beta \ll 1$ and $\bar{\kappa} \in \mc{C}_\beta$.
    \begin{enumerate}
    \item Assume $\supp \widetilde{u}_{\kappa} \subset {^\natural A}_{\alpha, \lambda}(\kappa)$ for $\kappa \in \mc{C}_\alpha$. Then
            $$ \Big\| \sum_{\substack{\kappa \in \mc{C}_\alpha \\ {^\natural\kappa} \cap \bar{\kappa} \not = \varnothing}} u_{\kappa} \Big\|_{[NF^\pm]^*(\bar{\kappa})} \lesa\Bigg( \sum_{\substack{\kappa \in \mc{C}_\alpha \\ {^\natural\kappa} \cap \bar{\kappa} \not = \varnothing}} \| u_{\kappa} \|_{[NF^\pm]^*(\bar{\kappa})}^2\Bigg)^\frac{1}{2}$$

    \item Let $F \in NF^\pm(\bar{\kappa})$ with $\supp \widehat{F} \subset \{ |\xi| \approx \lambda\}$. Then
            $$  \Bigg(\sum_{\substack{\kappa \in \mc{C}_\alpha \\ \kappa \cap \bar{\kappa} \not = \varnothing}} \big\|  R^\pm_{\kappa, \alpha^2 \lambda} \Pi_{+} F \big\|_{NF^\pm(\bar{\kappa})}^2 + \big\|   R^\mp_{\kappa, \alpha^2 \lambda} \Pi_{-} F\big\|_{NF^\pm(\bar{\kappa})}^2 \Bigg)^\frac{1}{2} \lesa \| F \|_{NF^\pm(\bar{\kappa})}.$$

    \item Let  $\omega \in \sph^{n-1}$. Let $J \subset \mc{C}_\alpha$ be a collection of caps with $\omega \not \in 2 \kappa$  and $\theta(\omega, \kappa) \gtrsim \delta$ for every $\kappa \in J$. Then
            $$ \Bigg(\sum_{\kappa \in J} \big\|   P^{\pm, \alpha}_{\lambda, \kappa} \Pi_{+} F \big\|_{NF^\pm(\kappa)}^2 + \big\|   P^{\pm, \alpha}_{\lambda, \kappa} \Pi_{-} F\big\|_{NF^\pm(\kappa)}^2 \Bigg)^\frac{1}{2} \lesa \|\Pi_{\pm \omega} F \|_{L^1_{t_\omega} L^2_{x_\omega}} + \delta^{-1} \| \Pi_{\mp \omega} F \|_{L^1_{t_\omega} L^2_{x_\omega}}.$$

  \end{enumerate}
\begin{proof}\textbf{(i):} If $\alpha \gtrsim \beta$ the sum only contain $\mc{O}(1)$ terms and thus the inequality follows by the triangle inequality. It remains to consider the case $\alpha \ll \beta$. Let $\omega \not \in 2 \bar{\kappa}$. Then for every $\kappa \in \mc{C}_\alpha$ with  $\kappa \cap \bar{\kappa} \not = \varnothing$ we have $\omega \not \in 2\kappa$. Note that the $\xi_\omega$ support of $u_\kappa$ lies in the set $\Omega_\alpha(\kappa)=\mc{P}r_\omega \big[{^\natural A}_{\alpha, \lambda}(\kappa)\big]$. By Lemma \ref{lem - geom prop} these sets are essentially disjoint and thus we deduce that
    $$ \bigg\| \sum_{\substack{\kappa \in \mc{C}_\alpha \\ \kappa \cap \bar{\kappa} \not = \varnothing}} \widehat{u}^*(t_\omega, \xi_\omega) \bigg\|_{L^2_{\xi_\omega}}^2 \lesa \sum_{\substack{\kappa \in \mc{C}_\alpha \\ \kappa \cap \bar{\kappa} \not = \varnothing}}
                        \| \widehat{u}^*(t_\omega, \xi_\omega) \|_{L^2_{\xi_\omega}}^2.$$
By taking $L^\infty_{t_\omega}$ norms of both sides, inserting the relevant $\Pi_\omega$ projections, and then taking the sup over $\omega \not \in 2 \bar{\kappa}$ we obtain (i).

\textbf{(ii):} It is enough to consider the case $\pm = +$. As in the proof of (i), if $\alpha \gtrsim \beta$ then the sum only contains $\mc{O}(1)$ terms and so the required inequality follows by (i) in Lemma \ref{lem - mult are disposable}. Thus we may assume that $\alpha \ll \beta$. Let $\omega \not \in 2 \bar{\kappa}$. Note that this implies that $\omega \not \in 2 \kappa$ for every $\kappa \in \mc{C}_\alpha$ with $\kappa \cap \bar{\kappa} \not = \varnothing$. As in Lemma \ref{lem - geom prop}, we let $\Omega_{\alpha}(\kappa) = \mc{P}r_\omega \big[ A_{\alpha, \lambda}(\kappa)]$. Define $\widetilde{P^\omega_{\Omega(\kappa)} G} = \ind_{\Omega_\alpha(\kappa)}(\xi_\omega) \widetilde{G}$, clearly $R^\pm_{\kappa, \alpha^2 \lambda} F  = P^\omega_{\Omega_\alpha(\kappa)} R^\pm_{\kappa, \alpha^2 \lambda} F   $. Thus an application of Lemma \ref{lem - mult are disposable} to dispose of the $R^\pm_{\kappa, \alpha^2 \lambda}$ multiplier, followed by an application of Plancheral  gives
  \begin{align*}
 \sum_{\substack{\kappa \in \mc{C}_\alpha \\ \kappa \cap \bar{\kappa} \not = \varnothing}}
     \big\|  R^\pm_{\kappa, \alpha^2 \lambda} \Pi_\pm  F \big\|_{NF^+(\bar{\kappa})}^2
     &\lesa \sum_{\substack{\kappa \in \mc{C}_\alpha \\ \kappa \cap \bar{\kappa} \not = \varnothing}}
     \big\|  P^\omega_{\Omega(\kappa)} F \big\|_{NF^+(\bar{\kappa})}^2 \\
        &\lesa \sum_{\substack{\kappa \in \mc{C}_\alpha \\ \kappa \cap \bar{\kappa} \not = \varnothing}} \| P^\omega_{\Omega(\kappa)}\Pi_\omega  F \|_{L^1_{t_\omega} L^2_{x_\omega}}^2 + \theta(\omega, \bar{\kappa})^{-2}  \|  P^\omega_{\Omega(\kappa)}\Pi_{-\omega}  F\|_{L^1_{t_\omega} L^2_{x_\omega}}^2.
  \end{align*}
We now observe that, by the finite overlap of the sets $\Omega_\alpha(\kappa)$ in Lemma \ref{lem - geom prop} together with the identity (\ref{eqn - ind functions and null projections}), we have
    $$\sum_{\substack{\kappa \in \mc{C}_\alpha \\ \kappa \cap \bar{\kappa} \not = \varnothing}} \| P^\omega_{\Omega(\kappa)} \Pi_{\pm \omega} F \|_{L^1_{t_\omega} L^2_{x_\omega}}^2\les
        \bigg\|\Big( \int_{\RR^n} \sum_{\substack{\kappa \in \mc{C}_\alpha \\ \kappa \cap \bar{\kappa} \not = \varnothing}} \big| \ind_{\Omega_\alpha(\kappa)}(0, \xi_\omega) \widehat{\Pi_{\pm \omega}F}^*(t_\omega, \xi_\omega)\big|^2 d\xi_\omega \Big)^\frac{1}{2}\bigg\|_{L^1_{t_\omega}}^2 \lesa \| \Pi_{\pm \omega}F \|_{L^1_{t_\omega} L^2_{x_\omega}}^2.$$
Applying these inequalities to $NF^+(\bar{\kappa})$ atoms, we obtain (ii).

\textbf{(iii):} We follow a similar argument to (ii). The properties of the collection $J$ imply that the sets $\Omega_\alpha(\kappa)$ finitely overlap for $\kappa \in J$. Hence, after an application of Lemma \ref{lem - mult are disposable} to dispose of the multipliers $P^{\pm, \alpha}_{\lambda, \kappa}$, we obtain
    \begin{align*}
       \sum_{\kappa \in J} \big\|  P^{\pm, \alpha}_{\lambda, \kappa} \Pi_{+} F \big\|_{NF^\pm(\kappa)}^2 + \big\|  P^{\mp, \alpha}_{\lambda, \kappa} \Pi_{-} F\big\|_{NF^\pm(\kappa)}^2
            &\lesa  \sum_{\kappa \in J} \big\|  P^\omega_{\Omega(\kappa)} F \big\|_{NF^\pm(\kappa)}^2 \\
            &\lesa \sum_{\kappa \in J} \big\|  P^\omega_{\Omega(\kappa)} \Pi_{\pm \omega} F \big\|_{L^1_{t_\omega} L^2_{x_\omega}}^2 + \theta(\omega, \kappa)^{-2} \big\|  P^\omega_{\Omega(\kappa)} \Pi_{\mp \omega} F \big\|_{L^1_{t_\omega} L^2_{x_\omega}}^2\\
            &\lesa \big\| \Pi_{\pm \omega} F \big\|_{L^1_{t_\omega} L^2_{x_\omega}}^2 + \delta^{-2} \big\| \Pi_{\mp \omega} F \big\|_{L^1_{t_\omega} L^2_{x_\omega}}^2
    \end{align*}
as required.
\end{proof}
\end{corollary}

%-------------------------------------------------------------------------------------------------%

\subsection{The Dirac Equation in Null coordinates} We want to write the equation
       \begin{equation}\label{eqn - inhomogeneous dirac equation} ( \p_t \pm \sigma \cdot \nabla)  u  = F \end{equation}
in null coordinates $(t_\omega, x_\omega)$. A computation shows that
    \begin{align*}
      \sqrt{2} \p_{t_\omega} \Pi_{\pm\omega} u  \pm (\sigma \cdot \nabla^\bot_{x_\omega}) \Pi_{\mp\omega} u   &= \Pi_{\pm\omega} F\\
       - (\omega \cdot \nabla_{x_\omega}) \Pi_{\mp\omega} u  \pm (\sigma \cdot \nabla^\bot_{x_\omega} )\Pi_{\pm\omega} u &= \Pi_{\mp\omega} F.
    \end{align*}
Rearranging, and assuming that we can divide by $\omega \cdot \nabla_{x_\omega}$ (i.e. we assume that $\widetilde{u}$ and $\widetilde{F}$ are supported away from the null plane $ \omega \cdot \xi_\omega =0$ $\Leftrightarrow$ $ \tau = \omega \cdot \xi$), we see that $u$ satisfies
    \begin{equation}\label{eqn - inhomogeneous Dirac in null coordinates} \begin{split}
     \Big(\sqrt{2} \p_{t_\omega}  + \frac{ |\nabla^\bot_{x_\omega}|^2}{ \omega \cdot \nabla_{x_\omega}} \Big)\Pi_{\pm \omega} u &=   \Pi_{\pm\omega} F \pm \frac{ \sigma \cdot \nabla^\bot_{x_\omega}}{\omega \cdot \nabla_{x_\omega}} \Pi_{\mp\omega} F \\
      \Pi_{\mp\omega} u &=  \pm \frac{\sigma \cdot \nabla^\bot_{x_\omega}}{\omega \cdot \nabla_{x_\omega}} \Pi_{\pm\omega} u - \frac{1}{\omega \cdot \nabla_{x_\omega}} \Pi_{\mp\omega} F.
      \end{split}
      \end{equation}
The equation (\ref{eqn - inhomogeneous Dirac in null coordinates}) is interesting as it shows that, in null coordinates $(t_\omega, x_\omega)$, the $\Pi_{\mp \omega} u$ component of $u$ solves what is essentially an \emph{elliptic} equation. This observation is the motivation for building the projections $\Pi_{\pm \omega}$ into the definition of the null frame spaces $PW^\pm(\kappa)$ and $NF^\pm(\kappa)$, as it allows us to isolate the ``elliptic'' and dispersive components of the evolution.

The equation (\ref{eqn - inhomogeneous Dirac in null coordinates}) also shows that we can write the forward fundamental solution in null coordinates as
    $$ \big(E^\pm_{\omega} * F\big)^*(t_\omega, x_\omega) = \pm \frac{\sigma \cdot \nabla_{x_\omega}}{\sqrt{2} \omega \cdot \nabla_{x_\omega}}\int_{-\infty}^{t_\omega} e^{ - (t_\omega - a) \frac{ |\nabla^\bot_{x_\omega}|^2}{\sqrt{2} \omega \cdot \nabla_{x_\omega}}} \Big( \Pi_{\pm \omega} \pm \frac{\sigma\cdot \nabla^\bot_{x_\omega}}{\omega \cdot \nabla_{x_\omega}}\Pi_{\mp \omega} \Big) F^*(a)\, d a - \frac{1}{\omega \cdot \nabla_{x_\omega}} \Pi_{\mp \omega} F. $$
It is easy to check that $ u = E^\pm_\omega * F$ gives us a solution to (\ref{eqn - inhomogeneous dirac equation}), and moreover, that
        $$\Big(\frac{\sigma \cdot \nabla_{x_\omega}}{\sqrt{2} \omega \cdot \nabla_{x_\omega}} e^{ - t_\omega \frac{ |\nabla^\bot_{x_\omega}|^2}{\sqrt{2} \omega \cdot \nabla_{x_\omega}}} f\Big)^*(x_\omega) = \int_{\RR^n} \,\frac{\sigma \cdot \xi_\omega}{ - 2 \xi^1_\omega} \,e^{  i t_\omega \frac{ |\xi^\bot_\omega|^2}{2 \xi^1_\omega}}\, \widehat{f}^*(\xi_\omega) \,  e^{ i x_\omega \cdot \xi_\omega} \,d\xi_\omega$$
is a solution to the equation $(\p_t + \sigma \cdot \nabla) u= 0$ (lets assume that $f$ has support away from $\xi^1_\omega = 0$ for simplicity). The fundamental solution operator $E^\pm_\omega$ plays a crucial role in the proof of Theorem \ref{thm - null frame bounds} as it gives a suitable substitute for the missing transference type principle. In other words, using the following lemma, we are able to essentially deduce the required null frame bounds from their homogeneous counterparts.

\begin{lemma}[Decomposition of $E^\pm_\omega*F$ into free waves, cf. {\cite[Proposition 3.4]{Tataru2001}}]\label{lem - null frame decomp into free waves}
Let $\alpha,  \beta \ll 1$ and $\kappa \in \mc{C}_\alpha$, $\bar{\kappa} \in \mc{C}_\beta$. Fix $\omega \not \in 2 \kappa$. Assume $\supp \widetilde{F} \subset {^\natural A}_{\alpha, \lambda}(\kappa) \cap {^\natural A}_{\beta, \lambda}(\bar{\kappa})$. Then we can write
        $$ E^\pm_{\omega} * F = \int_{-\infty}^{t_\omega} \psi_{a}\, da + \Pi_{\mp \omega} G$$
  where $\psi_{a}$ satisfies $(\p_t \pm \sigma \cdot \nabla) \psi_{a} = 0$ and
        $$\supp   \widehat{\Pi_\pm\psi}_{a} \subset {^\natural A}^\pm_\lambda(\tfrac{3}{2} \kappa) \cap {^\natural A}^\pm_{\lambda}(\tfrac{3}{2} \bar{\kappa}), \qquad \supp   \widehat{\Pi_\mp\psi}_{a} \subset {^\natural A}^\mp_\lambda(\tfrac{3}{2} \kappa) \cap {^\natural A}^\mp_{\lambda}(\tfrac{3}{2} \bar{\kappa}),$$
  and $G \in L^2_{t, x}$ with $\supp \widetilde{G} = \supp \widetilde{F}$. Moreover we have the bound
    \begin{equation} \label{eqn - lem fund soln decomp - bound on decomp}\int_{\RR} \| \psi_{a} \|_{L^\infty_t L^2_x} d a + \| G \|_{\dot{\mc{X}}^{\frac{1}{2}, 1}_\pm}
                \lesa \| \Pi_{\pm \omega} F \|_{L^1_{t_\omega} L^2_{x_\omega}} + \theta(\omega, \kappa)^{-1} \| \Pi_{\mp \omega} F \|_{L^1_{t_\omega} L^2_{x_\omega}}.\end{equation}
  \begin{proof}
    If we let
            $$ (\psi_{a})^*(t_\omega, x_\omega) = \pm \frac{\sigma \cdot \nabla_{x_\omega}}{\sqrt{2} \omega \cdot \nabla_{x_\omega}} e^{ - (t_\omega - a) \frac{ |\nabla^\bot_{x_\omega}|^2}{\sqrt{2} \omega \cdot \nabla_{x_\omega}}} \Big( \Pi_{\pm \omega} \pm \frac{\sigma\cdot \nabla^\bot_{x_\omega}}{\omega \cdot \nabla_{x_\omega}}\Pi_{\mp \omega} \Big) F^*(a)$$
    and $ G = - \frac{1}{\omega \cdot \nabla_{x_\omega}} \Pi_{\mp \omega}F$, then by definition of $E^\pm_\omega$, we have a decomposition
                $$ E^\pm_\omega * F = \int_{-\infty}^{t_\omega} \psi_{a} \,d a + G. $$
    A calculation using (\ref{eqn - inhomogeneous Dirac in null coordinates}) shows that $(\p_t \pm \sigma \cdot \nabla) \psi_{a}=0$ and thus $\widetilde{\psi}_a$ is supported on the light cone. On the other had, if $(\tau, \xi) \in \supp \widetilde{\psi}_a$, then
        $$\xi - \tau \omega \in \bigcup_{a \in \RR} \supp \widehat{\psi}^*(a) \subset \bigcup_{a \in \RR} \supp \widehat{F}^*(a)$$
    and so $(\tau, \omega) \in \mathcal{P}r_\omega\big( \supp \widetilde{F} \big)$. Consequently the claim on the support of $\psi_a$ follows from Lemma \ref{lem - geom prop}. Thus it only remains to prove the bound (\ref{eqn - lem fund soln decomp - bound on decomp}). To this end, note that
        \begin{align*}
         \big[ e^{ - t_\omega \frac{|\nabla_{x_\omega}^\bot|}{\sqrt{2}\omega \cdot \nabla_{x_\omega}}} f^*\big](t, x) &= \int_{\RR^n} \widehat{f}^*(\xi_\omega) e^{ - i(t + x \cdot \omega) \frac{|\xi^\bot_\omega|^2}{ 2\omega \cdot \xi_\omega}} e^{ i ( x - \frac{1}{2}( t + x \cdot \omega) \omega) \cdot \xi_\omega} d \xi_\omega \\
                                    &= \int_{\RR^n} \widehat{f}^*(\xi_\omega) e^{ - it \frac{|\xi_\omega|^2}{ 2\omega \cdot \xi_\omega}} e^{ i x \cdot( \xi_\omega - \frac{|\xi_\omega|^2}{2\omega \cdot \xi_\omega}\omega) } d \xi_\omega \\
                                    &= \int_{\RR^n} \Big[\widehat{f}^*(\xi_\omega) e^{ - it \frac{|\xi_\omega|^2}{ 2\omega \cdot \xi_\omega}} J^{-1}(\xi_\omega) \Big](y) e^{ i x \cdot y} dy
        \end{align*}
    where $dy = J(\xi_\omega) d\xi_\omega$ and\footnote{One way to see this is to note that we are only changing $\xi_\omega$ in the $\omega$ direction, thus if we let $a=\omega \cdot \xi_\omega$ and $b=|\xi_\omega|$, then we are effectively computing the Jacobian for the change of variables $a'= a - \frac{a^2 + b^2}{2a}$ which is $\frac{1}{2} + \frac{b^2}{2a^2} = \frac{a^2 + b^2}{2a^2} =\frac{1}{2} \big( \frac{|\xi_\omega|}{\omega \cdot \xi_\omega}\big)^2$. } the Jacobian  is given by $ J(\xi_\omega) = \frac{1}{2} \big(\frac{|\xi_\omega|}{\omega \cdot \xi_\omega}\big)^2$. Thus by an application of Plancheral we get
        \begin{align*}
          \Big\| \Big[e^{ - t_\omega \frac{|\nabla_{x_\omega}^\bot|}{\sqrt{2}\omega \cdot \nabla_{x_\omega}}} f^*\Big](t, x) \Big\|_{L^2_x}
          &=\Big\| \Big[\widehat{f}^*(\xi_\omega) e^{ - it \frac{|\xi_\omega|^2}{ 2\omega \cdot \xi_\omega}} J^{-1}(\xi_\omega) \Big](y) \Big\|_{L^2_y} \\
          &= 2 \Big\| \frac{(\omega \cdot \xi_\omega)}{|\xi_\omega|} f^*(\xi_\omega) \Big\|_{L^2_{\xi_\omega}}.
        \end{align*}
If we now observe that $\supp \widehat{F}^*(a) \subset \big\{ |\xi_\omega \cdot \omega| \approx \theta(\omega, \kappa)^2 \lambda,\,\, |\xi^\bot_\omega| \lesa \theta(\omega, \kappa) \lambda \, \big\}$ we obtain
         \begin{align*} \| \psi_{a} \|_{L^\infty_t L^2_x} &\approx \Big\| \frac{\omega \cdot \xi_\omega}{|\xi_\omega|} \frac{\sigma \cdot \xi_\omega}{\omega \cdot \xi_\omega}    \Big( \Pi_{\pm \omega} \pm \frac{\sigma\cdot \xi^\bot_\omega}{\omega \cdot \xi_{\omega}}\Pi_{\mp \omega} \Big) \widehat{F^*}(a, \xi_\omega) \Big\|_{L^2_{\xi_\omega}} \\
                &\lesa \| \Pi_{\pm \omega} F^*(a) \|_{L^2_{x_\omega}} + \theta(\omega, \kappa)^{-1} \| \Pi_{\mp \omega} F^*(a) \|_{L^2_{\xi_\omega}}.\end{align*}
    Integrating over $a$ then controls the $\psi_{a}$ component. To estimate $G$,  we write $G= \frac{1}{\omega \cdot \nabla_{x_\omega}} \Pi_{-\omega} \Pi_+  F + \frac{1}{\omega \cdot \nabla_{x_\omega}} \Pi_{-\omega} \Pi_- F$. Note that for  $(\tau, \xi) \in \supp \widetilde{\Pi_\pm F } \subset {^\natural A}^\pm_{\lambda, \alpha}(\kappa)$, we have $|\omega \cdot \xi_\omega| \approx |\xi^1_\omega| \approx \lambda \theta(\omega, \kappa)^2$ as well as the null form estimate $\Pi_{\mp \frac{\xi}{|\xi|}} \Pi_{-\omega} \lesa \theta(\omega, \kappa)$. Therefore, using the bound (\ref{eqn - lem L2 bound on NF atom - L2 bounded by L1L2 in null coord}) and the fact that $\alpha\lesa \theta(\omega, \kappa)$, we have
    \begin{align*}
      \Big\| \tfrac{1}{\omega \cdot \nabla_{x_\omega}} \Pi_{-\omega} \Pi_\pm F \Big\|_{\dot{\mc{X}}^{\frac{1}{2}, 1}_+} &\lesa \frac{1}{\lambda \theta(\omega, \kappa)^2} \bigg( \sum_{d \lesa \alpha^2 \lambda} d^\frac{1}{2} \big\| C^\pm_d \Pi_\pm \Pi_{-\omega} \Pi_\pm F \big\|_{L^2_{t, x}} + \sum_{d \lesa \lambda} d^\frac{1}{2} \big\| C^\mp_d \Pi_{\mp} \Pi_{-\omega} \Pi_{\pm} F \big\|_{L^2_{t, x}}\bigg) \\
      &\lesa \frac{\lambda^\frac{1}{2} \alpha + \lambda^\frac{1}{2} \theta(\omega, \kappa)}{\lambda \theta(\omega, \kappa)^2} \| \Pi_\pm F \|_{L^2_x} \lesa \| \Pi_\omega F \|_{L^1_{t_\omega} L^2_{x_\omega}} + \theta(\omega, \kappa)^{-1} \| \Pi_{-\omega} F \|_{L^1_{t_\omega} L^2_{x_\omega}}
    \end{align*}
    as required.
  \end{proof}

\end{lemma}

\begin{remark}
  The bound for the ``elliptic'' term $G = \tfrac{1}{\omega \cdot \nabla_{x_\omega}} \Pi_{\mp \omega} F$ can be improved somewhat. For instance, by a similar argument, we could replace $\| G \|_{\dot{\mc{X}}^{\frac{1}{2}, 1}_\pm}$ with the larger $\lambda^{\frac{1}{2}} \frac{ \theta(\omega, \kappa)^2}{\min\{ \alpha, \beta\}} \| \Pi_{\mp \omega} G \|_{L^2_{t, x}}$. However, the use of the $\dot{\mc{X}}^{\frac{1}{2}, 1}_\pm$ norm is technically convenient, and slightly simpler to state.
\end{remark}

As an application of the previous result, we obtain control the solution $E^\pm_\omega * F$ in $L^\infty_t L^2_x$.

\begin{corollary}[$L^\infty_t L^2_x$ Control of Fundamental solution in Null Frames]\label{cor - LinftyL2 control of null fund soln}
  Let $\alpha \ll 1$ and $\kappa \in \mc{C}_\alpha$. Assume $\supp \widetilde{F} \subset {^\natural A}_{\alpha, \lambda}(\kappa) $ and $\omega \not \in 2 \kappa$. Then
        $$ \| E^\pm_{\omega} * F \|_{L^\infty_t L^2_x} \lesa \big\| \Pi_{\pm \omega} F \big\|_{L^1_{t_\omega} L^2_{x_\omega}} + \theta(\omega, \kappa)^{-1} \big\| \Pi_{\mp \omega} F \big\|_{L^1_{t_\omega} L^2_{x_\omega}}.$$
\begin{proof}
  By Lemma \ref{lem - null frame decomp into free waves}, we can write
        $$ E^\pm_\omega * F = \int_{-\infty}^{t_\omega} \psi_{a} da + G$$
  where $\psi_{a}$, $G$ are as in the statement of the Lemma. Thus by an application of Minkowsi's inequality and (\ref{eqn - Xsb controls strichartz}) we have
    \begin{align*}
      \big\| E^\pm_\omega *F \big\|_{L^\infty_t L^2_x} \les \int_\RR \| \psi_{a} \|_{L^\infty_t L^2_x} da + \| G \|_{L^\infty_t L^2_x}
                \lesa \int_\RR \| \psi_{a} \|_{L^\infty_t L^2_x} da + \| G \|_{\dot{\mc{X}}^{\frac{1}{2}, 1}_\pm}.
    \end{align*}
  Hence result follows by Lemma \ref{lem - null frame decomp into free waves}.
\end{proof}
\end{corollary}

We also have the following crucial energy type inequality.

\begin{corollary}\label{cor - energy type est}
Suppose $F \in \mc{N}^\pm_\lambda$ and  $(\p_t \pm \sigma \cdot \nabla) u = F$. Then
        $$ \|  u \|_{L^\infty_t L^2_x} \les \|  u(0) \|_{L^2_x}  +  C \| F \|_{\mc{N}^\pm_\lambda}$$
for some constant $C$ (independent of $u$ and $F$).
\begin{proof}
It is enough to consider the case $\pm = +$.  By writing $ u =  u - \mc{U}_+(t)\big[u(0)\big] + \mc{U}_+(t)\big[u(0)\big]$ and using the homogeneous energy estimate, we reduce to the case $u(0)=0$. By definition of $\mc{N}^+_\lambda$, we reduce to considering the case where $F$ is an $L^1_t L^2_x$ atom, a $\dot{\mc{X}}^{-\frac{1}{2}, 1}_+$ atom, or a $NF^+_\lambda$ atom.

The case $F \in L^1_t L^2_x$ is immediate by the standard energy inequality. On the other hand, if $F$ is an $\dot{\mc{X}}^{-\frac{1}{2}, 1}_+$ atom, then we write
        $$   \Pi_\pm u =  (\p_t \pm i |\nabla| )^{-1} \Pi_\pm F - e^{\mp i t |\nabla|} \big[ (\p_t \pm i |\nabla| )^{-1} \Pi_\pm F(0)\big].$$
  Then
    $$  \big\|  \Pi_\pm u \big\|_{L^\infty_t L^2_x} \lesa  \big\| ( \p_t \pm i |\nabla|)^{-1}  \Pi_\pm F \big\|_{L^\infty_t L^2_x} \\
                                                            \lesa d^{\frac{1}{2}} \big\| ( \p_t \pm i |\nabla|)^{-1}  \Pi_\pm F \big\|_{L^2_{t, x}} \approx d^{-\frac{1}{2}} \| F \|_{L^2_{t, x}} \\
                                                            \lesa 1
    $$
  as required.

Finally, if $F$ is a $NF^+_\lambda$ atom, then there exists a decomposition $F = \sum_{\kappa \in \mc{C}_\alpha} F_\kappa$ with $ \supp \Pi_\pm F_{\kappa} \subset A^\pm_{\lambda, \alpha}(\kappa)$. Let $\omega \not \in 2 \kappa$. Then by Corollary \ref{cor - LinftyL2 control of null fund soln} we obtain
    \begin{align*}
      \Big\| \int_0^t \mc{U}_+(t-s) F_\kappa(s) ds \Big\|_{L^\infty_t L^2_x} &= \big\| E^+_\omega * F_\kappa - \mc{U}_+(t)\big[ E^+_\omega *F_\kappa(0)\big] \big\|_{L^\infty_t L^2_x} \\
      &\lesa \| E^+_\omega * F_\kappa \|_{L^\infty_t L^2_x} \\
      &\lesa \| \Pi_\omega F_\kappa \|_{L^1_{t_\omega}  L^2_{x_\omega}} + \theta(\omega, \kappa)^{-1} \| \Pi_{-\omega} F_\kappa \|_{L^1_{t_\omega}  L^2_{x_\omega}}.
    \end{align*}
If we apply this estimate to $NF^+(\kappa)$ atoms (and using an application of Lemma \ref{lem - mult are disposable} to retain the support properties) we deduce that
    $$ \Big\| \int_0^t \mc{U}_+(t-s) F_\kappa(s) ds \Big\|_{L^\infty_t L^2_x} \lesa \|F_{\kappa} \|_{NF^+(\kappa)}.$$
Thus, as the $F_\kappa$ are essentially orthogonal in $L^2_x$, we have
$$ \| u \|_{L^\infty_t L^2_x} \les \Big( \sum_{\kappa \in \mc{C}_\alpha} \Big\| \int_0^t \mc{U}_+(t-s) F_\kappa(s) ds \Big\|_{L^\infty_t L^2_x}^2 \Big)^\frac{1}{2} \lesa \Big( \sum_{\kappa \in \mc{C}_\alpha} \| F_\kappa \|_{NF^+(\kappa)}^2 \Big)^\frac{1}{2}$$
as required.
\end{proof}
\end{corollary}

%------------------------------------------------------------------------------------------------------------------------------%
\subsection{Null Frame Bounds - The Homogeneous Case}\label{subsec - null frame bounds hom case}
%------------------------------------------------------------------------------------------------------------------------------%

In this section we prove a number of preliminary bounds that are used in the proof of Theorem \ref{thm - null frame bounds}. We start by proving the following preliminary estimate.

\begin{proposition}[$\dot{\mc{X}}^{\frac{1}{2}, 1}_\pm$ controls $PW^-(\kappa)$]\label{prop - Xsb controls PW}
Let $\beta \les \alpha$, $\kappa \in \mc{C}_\alpha$, $\bar{ \kappa } \in \mc{C}_\beta$, and $2\kappa \cap 2\bar{\kappa} \not = \varnothing$. For every $s \in \RR$, let $b_s \in L^\infty_{t, x}$ be a scalar valued function, and let $\psi_s \in L^2_{t, x}$ with the support conditions
     $$\supp  \widehat{\Pi_+\psi}_s \subset A^\pm_\lambda( 2 \bar{\kappa}), \qquad \supp \widehat{\Pi_- \psi}_s \subset A^\mp_\lambda(2 \bar{\kappa}) . $$
Then
        $$\bigg\| \int_{\RR} b_s(t, x) \psi_s(t, x) \, ds\, \bigg\|_{PW^\mp(\kappa)} \lesa (\beta \lambda)^{\frac{n-1}{2}} \int_{\RR} \big\| b_s \big\|_{L^\infty_{t, x}} \big\| \psi_s \big\|_{\dot{\mc{X}}^{\frac{1}{2}, 1}_\pm} \, ds .$$
\begin{proof}
  We only prove the case $\pm = +$, the remain case follows by a reflection in $x$. The assumption $\psi_s \in L^2_{t, x}$ implies that $\psi_s = \sum_d C^\pm_d \psi_s$ and so after an application of Holder is is enough to prove
        $$\bigg\| \int_{\RR} b_s(t, x) \Pi_\pm C^\pm_d \psi_s(t, x) \, ds\, \bigg\|_{PW^-(\kappa)} \lesa (\beta \lambda)^{\frac{n-1}{2}} d^\frac{1}{2} \int_{\RR} \big\| b_s \big\|_{L^\infty_{t, x}} \big\| \Pi_\pm C^\pm_d \psi_s \big\|_{L^2_{t, x}} \, ds .$$
  If $d \gtrsim \beta^2 \lambda$, then the support assumptions on $\psi_s$ together with (\ref{eqn - estimate on dual coordinates no supp restrict}) imply that for every $\omega \in 2\bar{\kappa}$, $|\xi^\bot_\omega| \lesa \lambda \beta$ and $|\xi^1_\omega| \lesa d$. Hence, by taking $\omega \in 2 \bar{\kappa} \cap 2 \kappa$ and using the null form estimate $|\Pi_{-\omega} \Pi_{\frac{\xi}{|\xi|}}| \lesa \theta(\omega, \xi)$,  the definition of the norm $\| \cdot \|_{PW(\kappa)}$ gives
    \begin{align*}
      \bigg\| \int_{\RR} b_s(t, x) &\Pi_\pm C^\pm_d \psi_s(t, x) \, ds\, \bigg\|_{PW^-(\kappa)}  \\
      &\les \bigg\| \int_{\RR} b_s(t, x) \Pi_{-\omega}\Pi_\pm C^\pm_d \psi_s(t, x) \, ds\, \bigg\|_{L^2_{t_\omega} L^\infty_{x_\omega}} + \alpha^{-1} \bigg\| \int_{\RR} b_s(t, x) \Pi_{\omega}\Pi_\pm C^\pm_d \psi_s(t, x) \, ds\, \bigg\|_{L^2_{t_\omega} L^\infty_{x_\omega}} \\
      &\lesa \int_\RR \| b_s \|_{L^\infty_{t, x}} \Big( \big\| \Pi_\pm C^\pm_d \psi_s \big\|_{L^2_{t_\omega} L^\infty_{x_\omega}} + \alpha^{-1} \big\| \Pi_{- \omega} \Pi_\pm C^\pm_d \psi_s \big\|_{L^2_{t_\omega} L^\infty_{x_\omega}} \Big) \, ds \\
      &\lesa (\beta \lambda)^{\frac{n-1}{2}} d^\frac{1}{2} \int_\RR \| b_s \|_{L^\infty_{t, x}}  \big\| \big( 1 + \tfrac{\theta(\omega, \mp \xi)}{\alpha}\big)  \widetilde{\Pi_\pm C^\pm_d \psi_s} \big\|_{L^2_{\tau, \xi}}  \, ds \\
      &\lesa (\beta \lambda)^{\frac{n-1}{2}} d^\frac{1}{2} \int_\RR \| b_s \|_{L^\infty_{t, x}}  \big\| \Pi_\pm C^\pm_d \psi_s \big\|_{L^2_{t, x}}  \, ds
    \end{align*}
as required. On the other hand, if $d \ll \beta^2 \lambda$, then we decompose $ \Pi_\pm \psi_s = \sum_{ \substack{ \kappa' \in \mc{C}_{\beta'} \\ \kappa' \cap 2\bar{\kappa} \not = \varnothing}} \Pi_\pm R^\pm_{ \kappa'} \psi_s$ where $\beta' =  \sqrt{\frac{\lambda}{d}}$. Now as $d \gtrsim (\beta')^2 \lambda$ we can repeat the previous argument (with $\beta$ and $\bar{\kappa}$ replaced with $\beta'$ and $\kappa'$) together with the orthogonality of the projections $R^\pm_{ \kappa'}$ in $L^2_{t, x}$ to obtain
    \begin{align*}
      \bigg\| \int_{\RR} b_s(t, x) \Pi_\pm C^\pm_d \psi_s(t, x) \, ds\, \bigg\|_{PW^-(\kappa)}  &\les \sum_{ \substack{ \kappa' \in \mc{C}_{\beta'} \\ \kappa' \cap 2\bar{\kappa} \not = \varnothing}} \bigg\| \int_{\RR} b_s(t, x) \Pi_\pm R^\pm_{ \kappa'} C^\pm_d \psi_s(t, x) \, ds\, \bigg\|_{PW^-(\kappa)} \\
      &\lesa ( \beta' \lambda)^{\frac{n-1}{2}} d^\frac{1}{2} \sum_{ \substack{ \kappa' \in \mc{C}_{\beta'} \\ \kappa' \cap 2\bar{\kappa} \not = \varnothing}} \int_\RR \| b_s \|_{L^\infty_{t, x}}  \big\| \Pi_\pm R^\pm_{ \kappa'} C^\pm_d \psi_s \big\|_{L^2_{t, x}}  \, ds\\
      &\lesa ( \beta' \lambda)^{\frac{n-1}{2}} d^\frac{1}{2} \bigg( \sum_{ \substack{ \kappa' \in \mc{C}_{\beta'} \\ \kappa' \cap 2\bar{\kappa} \not = \varnothing}}\bigg)^\frac{1}{2}  \int_\RR \| b_s \|_{L^\infty_{t, x}}  \big\| \Pi_\pm C^\pm_d \psi_s \big\|_{L^2_{t, x}}  \, ds\\
      &\lesa ( \beta \lambda)^\frac{n-1}{2} d^\frac{1}{2} \int_\RR \| b_s \|_{L^\infty_{t, x}}  \big\| \Pi_\pm C^\pm_d \psi_s \big\|_{L^2_{t, x}}  \, ds
    \end{align*}
where we used the fact that the number of small caps $\kappa' \in \mc{C}_{\beta'}$ required to cover the larger cap $\bar{\kappa} \in \mc{C}_\beta$ is bounded above by $\big(\frac{\beta}{\beta'}\big)^{n-1}$.
\end{proof}
\end{proposition}

We can now prove the homogeneous case of Theorem \ref{thm - null frame bounds}.

\begin{corollary}[Null frame bounds - homogeneous case]\label{cor - null frame bounds hom case}
\leavevmode
\begin{enumerate}
  \item Let $\alpha \ll 1$ and $f \in L^2_x$ with
       $\supp \widehat{\Pi_+ f} \subset {A}^\pm_\lambda( \tfrac{3}{2} \kappa)$ and $ \supp \widehat{\Pi_- f } \subset{A}^\mp_\lambda(
       \tfrac{3}{2} \kappa)$.
  Then

            \begin{equation}\label{eqn - cor null frame bounds lin case - NF bound} \big\| \mc{U}_\pm(t) f    \|_{[NF^\pm]^*(\kappa)} \lesa \| f  \|_{L^2_x}.\end{equation}

  \item  Let $\beta \les \alpha \ll 1$,  $\kappa \in \mc{C}_\alpha$, and $\bar{\kappa} \in \mc{C}_\beta$ with $2\kappa \cap 2\bar{\kappa} \not = \varnothing$. Let $\rho \in C^\infty_0(\RR)$ and $T>0$. Assume $f \in  L^2_x$ with
    $\supp \widehat{\Pi_+ f} \subset {A}^\pm_\lambda( 2 \bar{\kappa})$ and $\supp \widehat{\Pi_- f }  \subset {A}^\mp_\lambda( 2 \bar{\kappa})$. Then
               \begin{equation}\label{eqn - cor null frame bounds lin case - PW bound}\big\| \rho(\tfrac{t}{T}) \mc{U}_\pm(t) f \big\|_{PW^\mp(\kappa)} \lesa (\beta \lambda)^{\frac{n-1}{2}} \| f \|_{L^2_x} \end{equation}
    with constant independent of $T$.
\end{enumerate}
\begin{proof} We start by proving $(i)$. By a reflection in the $x$ variable, we may assume that $\pm = +$.  The estimate (\ref{eqn - hom soln in null coord}) gives
    $$ \| e^{\mp i t |\nabla|} \Pi_{\pm} f \|_{L^\infty_{t_\omega} L^2_{x_\omega}}  = \| e^{\mp i (t - x \cdot \omega) |\nabla|} \Pi_{\pm}f(x) \|_{L^\infty_{t} L^2_{x}} \approx \big\| \theta(\omega, \mp \xi)^{-1} \widehat{\Pi_{\pm}f} \big\|_{L^2_{\xi}} \lesa \theta(\omega, \kappa)^{-1} \| \Pi_{\pm}f \|_{L^2_x}.$$
Note that for $\omega \not \in 2\kappa$ and $\xi \in \supp \widehat{\Pi_\pm f}$, we have the null form estimate $|\Pi_\omega \Pi_{\pm \frac{\xi}{|\xi|}} | \lesa \theta(\omega, \mp \xi)\approx \theta(\omega, \kappa)$. Therefore, by decomposing $\mc{U}_+(t) = e^{ - i t |\nabla|} \Pi_+ + e^{ i t |\nabla|} \Pi_-$,  we deduce that
    \begin{align*}
      \big\| \Pi_\omega \mc{U}_+(t) f  \big\|_{L^\infty_{t_\omega} L^2_{x_\omega}} + \theta(\omega, \kappa) \big\| \Pi_{-\omega} \mc{U}_+(t) f \big\|_{L^\infty_{t_\omega} L^2_{x_\omega}} &\lesa \sum_\pm \big\| e^{\mp i t|\nabla|} \Pi_\omega \Pi_\pm f\big\|_{L^\infty_{t_\omega} L^2_{x_\omega}} + \theta(\omega, \kappa) \big\| e^{ \mp i t|\nabla|} \Pi_{-\omega} \Pi_\pm f \big\|_{L^\infty_{t_\omega} L^2_{x_\omega}} \\
                &\lesa \sum_\pm \theta(\omega, \kappa)^{-1} \big\| \Pi_\omega \widehat{\Pi_\pm f} \big\|_{L^2_{\xi}} +  \big\| \Pi_{-\omega} \widehat{\Pi_\pm f} \big\|_{L^2_{\xi}} \\
                &\lesa \| f \|_{L^2_x}.
    \end{align*}
Taking the sup over $\omega \not \in 2\kappa$ then gives (\ref{eqn - cor null frame bounds lin case - NF bound}).

On the other hand, (ii) follows directly from Proposition \ref{prop - Xsb controls PW} (with $\psi_s(t, x) = \rho(\frac{t}{T}) \mc{U}_\pm(t)f$ and $b_s(t, x) = \ind_{[0, 1]}(s)$) and Lemma \ref{lem - Xsb, homogeneous solns, and time cutoffs}.
\end{proof}
\end{corollary}

Finally, we need to be able to commute the projections $C^\pm_{\ll \alpha^2 \lambda}$ and the time cutoff $\rho(\frac{t}{T})$.

\begin{lemma}\label{lem - PW est, cutoff and Proj commute}
  Let $\rho \in C^\infty_0(\RR)$ and $T>0$. If $u \in F^\pm_\lambda$ then
         \begin{equation}\label{eqn - lem PW est cutoff and Proj commute - main est}\Big( \sum_{\kappa \in \mc{C}_\alpha} \big\| R^\pm_{\kappa, \alpha^2 \lambda} \Pi_\pm \big( \rho(\tfrac{t}{T}) u \big) \big\|_{PW^-(\kappa)}^2 \Big)^\frac{1}{2} \lesa  \Big( \sum_{\kappa \in \mc{C}_\alpha} \big\| \rho(\tfrac{t}{T}) R^\pm_{\kappa, \alpha^2 \lambda}  \Pi_\pm u  \big\|_{PW^-(\kappa)}^2 \Big)^\frac{1}{2} + (\alpha \lambda)^{\frac{n-1}{2}} \|u \|_{F^+_\lambda}  \end{equation}
  with constant independent of $T$.
\begin{proof}
 The idea is to decompose $u = C^\pm_{ \ll \alpha^2 \lambda} u + C^\pm_{\gtrsim \alpha^2 \lambda} u $ into a component close to the cone, and a component far from the cone. For the close cone term, after an application of Lemma \ref{lem - mult are disposable} to dispose of the outer multiplier, we deduce that
    \begin{align*}
      \Big( \sum_{\kappa \in \mc{C}_\alpha} \big\|  R^\pm_{\kappa, \alpha^2 \lambda} \Pi_\pm \big( \rho(\tfrac{t}{T}) C^\pm_{\ll \alpha^2 \lambda} u \big) \big\|_{PW^-(\kappa)}^2 \Big)^\frac{1}{2} &= \Big( \sum_{\kappa \in \mc{C}_\alpha} \big\| C^\pm_{\ll \alpha^2 \lambda} \big( \rho(\tfrac{t}{T}) \Pi_\pm R^\pm_{\kappa, \alpha^2 \lambda}  u \big) \big\|_{PW^-(\kappa)}^2 \Big)^\frac{1}{2}\\
      &\lesa \Big( \sum_{\kappa \in \mc{C}_\alpha} \big\| \rho(\tfrac{t}{T})  R^\pm_{\kappa, \alpha^2 \lambda}  \Pi_\pm u  \big\|_{PW^-(\kappa)}^2 \Big)^\frac{1}{2}
    \end{align*}
 which gives the first term on the righthand side of (\ref{eqn - lem PW est cutoff and Proj commute - main est}). On the other hand, for the far cone term, an application of the $X^{s, b}$ estimate in Proposition \ref{prop - Xsb controls PW} together with the $L^2_x$ orthogonality of the $R^\pm_{\kappa, \alpha^2 \lambda} $ multipliers gives
    \begin{align*}
      \Big( \sum_{\kappa \in \mc{C}_\alpha} \big\| R^\pm_{\kappa, \alpha^2 \lambda}  \Pi_\pm \big( \rho(\tfrac{t}{T}) C^\pm_{\gtrsim \alpha^2 \lambda} u \big) \big\|_{PW^-(\kappa)}^2 \Big)^\frac{1}{2} &\lesa (\alpha \lambda)^{\frac{n-1}{2}} \Big( \sum_{\kappa \in \mc{C}_\alpha} \big\|R^\pm_{\kappa, \alpha^2 \lambda}  \Pi_\pm \big( \rho(\tfrac{t}{T}) C^\pm_{\gtrsim \alpha^2 \lambda} u \big) \big\|_{\dot{\mc{X}}^{\frac{1}{2}, 1}_+}^2 \Big)^\frac{1}{2} \\
      &\lesa (\alpha \lambda)^{\frac{n-1}{2}} \big\|  \Pi_\pm C^\pm_{\ll \alpha^2 \lambda} \big( \rho(\tfrac{t}{T}) C^\pm_{\gtrsim \alpha^2 \lambda}  u \big) \big\|_{\dot{\mc{X}}^{\frac{1}{2}, 1}_+} \\
      &\lesa (\alpha \lambda)^{\frac{n-1}{2}} ( \alpha^2 \lambda)^\frac{1}{2} \big\|  \rho(\tfrac{t}{T})\Pi_\pm C^\pm_{\gtrsim \alpha^2 \lambda} u \big\|_{L^2_{t, x}} \\
      &\lesa (\alpha \lambda)^{\frac{n-1}{2}} \| \rho \|_{L^\infty}  ( \alpha^2 \lambda)^\frac{1}{2}  \|  C^\pm_{\gtrsim \alpha^2 \lambda}  \Pi_\pm u \|_{L^2_{t, x}} \\
      &\lesa (\alpha \lambda)^\frac{n-1}{2} \| u \|_{F_\lambda^+}
    \end{align*}
where the last inequality followed from (\ref{eqn - L2 control away from null cone}).

\end{proof}
\end{lemma}

%------------------------------------------------------------------------------------------------------------------------------%
\subsection{Proof of Theorem \ref{thm - null frame bounds}}\label{subsec - proof of null frame bounds}
%------------------------------------------------------------------------------------------------------------------------------%

We now turn to the proof of Theorem \ref{thm - null frame bounds}.

  \begin{proof}[Proof of Theorem \ref{thm - null frame bounds}]  The proof proceeds by essentially reducing the problem to the homogeneous case, at which point we may apply Corollary \ref{cor - null frame bounds hom case}. This type of argument is fairly straightforward if we are in $L^1_t L^2_x$ or $X^{s, b}$, but is more involved in the null frame case as we need to use the Duhamel formula in null coordinates, together with the decomposition into free waves contained in Lemma \ref{lem - null frame decomp into free waves}.\\

  We begin by noting that after a reflection in the $x$ variable, we may assume $\pm = +$. An application of  Corollary \ref{cor - orthog in null frame} gives the orthogonality bound\footnote{This follows by decomposing $F$ into atoms. For energy and $\dot{\mc{X}}^{-\frac{1}{2}, 1}_+$ atoms, we can just use the orthogonality in $L^2_x$ of $R^\pm_{\kappa}$. For $NF^+_\lambda$ atoms we just use (ii) in Corollary \ref{cor - orthog in null frame}.}
            $$ \Big( \sum_{\kappa \in \mc{C}_\alpha} \|  R^\pm_{\kappa, \alpha^2 \lambda} \Pi_\pm F \|_{\mc{N}^+_\lambda}^2 \Big)^\frac{1}{2} \lesa \| F \|_{\mc{N}^+_\lambda}.$$
  Consequently, after an application of the homogeneous case Corollary \ref{cor - null frame bounds hom case} (together with Lemma \ref{lem - PW est, cutoff and Proj commute} in the $PW^-(\kappa)$ case) it suffices to prove
  $$ \| u \|_{[NF^+]^*(\kappa)} \lesa  \big\|  (\p_t + \sigma \cdot \nabla) R^\pm_{\kappa, \alpha^2 \lambda} \Pi_\pm u  \big\|_{\mc{N}^+_\lambda}$$
and
    $$
        \big\| \rho(\tfrac{t}{T}) u \big\|_{PW^-(\kappa)} \lesa (\alpha \lambda)^{\frac{n-1}{2}} \big\| (\p_t + \sigma \cdot \nabla) R^\pm_{\kappa, \alpha^2 \lambda} \Pi_\pm u \big\|_{\mc{N}^+_\lambda}$$
If we now decompose $(\p_t + \sigma \cdot \nabla) R^\pm_{\kappa, \alpha^2 \lambda} \Pi_\pm u$ into atoms, we reduce to showing that
            \begin{equation}\label{eqn - thm null frame bounds - main inequality} \| u \|_{[NF^+]^*(\kappa)}  + (\alpha \lambda)^{- \frac{n-1}{2}} \big\| \rho(\tfrac{t}{T}) u \big\|_{PW^-(\kappa)}  \lesa 1 \end{equation}
  where $u$ is the solution to $(\p_t + \sigma \cdot \nabla) u= F$ with $u(0) = 0$,
and  $F = \Pi_\pm F$ is a $\mc{N}^+_\lambda$ atom. Note that we may assume  $ \supp \widetilde{\Pi_\pm F} \subset {^\natural A}^\pm_{\lambda, \alpha}(\kappa)$. We now separately consider the three possible cases;  $F$ is an energy atom, $F$ is a $\dot{\mc{X}}^{-\frac{1}{2}, 1}_+$ atom, or $F$ is a $NF^+_\lambda$ atom. \\

  \textbf{Case 1: $F$ is a energy atom.} If we write $u$ using the Duhamel formula we have
                $$ u(t, x) = \int_0^t \mc{U}_+(t-s) F(s) ds .$$
  Note that for each fixed $s$, the $\xi$ support of $\mc{U}_+(t-s)F(s)$ is contained in the set ${^\natural A}^\pm_{\lambda}(\kappa)$.  Therefore, as $\| \cdot \|_{[NF^+]^*(\kappa)}$ satisfies Minkowski's inequality, we have by Corollary \ref{cor - null frame bounds hom case} (for a fixed cap)
  \begin{align*}  \big\| u \big\|_{[NF^+]^*(\kappa)}
            &\lesa \int_\RR \big\| \ind_{[0, T]}(s) \mc{U}_+(t-s) F(s) \big]  \big\|_{[NF^+]^*(\kappa)} ds \\
            &\lesa \| F \|_{L^1_t L^2_x } \les 1.
  \end{align*}
  Similarly, an application of Proposition \ref{prop - Xsb controls PW} followed by Lemma \ref{lem - Xsb, homogeneous solns, and time cutoffs} gives
  \begin{align*}
    \big\| \rho(\tfrac{t}{T})  u \big\|_{PW^-(\kappa)} &=  \bigg\| \int_{\RR} \ind_{[0,T]}(s) \, \rho(\tfrac{t}{T}) \mc{U}_+(t-s) F(s) ds \bigg\|_{PW^-(\kappa)} \\
    &\lesa (\alpha \lambda)^{\frac{n-1}{2}} \int_\RR \big\| \rho(\tfrac{t}{T})  \mc{U}_+(t-s) F(s) \big\|_{\dot{\mc{X}}^{\frac{1}{2}, 1}_+} \, ds \\
    &\lesa (\alpha \lambda)^{\frac{n-1}{2}} \| F \|_{L^1_t L^2_x} \les 1
  \end{align*}
 Therefore (\ref{eqn - thm null frame bounds - main inequality}) follows in the case where $F$ is a $L^1_t L^2_x$ atom. \\

  \textbf{Case 2: $F$ is a $\mc{X}^{-\frac{1}{2}, 1}_+$ atom.}  Assume $F$ is a $\dot{\mc{X}}^{-\frac{1}{2}, 1}_+$ atom, thus
    $$\supp \widetilde{F} = \supp \widetilde{\Pi_\pm F} \subset \big\{ \big| \tau \pm |\xi| \big| \approx d \, \big\} \cap {^\natural A}^\pm_{\lambda, \alpha}(\kappa)$$
  and $ \|  F \|_{L^2_{t, x}} \les d^{\frac{1}{2}}$. Note that an application of Lemma \ref{lem - Xsb decomp into free waves} together with Corollary \ref{cor - null frame bounds hom case} shows that, provided $v \in L^2_{t, x}$ with $\supp \widehat{\Pi_\pm v} \subset {^\natural A}^\pm_{\lambda}( \tfrac{3}{2} \kappa)$, we have
    \begin{equation}\label{eqn - thm null frame bounds - Xsb controls NF*}
        \| v \|_{[NF^+]^*(\kappa)} \lesa \| v \|_{\dot{\mc{X}}^{\frac{1}{2}, 1}_+}.
    \end{equation}
 Therefore, by writing
    $$u = \Pi_\pm u = ( \p_t \pm i |\nabla|)^{-1} \Pi_\pm F - \mc{U}_+(t) \big[ ( \p_t \pm i |\nabla|)^{-1} \Pi_\pm F(0)\big] $$
 we have by Corollary \ref{cor - null frame bounds hom case} and (\ref{eqn - thm null frame bounds - Xsb controls NF*})
 \begin{align*}
   \| u \|_{[NF^+]^*(\kappa)} &\les \big\| ( \p_t \pm i |\nabla|)^{-1} \Pi_\pm F\big\|_{[NF^+]^*(\kappa)}  + \big\|  \mc{U}_+(t) \big[ ( \p_t \pm i |\nabla|)^{-1} \Pi_\pm F(0)\big] \big\|_{[NF^+]^*(\kappa)} \\
   &\lesa \big\| ( \p_t \pm i |\nabla|)^{-1} \Pi_\pm F\big\|_{\dot{\mc{X}}^{\frac{1}{2}, 1}_+}  + \big\|  ( \p_t \pm i |\nabla|)^{-1} \Pi_\pm F \big\|_{L^\infty_t L^2_x}\\
   &\lesa d^{ - \frac{1}{2}} \| F \|_{L^2_{t, x}} \les 1
 \end{align*}
 as required. The $PW^-(\kappa)$ estimate is similar, we just replace the estimate (\ref{eqn - thm null frame bounds - Xsb controls NF*})
 with  an application of Proposition \ref{prop - Xsb controls PW} (where $b_s(t, x) = \ind_{[0, 1]}(s)$, $\psi_s(t, x) = ( \p_t \pm i |\nabla|)^{-1} \Pi_\pm F$, and $\kappa = \bar{\kappa}$). \\

  \textbf{Case 3: $F$ is a $NF^+_\lambda$ atom.}
By definition, we have a decomposition $F = \sum_{\bar{\kappa} \in \mc{C}_\beta} F_{\bar{\kappa}}$ where we may assume $F_{\bar{\kappa}} = \Pi_\pm F_{\kappa}$,
        $$ \supp \widetilde{\Pi_\pm F}_{\bar{\kappa}} \subset {^\natural A}^\pm_{\lambda, \alpha}(\kappa) \cap A^\pm_{\lambda, \beta}(\bar{\kappa})$$
  and
    $$ \Big( \sum_{\bar{\kappa}} \| F_{\bar{\kappa}} \|_{NF^+(\bar{\kappa})}^2 \Big)^\frac{1}{2} \les 1.$$
Define $u_{\bar{\kappa}}$ as the solution to $(\p_t + \sigma \cdot \nabla) u_{\bar{\kappa}} = F_{\bar{\kappa}}$ with $u_{\bar{\kappa}}(0) =0$. Assume for the moment that we have the cap localised estimates
        \begin{equation} \label{eqn - thm null frame bound - NF case cap localised NF*}
                \big\|   u_{\bar{\kappa}} \big\|_{[NF^+]^*(\kappa)} \lesa \|   F_{\bar{\kappa}} \|_{NF^+(\bar{\kappa})}
        \end{equation}
and
     \begin{equation} \label{eqn - thm null frame bound - NF case cap localised PW}
                \big\| \rho(\tfrac{t}{T})  u_{\bar{\kappa}} \big\|_{PW^-(\kappa)} \lesa (\min\{\alpha, \beta\} \lambda)^{\frac{n-1}{2}} \|  F_{\bar{\kappa}} \|_{NF^+(\bar{\kappa})}.
        \end{equation}
 Then, writing
    $$u = \sum_{\bar{\kappa} \in \mc{C}_\beta} u_{\bar{\kappa}} = \sum_{\substack{\bar{\kappa} \in \mc{C}_\beta \\ {^\natural \kappa} \cap \bar{\kappa} \not = \varnothing }} u_{\bar{\kappa}}$$
 and using the orthogonality given  by $(i)$ in  Corollary \ref{cor - orthog in null frame}, together with (\ref{eqn - thm null frame bound - NF case cap localised NF*}) we deduce that
    \begin{align*}
      \|  u \|_{[NF^+]^*(\kappa)}  \lesa \bigg( \sum_{\bar{\kappa} \in \mc{C}_\beta} \|   u_{\bar{\kappa}} \|_{[NF^+]^*(\kappa)}^2 \bigg)^{\frac{1}{2}}
      \lesa  \bigg( \sum_{ \bar{\kappa} \in \mc{C}_\beta} \|  F_{\bar{\kappa}} \|_{NF^+(\bar{\kappa})}^2 \bigg)^\frac{1}{2} \les 1
    \end{align*}
  as required. For the $PW^-(\kappa)$ estimate, the argument is slightly different as we don't have any orthogonality in the $\bar{\kappa}$ sum due to the fact that the $PW^-$ norm is built up of $L^\infty_{x_\omega}$ terms. This is not a problem in the case $\alpha \lesa \beta$ as the sum only contains $\mc{O}(1)$ terms. On the other hand, if $\alpha \gg \beta$, then the estimate (\ref{eqn - thm null frame bound - NF case cap localised PW}) has a much better constant than what is needed, since we want to end up with $(\alpha \lambda)^{\frac{n-1}{2}}$ and have $(\beta \lambda)^{\frac{n-1}{2}}$. Thus, in place of any orthogonality argument, we use the triangle inequality to reduce to a single cap $\bar{\kappa}$, followed by Holder to regain the square sum over the caps $\bar{\kappa}$. In more detail, from the cap localised estimate (\ref{eqn - thm null frame bound - NF case cap localised PW}), and an application of Holder in the $\bar{\kappa}$ sum, we obtain
  \begin{align*}
     \big\|  \rho(\tfrac{t}{T}) u \big\|_{PW^-(\kappa)} &\lesa \sum_{\substack{\bar{\kappa} \in \mc{C}_\beta \\ {^\natural \kappa} \cap \bar{\kappa} \not = \varnothing}} \big\| \rho(\tfrac{t}{T}) u_{\bar{\kappa}} \big\|_{PW^-(\kappa)} \\
     &\lesa (\min\{ \alpha, \beta\} \lambda)^{\frac{n-1}{2}}  \sum_{\substack{\bar{\kappa} \in \mc{C}_\beta \\ {^\natural \kappa} \cap \bar{\kappa} \not = \varnothing}} \big\| F_{\bar{\kappa}} \big\|_{NF^+(\bar{\kappa})} \\
     &\lesa (\min\{ \alpha, \beta\} \lambda)^{\frac{n-1}{2}} \Big( \frac{ \alpha}{\min\{ \alpha, \beta\}} \Big)^{\frac{n-1}{2}} \bigg( \sum_{\bar{\kappa} \in \mc{C}_\beta} \big\|   F_{\bar{\kappa}} \big\|_{NF^+(\bar{\kappa})}^2 \bigg)^{\frac{1}{2}} \\
     &\lesa (\alpha \lambda)^{\frac{n-1}{2}}
  \end{align*}
where we used the fact that, for a fixed $\kappa \in \mc{C}_\alpha$,  $\#\big\{ \bar{\kappa} \in \mc{C}_\beta \, | \,{^\natural\kappa} \cap \bar{\kappa} \not = \varnothing \, \big\} \lesa (\frac{\alpha}{\min\{ \alpha, \beta\}})^{n-1}$.

It remains to proof the cap localised estimates (\ref{eqn - thm null frame bound - NF case cap localised NF*}) and (\ref{eqn - thm null frame bound - NF case cap localised PW}). The atomic definition of $NF^+(\bar{\kappa})$, shows that it is enough to consider the case where $F_{\bar{\kappa}}$ is an atom, in other words there exists $\omega \not \in 2\bar{\kappa}$ such that
        $$ \big\| \Pi_\omega F_{\bar{\kappa}} \|_{L^1_{t_\omega} L^2_{x_\omega}} + \theta(\omega, \bar{\kappa})^{-1} \big\| \Pi_{-\omega} F_{\bar{\kappa}} \big\|_{L^1_{t_\omega} L^2_{x_\omega}} \les 1$$
and  (by  Lemma \ref{lem - mult are disposable}) we may assume that $F_{\bar{\kappa}} = \Pi_\pm F_{\bar{\kappa}}$ and $ \supp \widetilde{ \Pi_\pm F}_{\bar{\kappa}}\subset {^\natural A}^\pm_{\lambda, \alpha}(\kappa) \cap {^\natural A}^\pm_{\lambda, \beta}(\bar{\kappa})$. As $u_{\bar{\kappa}}$ satisfies $(\p_t + \sigma \cdot \nabla)  u_{\bar{\kappa}} =  F_{\bar{\kappa}}$ with $u_{\bar{\kappa}}(0) =0$, we have
            $$  u_{\bar{\kappa}} = \Pi_\pm E^+_\omega*  F_{\bar{\kappa}} -  \Pi_\pm \mc{U}_+(t)\big[ E^+_\omega*  F_{\bar{\kappa}}(0)\big].$$
The homogeneous term can be controlled by\footnote{Note that $\supp \widetilde{[E^+_\omega* F}_{\bar{\kappa}}] = \supp \widetilde{F}_{\bar{\kappa}} \subset {^\natural A}^\pm_{\lambda, \alpha}(\kappa)\cap {^\natural A}^\pm_{\lambda, \beta}(\bar{\kappa})$.}
 Corollary \ref{cor - null frame bounds hom case} followed by Corollary \ref{cor - LinftyL2 control of null fund soln}. For the $E^+_\omega* F_{\bar{\kappa}}$ term, we use an application of Lemma \ref{lem - null frame decomp into free waves} to write
  $$ E^+_\omega* F_{\bar{\kappa}} = \int_{-\infty}^{t_\omega} \psi_{a} d a + \Pi_{-\omega} G$$
  where $\psi_{a}$ is a homogeneous solution with $\supp  \widehat{\Pi_\pm \psi}_{a} \subset {^\natural A}^\pm_\lambda(\tfrac{3}{2} \kappa) \cap {^\natural A}^\pm_\lambda(\tfrac{3}{2} \bar{\kappa})$,  $\supp \widetilde{G} = \supp \widetilde{  F}_{\bar{\kappa}} $, and we have the bound
         \begin{equation}\label{eqn - thm null frame bound - est for free wave decomp} \begin{split} \int_{\RR} \| \psi_{a}\|_{L^\infty_t L^2_x} d a +& \big\|  G\big\|_{\dot{\mc{X}}^{\frac{1}{2}, 1}_+}  \\
                &\lesa  \big\| \Pi_\omega F_{\bar{\kappa}}\big\|_{L^1_{t_\omega} L^2_{x_\omega}} + \theta(\omega, \kappa)^{-1} \big\| \Pi_{-\omega} F_{\bar{\kappa}} \big\|_{L^1_{t_\omega} L^2_{x_\omega}} \les 1.\end{split} \end{equation}
  The integral term is easy to control via Proposition \ref{prop - Xsb controls PW} and Corollary \ref{cor - null frame bounds hom case}. For instance, using (\ref{eqn - thm null frame bound - est for free wave decomp}) and Corollary \ref{cor - null frame bounds hom case} we have
    \begin{align*}
      \Big\|\int_{-\infty}^{t_\omega} \psi_{a} da \Big\|_{[NF^+](\kappa)}
      \les \int_\RR \| \psi_{a} \|_{[NF^+]^*(\kappa)} da
      \lesa \int_\RR \| \psi_{a} \|_{L^\infty_t L^2_x} da
      \lesa 1.
    \end{align*}
Similarly, using Proposition \ref{prop - Xsb controls PW} and Lemma \ref{lem - Xsb, homogeneous solns, and time cutoffs}, gives
\begin{align*}
      \Big\|\rho(\tfrac{t}{T}) \, \int_{-\infty}^{t_\omega} \psi_{a} da \Big\|_{PW^-(\kappa)}
      &\lesa ( \min\{ \alpha, \beta\} \lambda)^{\frac{n-1}{2}} \int_\RR \big\| \rho( \tfrac{t}{T}) \psi_{a} \big\|_{\dot{\mc{X}}^{\frac{1}{2}, 1}_+} da\\
      &\lesa  (\min\{ \alpha, \beta\}\lambda)^{\frac{n-1}{2}}\int_\RR \| \psi_{a} \|_{L^\infty_t L^2_x} da
      \lesa  (\min\{ \alpha, \beta\}\lambda)^{\frac{n-1}{2}}.
    \end{align*}
Finally, the estimate for the $G$ term simply follows from (\ref{eqn - thm null frame bounds - Xsb controls NF*}) (in the $[NF^+]^*(\kappa)$ case) and Proposition \ref{prop - Xsb controls PW} (in the $PW^-(\kappa)$ case). This completes the proof of the $NF^+_\lambda$ case, and hence Theorem \ref{thm - null frame bounds} follows.
\end{proof}

%------------------------------------------------------------------------------------------------------------------------------%

%------------------------------------------------------------------------------------------------------------------------------%
   %------------------------------------------------------------------------------------------------------------------------------%
%------------------------------------------------------------------------------------------------------------------------------%
\section{Strichartz Type Estimates} \label{sec - proof of stricharz est}
%------------------------------------------------------------------------------------------------------------------------------%
%------------------------------------------------------------------------------------------------------------------------------%

In this section our aim is to prove Theorem \ref{thm - F controls strichartz and angular sum}, i.e. we want to show that the norm $F^\pm_\lambda$ controls the Strichartz norms $L^q_t L^r_x$. The observation that it is possible to control the Strichartz norms by the null frame type norms was first observed by Sterbenz-Tataru in \cite[Lemma 5.8]{Sterbenz2010a} in work related to the wave maps equation. The proof use a version of the $X^{s, b}$ spaces, with the derivative in the ``time'' direction, replaced with spaces of bounded variation. Spaces of this type have been used in the work of Koch-Tataru \cite{Koch2005}, and Hadac-Herr-Koch \cite{Hadac2009}. The crucial point is an atomic decomposition contained in \cite[Lemma 6.4]{Koch2005}. The argument presented below  is based heavily on the arguments used by Sterbenz-Tataru in \cite{Sterbenz2010a}, although it has been slightly simplified, compressed, and adapted to our context.   \\

%------------------------------------------------------------------------------------------------------------------------------%
\subsection{The Spaces $V^p$}
%------------------------------------------------------------------------------------------------------------------------------%

We define the $p$-variation of a function $u : \RR \rightarrow L^2_x$ as
        $$ | u |_{V^p}^p = \sup_{ (t_k) \in \mc{Z}} \sum_{ k \in \ZZ}  \big\| u(t_{k+1}) -  u(t_k) \big\|_{L^2_x}^p$$
where $\mc{Z} = \{ (t_k)_{k \in \ZZ} \, | \, t_k \les t_{k+1} \}$ denotes the set of all increasing sequences on $\RR$. If $| u |_{V^p} < \infty$, then $u$ has at most countable discontinuities, and its left and right limits exist everywhere. In particular $\lim_{t \rightarrow \pm \infty} u(t)$ exists in $L^2_x$. These properties are all classical results, but for completeness we sketch the proof here.

\begin{lemma}\label{lem - left and right limits always exist in Vp spaces}
Let $0<p<\infty$ and $u: \RR \rightarrow L^2_x(\RR^n)$ with $|u|_{V^p} < \infty$. Then $u$ has left and right limits everywhere, and in particular $u(t)$ converges to some function $f_{\pm \infty} \in L^2_x(\RR^n)$ as $t \rightarrow \pm \infty$.
\begin{proof}
  Let
        $$ \rho(t) = \sup \Big\{ \sum_{k=1}^{N-1} \| u(t_{k+1}) - u(t_k) \|_{L^2_x}^p \, \,\Big|\,\, -\infty<t_1<t_2<...<t_N=t \Big\}.$$
 We claim that $\rho$ is increasing, and for $s<t$ we have the inequality
        \begin{equation}\label{eqn rho increasing appendix} \| u(t) - u(s) \|_{L^2_x}^p \les \rho(t) - \rho(s). \end{equation}
 This follows by observing that we have a sequence $-\infty< s_1 < ... < s_N=s$ such that
        $$ \rho(s) \les \epsilon + \sum_{k=1}^{N-1} \| u(s_{k+1}) - u(s_k) \|_{L^2_x}^p$$
 and consequently
       \begin{align*} \| u(t) - u(s) \|_{L^2_x}^p &\les \sum_{k=1}^{N-1} \| u(s_{k+1}) - u(s_k) \|_{L^2_x}^p + \| u(t) - u(s) \|_{L^2_x}^p - \Big( \sum_{k=1}^{N-1} \| u(s_{k+1}) - u(s_k) \|_{L^2_x}^p \Big)\\
       & \les \rho(t) - \big( \rho(s) - \epsilon\big) = \rho(t) - \rho(s) + \epsilon \end{align*}
 since this is true for every $\epsilon>0$, we obtain (\ref{eqn rho increasing appendix}) and consequently $\rho$ must be increasing.

 Now as $\rho$ is increasing and bounded (we clearly have $0 \les \rho \les |u|_{V^p}^p$), its left and right limits must exist everywhere. Hence, given any sequence $t_k \rightarrow t$ from below, $\rho(t_k)$ forms a Cauchy sequence in $\RR$, which implies by (\ref{eqn rho increasing appendix}), that $u(t_k)$ also forms a Cauchy sequence. Thus $u(t_k)$ must converge in $L^2$, and consequently, must have limits from the left and right everywhere.
\end{proof}
\end{lemma}

We now define $V^p$ to be the set of all right continuous functions from $\RR$ into $L^2_x$ with norm
        $$ \| u \|_{V^p}^p = \| u \|_{L^\infty_t L^2_x}^p + | u |_{V^p}^p,$$
the $L^\infty_t$ term is needed to ensure that $\| \cdot \|_{V^p}$ is a norm. The key property of the $V^p$ spaces that we require is the following.

\begin{lemma}[Lemma 6.4 in \cite{Koch2005}]\label{lem - V2 decomp}
  Assume $u \in V^p$. Then we have a decomposition $u = \sum_{j=1}^\infty v_j$, where the sum converges in $L^\infty_t L^2_x$, and moreover
    \begin{enumerate}
      \item For each $j$, we have a partition $\mc{I}_j$ of $\RR$, into intervals
        $$I_0=\big(-\infty, t^{(j)}_0\big),\,\,\, I_1=\big[t^{(j)}_0, t^{(j)}_1\big),\,\, ...\,, \,\,\,I_N=\big[t^{(j)}_{N-1}, \infty\big),$$
       and we can write
            $$v_j(t) = \sum_{ I_k \in \mc{I}_j} \ind_{I_k}(t) f_k^{(j)}$$
      for functions $f^{(j)}_k \in L^2_x$.

    \item We have the bounds
                $$ \# \mc{I}_j \lesa 2^{pj}, \qquad \qquad \sup_{k} \| f^{(j)}_k \|_{L^2_x} \lesa 2^{-j} \| u \|_{V^p}.$$
    \end{enumerate}
\begin{proof} The proof follows from minor modifications of the argument in Lemma 6.4 in \cite{Koch2005}.
\end{proof}
\end{lemma}

 Recall that $\mc{U}_\pm(t)$ denotes the forward solution operator for $(\p_t \pm \sigma \cdot \nabla)u = 0$ with data at $t=0$.  With this notation in hand, the previous lemma then has the following important corollary.

 \begin{corollary}\label{cor - V2 controls strichartz}
 Let $2 \les q , r \les \infty$ with $q>2$ and $\frac{1}{q} + \frac{n-1}{2r} \les \frac{n-1}{4}$. Assume $\mc{U}_\pm(-t) u \in V^2$ with $\supp \widehat{u} \subset \{ |\xi| \approx \lambda\}$. Let $\mc{M}$ denote a spatial Fourier multiplier with matrix valued symbol $m(\xi)$ such that $|m(\xi)| \lesa \delta$ for all $\xi \in \supp \widehat{u}$. Then
  $$ \| \mc{M} u \|_{L^q_t L^r_x} \lesa \delta  \lambda^{ n ( \frac{1}{2} - \frac{1}{r}) - \frac{1}{q}} \| \mc{U}_\pm(-t) u \|_{V^2}. $$
  \begin{proof}
    Since $\mc{U}_\pm(-t) u \in V^2$, an application of Lemma \ref{lem - V2 decomp} gives a decomposition
                $$ u = \mc{U}_\pm(t) \mc{U}_\pm(-t) u = \sum_j \mc{U}_\pm(t) v_j $$
    with $v_j$ satisfying the properties in $(i)$ and $(ii)$ in Lemma \ref{lem - V2 decomp}. We may assume that $\supp \widehat{v_j} = \supp \widehat{u}$, and hence the same holds for the $L^2_x$ functions $f^{(j)}_k$ making up the sum in $v^j$. Then recalling that
        $$\mc{U}_\pm(t) = e^{\pm i t |\nabla|} \Pi_+ + e^{ \mp i t |\nabla|} \Pi_-$$
    we obtain from Lemma \ref{lem - V2 decomp}
        \begin{align*}
           \| \mc{M} u \|_{L^q_t L^r_x}  &\les \sum_j \| \mc{M} \mc{U}_\pm(t)  v_j \|_{L^q_t L^r_x} \\
                                                                &\les \sum_j \Big( \sum_{I_k \in \mc{I}_j} \big\| e^{\pm i t|\nabla|} \big(\mc{M} \Pi_+  f^{(j)}_k\big)  \big\|_{L^q_t L^r_x(I_k \times \RR^n)}^q  + \big\| e^{\mp i t|\nabla|} \big(\mc{M} \Pi_-  f^{(j)}_k\big)  \big\|_{L^q_t L^r_x(I_k \times \RR^n)}^q \Big)^{\frac{1}{q}}\\
                                                                &\lesa \lambda^{n(\frac{1}{2} - \frac{1}{r}) - \frac{1}{q}}  \sum_j \Big( \sum_{I_k \in \mc{I}_j} \big(\| \mc{M} \Pi_+f^{(j)}_k \|_{L^2_x}+ \| \mc{M} \Pi_+f^{(j)}_k \|_{L^2_x}\big)^q \Big)^{\frac{1}{q}}\\
                                                                &\les \delta  \lambda^{n(\frac{1}{2} - \frac{1}{r}) - \frac{1}{q}}  \sum_j \Big( \sum_{I_k \in \mc{I}_j} \big\| f^{(j)}_k \big\|_{L^2_x}^q\Big)^{\frac{1}{q}}.
        \end{align*}
  Now using the properties $\sup_k \| f^{(j)}_k \|_{L^2_x} \lesa 2^{-j} \| \mc{U}_\pm(-t) u \|_{V^2}$ and $\# \mc{I}_j \lesa 2^{2j}$ we obtain
   $$ \sum_j \Big( \sum_{I_k \in \mc{I}_j} \big\| f^{(j)}_k \big\|_{L^2_x}^q\Big)^{\frac{1}{q}} \lesa \| \mc{U}_{\pm}(-t) u \|_{V^2} \sum_j 2^{-j} \big( 2^{2j} \big)^{\frac{1}{q}} \lesa \| \mc{U}_{\pm}(-t) u \|_{V^2} $$
  where we needed $q>2$ to ensure that the sum converges.
  \end{proof}
\end{corollary}

\begin{remark}\label{rem - refined strichartz implies better bound for Vpm est}
  If we restrict the support of $u$ further to a ball of radius $\mu$ in the annulus $\{|\xi| \approx \lambda\}$, i.e. assume that $\widehat{u} \subset \{ |\xi - \xi^*| \les \mu\}$ for some $\mu \les \lambda$ and $|\xi^*| \approx \lambda$, then the refined Strichartz estimate of Klainerman-Tataru \cite{Klainerman1999} implies that
        $$ \| \mc{M} u \|_{L^q_t L^r_x} \lesa \delta  \Big(\frac{\mu}{\lambda} \Big)^{ n ( \frac{1}{2} - \frac{1}{r}) - \frac{2}{q}} \lambda^{ n ( \frac{1}{2} - \frac{1}{r}) - \frac{1}{q}} \| \mc{U}_\pm(-t) u \|_{V^2}. $$
  We have no need for this additional refinement here, but it may prove useful elsewhere.
\end{remark}

The final result we need for the $V^2$ spaces is the crucial fact that our iteration norm $F^\pm_\lambda$ controls $V^2$,  this theorem (together with the previous corollary) is the key reason why the $V^2$ norms are so useful.

\begin{theorem}\label{thm - F controls V2}
 Let $u \in F^\pm_\lambda$. Then we have
            $$ \| \mc{U}_\pm(-t) u \|_{V^2} \lesa \| u \|_{F^\pm_\lambda}.$$
 \begin{proof}
   As usual, by a reflection, we may assume that $\pm = +$.  Let $F=(\p_t + \sigma \cdot \nabla) u$, by the definition of $F^\pm_\lambda$, it is enough to show that
        $$ \sum_{j \in \ZZ} \| \mc{U}_+(-t_{j+1}) u(t_{j+1}) - \mc{U}_+(-t_{j})u(t_{j}) \|_{L^2_x}^2 \lesa \| F \|_{N^+_\lambda}^2$$
   with the constant independent of the sequence $ (t_j)_{j \in \ZZ} \in \mathcal{Z}$.   If we observe that
        \begin{align*} \mc{U}_+(-t_{j+1})u(t_{j+1}) -  \mc{U}_+(-t_{j}) u(t_{j}) &= \int_{t_j}^{t_{j+1}} \mc{U}_+(-s) F(s) ds \\
        &= \int_{t_j}^{t_{j+1}} \mc{U}_+(-s) \ind_{I_j} F(s) ds\\
        & =\int_0^{t_{j+1}} \mc{U}_+(-s) \ind_{I_j} F(s)ds - \int_{0}^{t_{j}} \mc{U}_+(-s) \ind_{I_j} F(s)ds
        \end{align*}
   where $I_j = [t_j, t_{j+1})$, then we have
        \begin{align*}
          \sum_{j} \| \mc{U}_+(-t_j) u(t_j) - \mc{U}_+(-t_{j+1}) u(t_{j+1}) \|_{L^2_x}^2
                &= \sum_j \Big\| \int_0^{t_{j+1}} \mc{U}_+(-s) \ind_{I_j} F(s)ds - \int_{0}^{t_{j}} \mc{U}_+(-s) \ind_{I_j} F(s)ds \Big\|_{L^2_x}^2 \\
                &\les  2 \sum_j \Big\| \int_0^{t}  \mc{U}_+(-s)\ind_{I_j} F(s) ds \Big\|_{L^\infty_t L^2_x}^2
        \end{align*}
  An application of Corollary \ref{cor - energy type est} shows that
    $$\Big\| \int_0^t \mc{U}_+(-s)\ind_{I_j} F(s) ds \Big\|_{L^\infty_t L^2_x}  =\Big\| \int_0^t \mc{U}_+(t-s)\ind_{I_j} F(s) ds \Big\|_{L^\infty_t L^2_x} \lesa \| \ind_{I_j} F \|_{\mc{N}^+_\lambda}$$
  and so we reduce to proving the inequality
            \begin{equation}\label{eqn l2 sum over intervals bounded on N+} \Big( \sum_j \| \ind_{I_j} F \|_{\mc{N}^+_\lambda}^2\Big)^\frac{1}{2} \lesa \| F \|_{\mc{N}^+_\lambda}.\end{equation}
   By the atomic definition of $\mc{N}^+_\lambda$, it suffices to consider separately the cases where $F$ is a $L^1_t L^2_x$ atom, $ F$ is a $\dot{\mc{X}}^{-\frac{1}{2}, 1}_+$ atom, and $F$ is a $NF^+_\lambda$ atom.\\

   \textbf{Case 1: $F$ is a $L^1_t L^2_x$ atom.} This is the easiest case as we in fact have the stronger estimate
     $$ \sum_j \| \ind_{I_j} F \|_{\mc{N}^+_\lambda} \les \sum_j \| \ind_{I_j} F \|_{L^1_t L^2_x} \les \| F \|_{L^1_t L^2_x} .$$

   \textbf{Case 2:  $ F$ is a $\dot{\mc{X}}^{-\frac{1}{2}, 1}_{+}$ atom.} By definition, there exists $d \in 2^\ZZ$ such that
    $$\supp \widetilde{\Pi_\pm F} \subset \big\{ |\xi| \approx \lambda, \,\, |\tau \pm |\xi| \approx d \big\}$$
   and $\| F \| \les d^{-\frac{1}{2}}$.    We start by decomposing $\ind_{I_j} F$ into close cone and far cone terms, and estimate
        \begin{align} \| \ind_{I_j}F \|_{\mc{N}^+_\lambda} &\les  \big\| \mathfrak{C}^+_{\ll d} \big( \ind_{I_j}  F \big) \big\|_{\mc{N}^+_\lambda} + \big\| \mathfrak{C}^+_{\gtrsim d} \big( \ind_{I_j}  F \big) \big\|_{\mc{N}^+_\lambda} \notag \\
        &\les \big\| \mathfrak{C}^+_{\ll d}\big( \ind_{I_j}   F \big) \big\|_{L^1_t L^2_x} + \sum_{d' \gtrsim d} (d')^{-\frac{1}{2}} \big\| \mathfrak{C}^+_{d'} \big( \ind_{I_j}  F \big) \big\|_{L^2_{t, x}}\\
        &\lesa \big\| \mathfrak{C}^+_{\ll d}\big( \ind_{I_j}   F \big) \big\|_{L^1_t L^2_x} + d^{-\frac{1}{2}} \big\| \ind_{I_j}  F  \big\|_{L^2_{t, x}}. \label{eqn case 2 decomp - Xsb case}
        \end{align}
 The second term is easy to control by simply summing up in $j$.  On the other hand, for the first term in (\ref{eqn case 2 decomp - Xsb case}) the argument is more complicated. By rescaling\footnote{I.e. use the identity
    $$C^\pm_{\ll d}(\ind_{I_j} C^\pm_d F \big)(t, x)  = d^{n+1}  C^\pm_{\ll 1}(\ind_{d I_j} C^\pm_1 F_d \big)(dt, dx)$$
 where $\widetilde{F_d}(\tau, \xi) = \widetilde{F}( d \tau, d \xi)$. Note that the rescaled intervals $d I_j$ satisfy the same properties as the original intervals...} it is enough to consider the case $d=1$. We start by decomposing the intervals $I_j$ into those intervals which are smaller than $1$, and those that are larger than $1$, i.e. we write $\{ I_j \} = \{ K_j \} \cup \{ K'_j \}$ where $|K_j|<1$ and $|K'_j| \g 1$. For the intervals smaller than $1$, we can simply discard the outer multipliers, apply Holder in time, and sum up in $j$
    $$ \sum_j \big\| \mathfrak{C}^+_{\ll 1} \big( \ind_{K_j}   F \big) \big\|_{L^1_t L^2_x}^2 \lesa \sum_j \big\| \ind_{K_j}   F  \big\|_{L^1_t L^2_x}^2 \les \sum_j \| \ind_{K_j} F  \big\|_{L^2_{t, x}}^2 \les  \|  F  \big\|_{L^2_{t, x}}^2.$$
To deal with the intervals greater than $1$, we note that since $|\tau - \tau'| = | \tau \pm |\xi| - (\tau' \pm |\xi|)|$, for any $\big|\tau \pm |\xi| \big| \ll 1 $  we have the identity
        \begin{align} \widetilde{\big( \Pi_\pm ( \ind_{K'_j} F) \big)}(\tau, \xi) = \int_{|\tau - \tau'| \approx 1} \widehat{\ind_{K'_j}}(\tau - \tau') \widetilde{\Pi_\pm F}(\tau', \xi) d\tau' =  \int_{\RR} \widehat{\rho}_j(\tau - \tau') \widetilde{F}(\tau', \xi) d\tau' \label{eqn - thm F con stri - Xsb case large intervals decomp}\end{align}
where $\widehat{\rho}_j(\tau)  = \sigma(\tau) \widehat{\ind_{K_j}}(\tau)$ and $\sigma$ has support in the set $\{ |\tau| \approx 1 \}$.  Now, as $\frac{\sigma(\tau)}{\tau}$ is smooth and bounded, we can use integration by parts to deduce that for any $N>0$
    \begin{align*}
      \rho_j(t) = \frac{1}{2\pi}\int_\RR \sigma(\tau) \,\widehat{\ind_{K'_j}}(\tau) e^{  i t \tau} d \tau
                                               = \frac{1}{2\pi}\int_\RR \frac{i \sigma(\tau)}{\tau}  \Big( e^{ i \tau ( t-b_j) } -   e^{ i \tau (t-a_j)} \Big) d \tau \lesa \frac{1}{(1+|t - b_j|)^N} + \frac{1}{(1+|t - a_j|)^N}
    \end{align*}
where we let $K'_j=[a_j, b_j)$.  Hence, applying Holder in $t$ and assuming $N$ large,
    \begin{align*}
     \sum_j \big\| \mathfrak{C}^+_{\ll 1} \big( \ind_{K'_j} F \big) \big\|_{L^1_t L^2_x}^2 &= \sum_j \big\| \mathfrak{C}^+_{\ll 1}   \big( \rho_j F \big) \big\|_{L^1_t L^2_x}^2 \\
     &\lesa \sum_j  \big\|  \rho_j  F  \big\|_{L^1_t L^2_x}^2 \\
      &\lesa \sum_j \Big\| |\rho_j|^\frac{1}{2}  \| F\|_{L^2_x} \Big\|_{L^2_t}^2\\
       &\lesa \Big\| \| F \|_{L^2_{t, x} } \sum_j  \Big( \frac{1}{(1+|t-b_j|)^\frac{N}{2} } + \frac{1}{(1+|t-a_j|)^\frac{N}{2} }\Big)  \Big\|_{L^2_t}^2 \lesa \| F \|_{L^2_{t, x}}
    \end{align*}
where the sum converges since $|K'_j| \g 1 \Rightarrow $ $|a_j - a_{j+1}|, |b_j - b_{j+1}| \g 1$.\\

\textbf{ Case 3: $F$ a $NF^+_\lambda$ atom.} By definition, we have $\alpha\ll 1$ and a  decomposition $F= \sum_{\kappa \in \mc{C}_\alpha} F_\kappa$ with         $ \supp \widetilde{\Pi_\pm F_\kappa} \subset A^\pm_{\lambda, \alpha}(\kappa)$ and $ \sum_{\kappa \in \mc{C}_\alpha} \| F_\kappa \|_{NF^+(\kappa)}^2 \les 1$. Our aim is to deduce that
    $$ \sum_j \| \ind_{I_j} F \|_{\mc{N}^+_\lambda}^2 \lesa 1.$$
To this end we decompose $\ind_{I_j} F$ into regions close to the cone, and far from the cone
        \begin{equation} \label{eqn - 3rd case decomp into close cone far cone} \ind_{I_j} F = \sum_\pm \Big( \Pi_\pm C^\pm_{\ll \alpha^2 \lambda} \big( \ind_{I_j} F \big) + \Pi_{\pm} C^\pm_{ \gtrsim \alpha^2 \lambda} \big( \ind_{I_j} F\big)\Big). \end{equation}
For the close cone case, since the spatial Fourier projections commute with $\ind_{I_j}(t)$, $\Pi_\pm C^\pm_{\ll \alpha^2 \lambda} \big( \ind_{I_j} F \big)$ forms a (perhaps scalar multiple) of a $NF^+_\lambda$ atom. Therefore we can write
    \begin{align*}
      \Big\| \sum_{\pm} \Pi_\pm C^\pm_{\ll \alpha^2 \lambda}\big( \ind_{I_j} F\big) \Big\|_{\mc{N}^+_\lambda} &\les \Big( \sum_{\kappa \in \mc{C}_\alpha} \big\| \sum_{\pm} \Pi_{\pm} C^\pm_{\ll \alpha^2 \lambda}\big( \ind_{I_j} F_\kappa\big)  \big\|_{NF^+(\kappa)}^2 \Big)^{\frac{1}{2}} \\
      &\lesa  \Big( \sum_{\kappa \in \mc{C}_\alpha} \big\|  \ind_{I_j} F_\kappa  \big\|_{NF^+(\kappa)}^2 \Big)^{\frac{1}{2}}
    \end{align*}
where we used Lemma \ref{lem - mult are disposable} to dispose of the outer multipliers. Since we have an $\ell^2$ sum in $j$, we can swap the $j$ and $\kappa$ summations, and reduce to proving the inequality
        \begin{equation} \label{eqn sum over j of NF+(kappa)}
                \sum_j \| \ind_{I_j} F_\kappa \|_{NF^+(\kappa)}^2 \lesa \| F_\kappa \|_{NF^+(\kappa)}^2.
        \end{equation}
To this end, we note that for every $\omega \not \in 2\kappa$, we have
    \begin{align*}
      \sum_j \| \ind_{I_j} G \|_{NF^+(\kappa)}^2 &\les \sum_j \big( \| \ind_{I_j} \Pi_\omega G \|_{L^1_{t_\omega} L^2_{x_\omega}} + \theta(\omega, \kappa)^{-1} \| \ind_{I_j} \Pi_{-\omega} G \|_{L^1_{t_\omega} L^2_{x_\omega}}\big)^2 \\
      &\lesa \Big\| \Big(\sum_j \ind_{I_j} \Big)^{\frac{1}{2}} \Pi_{\omega} G \Big\|_{L^1_{t_\omega} L^2_{x_\omega}}^2 + \theta(\omega, \kappa)^{-2} \Big\| \Big(\sum_j \ind_{I_j} \Big)^{\frac{1}{2}} \Pi_{-\omega} G \Big\|_{L^1_{t_\omega} L^2_{x_\omega}}^2 \\
      &= \| \Pi_{\omega} G \|_{L^1_{t_\omega} L^2_{x_\omega}}^2 + \theta(\omega, \kappa)^{-2} \| \Pi_{-\omega} G \|_{L^1_{t_\omega} L^2_{x_\omega}}^2.
    \end{align*}
Taking infimum over $\omega \not \in 2\kappa$, and then applying the previous inequality to $NF^+(\kappa)$ atoms, then gives (\ref{eqn sum over j of NF+(kappa)}).

For the remaining far cone term in (\ref{eqn - 3rd case decomp into close cone far cone}), if we put the left hand side into $\mc{X}^{-\frac{1}{2}, 1}_{\lambda, +}$, and use Lemma \ref{lem - L2 bound on NF atom} to control the resulting $L^2_{t, x}$ norm of the atom $F$, we deduce that
    \begin{align*}
      \sum_j \big\| \sum_\pm \Pi_\pm C^\pm_{\gtrsim \alpha^2 \lambda} \big( \ind_{I_j} F \big) \big\|_{N^+_\lambda}^2 &\les \sum_j \Big(\sum_\pm \sum_{d \gtrsim \alpha^2 \lambda} d^{-\frac{1}{2}} \big\| \Pi_\pm C^\pm_d \big( \ind_{I_j} F \big) \big\|_{L^2_{t, x}}\Big)^2 \\
                    &\lesa (\alpha^2 \lambda)^{-1} \sum_j \| \ind_{I_j} F \|_{L^2_{t, x}}^2 \\
                    &\lesa (\alpha^2 \lambda)^{-1}  \| F \|_{L^2_{t, x}}^2  \lesa 1
    \end{align*}
as required.
 \end{proof}
\end{theorem}

By repeating the proof of the previous theorem, we have the following corollary which will prove useful when we come to the proof of (ii) in Theorem \ref{thm - energy inequality + invariance under cutoffs}.

\begin{corollary}\label{cor - N invariant under cutoffs}
Let $\rho \in \dot{B}^{\frac{1}{2}}_{2, \infty}(\RR) \cap L^\infty(\RR)$ and $F \in \mc{N}^\pm_\lambda$. Then
        $$ \| \rho(t) F \|_{\mc{N}^\pm_\lambda} \lesa \Big( \| \rho \|_{\dot{B}^\frac{1}{2}_{2, \infty}(\RR)} + \| \rho \|_{L^\infty(\RR)}\Big) \| F \|_{\mc{N}^\pm_\lambda}.$$
Similarly, if $u \in G^\pm_\lambda$, then
     $$ \lambda^{ - \frac{n+1}{4n}} \| \rho(t) u \|_{\mc{Y}^\pm } \lesa \Big( \| \rho \|_{L^\infty(\RR)} +  \lambda^{ - \frac{n+1}{4n}} \| \p_t \rho \|_{L^{\frac{4n}{3n-1}}(\RR)} \Big) \| u \|_{G^\pm_\lambda}.$$
\begin{proof}
 Fix $\pm = +$, the $\pm = -$ case follows by a reflection in $x$.  As usual, we decompose $F$ into atoms. If $F$ is a $L^1_t L^2_x$ atom, we clearly have
    $$ \| \rho(t) F \|_{\mc{N}^\pm_\lambda} \les \| \rho(t) F \|_{L^1_t L^2_x} \les \| \rho\|_{L^\infty_t}.$$
  On the other hand, if $F$ is a $\dot{\mc{X}}^{-\frac{1}{2}, 1}_+$ atom with $\supp \widetilde{\Pi_\pm F} \subset \big\{ \big| \tau \pm |\xi| \big| \approx d \big\}$, then from (\ref{eqn case 2 decomp - Xsb case}) we have
        $$ \| \rho(t) F \|_{\mc{N}^\pm_\lambda} \lesa \| \mathfrak{C}^\pm_{\ll d}\big( \rho(t) F \big) \big\|_{L^1_t L^2_x} + d^{-\frac{1}{2}} \|  \rho(t) F \|_{L^2_{t, x}} \les \big\| \mathfrak{C}^\pm_{\ll d}\big( \rho(t) F \big) \big\|_{L^1_t L^2_x} + \|\rho\|_{L^\infty_t(\RR)}.$$
  To control the first term we use the identity (\ref{eqn - thm F con stri - Xsb case large intervals decomp}) to deduce that
          $$ \big\| \mathfrak{C}^\pm_{\ll d}\big( \rho(t) F \big) \big\|_{L^1_t L^2_x} \lesa \| \widehat{\rho} \|_{L^2_\tau( |\tau| \approx d)} \| F \|_{L^2_{t, x}} \les d^\frac{1}{2} \| \widehat{\rho} \|_{L^2_\tau( |\tau| \approx d)}  \les \| \rho \|_{\dot{B}^\frac{1}{2}_{2, \infty}}$$
  and hence the required estimate is true whenever $F$ is a $\dot{\mc{X}}^{\frac{1}{2}, 1}_+$ atom. Finally, suppose $F = \sum_{\kappa \in \mc{C}_\alpha} F_{\kappa}$ is a $NF^+_\lambda$ atom. As in the $NF^+_\lambda$ case above, we write
            $$ \rho(t) F = \mathfrak{C}^+_{ \ll \alpha^2 \lambda}(\rho(t) F) +  \mathfrak{C}^+_{\gtrsim \alpha^2 \lambda}( \rho(t) F)$$
  The first term is a scalar multiple of a $NF^+_\lambda$ atom, and hence via Lemma \ref{lem - mult are disposable} we obtain
            $$\big\| \mathfrak{C}^+_{ \ll \alpha^2 \lambda}(\rho(t) F)  \big\|_{\mc{N}^+_\lambda} \les \bigg( \sum_{\kappa} \big\| \mathfrak{C}^+_{ \ll \alpha^2 \lambda}(\rho(t) F)\big\|_{NF^+(\kappa)}^2 \bigg)^\frac{1}{2} \lesa \bigg( \sum_{\kappa} \big\| \rho(t) F \big\|_{NF^+(\kappa)}^2 \bigg)^\frac{1}{2} \les \| \rho \|_{L^\infty_x}   $$
  where we made use of the obvious bound $\| \rho F_{\kappa} \|_{NF^+(\kappa)} \les \| \rho \|_{L^\infty_t} \| F_\kappa \|_{NF^+(\kappa)}$. For the remaining term, we estimate $\mc{N}^+_\lambda$ by $\dot{\mc{X}}^{\frac{1}{2}, 1}_+$, and use the $L^2_{t, x}$ bound for Null Frame atoms in Lemma \ref{lem - L2 bound on NF atom} to deduce
      $$ \big\| \mathfrak{C}^+_{ \gtrsim \alpha^2 \lambda}(\rho(t) F)  \big\|_{\mc{N}^+_\lambda} \lesa (\alpha^2 \lambda)^{-\frac{1}{2}} \| \rho(t) F \|_{L^2_{t, x}} \les (\alpha^2 \lambda)^{-\frac{1}{2}} \| \rho \|_{L^\infty_t} \| F \|_{L^2_{t, x}} \lesa \| \rho \|_{L^\infty_t}.$$
  It only remains to prove the $\mc{Y}^\pm$ estimate. We again make use of a similar argument to that used to control the $\dot{\mc{X}}^{-\frac{1}{2}, 1}_\pm$ case above. We start by observing that
        $$ \| \rho(t) u \|_{ \mc{Y}^\pm} \les \sup_d \,d\, \big\| \mathfrak{C}^\pm_d \big( \rho(t) \mathfrak{C}^\pm_{\gtrsim d} u\big) \big\|_{ L^\frac{4n}{3n-1}_t L^2_x} +  \sup_d \,d\, \big\| \mathfrak{C}^\pm_d \big( \rho(t) \mathfrak{C}^\pm_{\ll d} u\big) \big\|_{ L^\frac{4n}{3n-1}_t L^2_x}. $$
  To control the first term, we discard the outer multiplier and put $\rho \in L^\infty(\RR)$
       $$ \sup_d \,d\, \big\| \mathfrak{C}^\pm_d \big( \rho(t) \mathfrak{C}^\pm_{\gtrsim d} u\big) \big\|_{ L^\frac{4n}{3n-1}_t L^2_x} \lesa \| \rho \|_{L^\infty} \sum_{d' \gtrsim d} d \| \mathfrak{C}^\pm_{d'} u\|_{L^\frac{4n}{3n-1}_t L^2_x} \lesa \| \rho \|_{L^\infty} \| u \|_{\mc{Y}^\pm} \sum_{ d' \gtrsim d} \frac{d}{d'} \lesa \lambda^{\frac{n+1}{4n}}  \| \rho \|_{L^\infty} \| u \|_{G^\pm_\lambda}. $$
  For the second term, the identity (\ref{eqn - thm F con stri - Xsb case large intervals decomp}) allows us to replace $\rho$ with $P_{\approx d} \rho$ where $P_{\approx d}$ restricts the Fourier support of $\rho$ to the region $\tau \approx d$. We now use some standard Harmonic analysis to deduce that
        $$\sup_d \,d\, \big\| \mathfrak{C}^\pm_d \big( \rho(t) \mathfrak{C}^\pm_{\ll d} u\big) \big\|_{ L^\frac{4n}{3n-1}_t L^2_x} \lesa \| u \|_{L^\infty_t L^2_x} \sup_d d \| P_{\approx d} \rho \|_{L^\frac{4n}{3n-1}_t} \lesa \| \p_t \rho \|_{L^\frac{4n}{3n-1}_t} \| u \|_{G^\pm_\lambda}$$
  as required.

\end{proof}
\end{corollary}

%------------------------------------------------------------------------------------------------------------------------------%
\subsection{Proof of Theorem \ref{thm - F controls strichartz and angular sum}}\label{subsec - proof of Strichartz bounds}
%------------------------------------------------------------------------------------------------------------------------------%

We now come to the proof of Theorem \ref{thm - F controls strichartz and angular sum}.

\begin{proof}[Proof of Theorem \ref{thm - F controls strichartz and angular sum}]
 By Corollary \ref{cor - V2 controls strichartz} and Theorem \ref{thm - F controls V2}, it is enough to show that the $C^\pm_{\les d}$ multipliers are disposable in $V^2$. If we use the boundedness of $C^\pm_{\les d}$ on $L^\infty_t L^2_x$ (which follows from (iii) in Lemma \ref{lem - mult are disposable}), it is enough to show that $| C^\pm_{\les d} u |_{V^2} \lesa |u|_{V^2}$. To this end, by noting the identity
    $$ \widehat{C^\pm_{\les d} u}(t) = \int_\RR \Phi_0\big(\tfrac{\tau \pm |\xi|}{d}\big) \widetilde{u}(\tau, \xi) e^{i t \tau} \,d\tau =\int_\RR \widehat{\Phi}_0(a) e^{ \pm i \tfrac{a}{d} |\xi|} \widehat{u}(t+\tfrac{a}{d}, \xi) \,\,da$$
 we have for any sequence $(t_k) \in \mc{Z}$,
    \begin{align*}
      \Big(\sum_k \big\| C^\pm_{\les d} u(t_{k+1}) - C^\pm_{\les d} u(t_k)\big\|_{L^2_x}^2 \Big)^\frac{1}{2} &\les \int_\RR \big| \widehat{\Phi}_0(a) \big| \Big( \sum_k \big\| \widehat{u}(t_{k+1} + \tfrac{a}{d}) - \widehat{u}(t_k + \tfrac{a}{d}\big) \big\|_{L^2_\xi}^2 \Big)^\frac{1}{2} da \\
      &\les \int_\RR \big|\widehat{\Phi_0}(a) \big|   \, |u |_{V^2} da \lesa |u |_{V^2}.
    \end{align*}
 Taking the sup over $(t_k) \in \mc{Z}$, then gives $| C^\pm_{\les d} u |_{V^2} \lesa | u|_{V^2}$ as required.
 \end{proof}

%------------------------------------------------------------------------------------------------------------------------------%

%------------------------------------------------------------------------------------------------------------------------------%
   %------------------------------------------------------------------------------------------------------------------------------%
%------------------------------------------------------------------------------------------------------------------------------%
\section{The Energy Inequality} \label{sec - energy inequality}
%------------------------------------------------------------------------------------------------------------------------------%
%------------------------------------------------------------------------------------------------------------------------------%

Here we give the proof of Theorem \ref{thm - energy inequality + invariance under cutoffs}.

\begin{proof}[Proof of Theorem \ref{thm - energy inequality + invariance under cutoffs}]%\leavevmode

\textbf{(i)} We start by noting that $F^\pm_\lambda$ is a Banach space, since if $u_j \in F^\pm_\lambda$ is a Cauchy sequence with respect to $\| \cdot \|_{F^\pm_\lambda}$, then it is Cauchy with respect to $\| \cdot \|_{L^\infty_t L^2_x}$ and hence converges to some $u \in L^\infty_t L^2_x$ with $\supp u \subset \{ |\xi| \approx \lambda\}$. On the other hand, as $\mc{N}^\pm_\lambda$ is a Banach space, there exists $F \in \mc{N}^\pm_\lambda$ such that $(\p_t \pm \sigma \cdot \nabla) u_j$ converges to $F$ (with respect to $\| \cdot \|_{\mc{N}^\pm_\lambda}$). Consequently, $(\p_t \pm \sigma \cdot \nabla) u_j$ converges to $F$ in $\mc{S}'$, and hence by uniqueness of limits, we must have $(\p_t \pm \sigma \cdot \nabla) u = F \in \mc{N}^\pm_\lambda $. Therefore $u \in F^\pm_\lambda$ as required and so $F^\pm_\lambda$ is a Banach space.

To prove $F^{s, \pm}$ is a Banach space follows a similar argument, namely, if we have a Cauchy sequence in $F^{s, \pm}$, then it is also Cauchy in $L^\infty_t \dot{H}^s_x$ and hence converges to some $u \in L^\infty_t \dot{H}^s_x$. By uniqueness of limits, and the fact that $F^\pm_\lambda$ and $\ell^2$ are Banach spaces, we then deduce that $u \in F^{s, \pm}$ as required. Thus $F^{s, \pm}$ is a Banach space.

To prove the energy inequality for $F^{s, \pm}$, we clearly have
    $$ \| u \|_{F^{s, \pm}} \les \bigg( \sum_\lambda \lambda^{2s} \| P_\lambda u \|_{L^\infty_t L^2_x}^2 \bigg)^\frac{1}{2} + \| (\p_t \pm \sigma \cdot \nabla) u \|_{\mc{N}^{s, \pm}}$$
thus it is enough to prove that if $(\p_t \pm \sigma \cdot \nabla) u = F$ with $u(0) = 0$, then
    $$ \| P_\lambda u \|_{L^\infty_t L^2_x} \lesa \| P_\lambda F \|_{\mc{N}^\pm_\lambda}$$
but this is just Corollary \ref{cor - energy type est}. Similarly, to the prove the energy inequality for $G^{s, \pm}$, again using Corollary \ref{cor - energy type est}, it suffices to show that
    $$  d \| C^\pm_d  u \|_{L^\frac{4n}{3n-1}_t L^2_x} \lesa \| (\p_t \pm i |\nabla|) u \|_{L^{\frac{4n}{3n-1}}_t L^2_x} .$$
But this is straightforward by writing
$$ \widehat{\big( C^\pm_d u\big)}(t, \xi) = \frac{e^{ \mp i t |\xi|} }{2\pi} \int_{\RR} \Phi(\tfrac{\tau}{d})\, \widetilde{u}(\tau \mp |\xi|, \xi) e^{ i \tau t} \,d\tau = \frac{e^{ \mp i t |\xi|} }{d} \big[\rho_d *_\RR \widehat{v}(\xi)\big](t)$$
where $\widehat{\rho}_d(\tau) = \tfrac{d}{\tau} \Phi(\tfrac{\tau}{d}) \in C^\infty_0$ and $\widetilde{v}(\tau, \xi) = \tau \widetilde{u}(\tau \mp |\xi|, \xi)$. If we now apply Plancheral, Holder, and a change of variables, we deduce that
    $$ d \| C^\pm_d  u \|_{L^\frac{4n}{3n-1}_t L^2_x} \lesa \| \rho_d \|_{L^1_t(\RR)} \| \widehat{v} \|_{L^\frac{4n}{3n-1}_t L^2_\xi} \lesa \| (\p_t \pm |\nabla|) u \|_{L^\frac{4n}{3n-1}_t L^2_x} $$
as required.

\textbf{(ii)} Let $\phi \in C^\infty_0$. An application of Corollary \ref{cor - N invariant under cutoffs} gives
     \begin{align*}\big\| ( \p_t \pm \sigma \cdot \nabla)\big[ \phi(t) u \big] \big\|_{\mc{N}^\pm_\lambda} &\les \big\| \phi(t) (\p_t \pm \sigma \cdot \nabla) u \big\|_{\mc{N}^\pm_\lambda} + \| (\p_t \phi )(t) u \big\|_{L^1_t L^2_x} \\
     &\lesa \Big( \| \phi \|_{L^\infty_t} + \| \phi \|_{\dot{B}^\frac{1}{2}_{2, \infty}}\Big) \| (\p_t \pm \sigma \cdot \nabla) u \|_{\mc{N}^\pm_\lambda} + \| \p_t \phi \|_{L^1_t} \| u \|_{L^\infty_t L^2_x} \\
     &\lesa \Big( \| \phi \|_{L^\infty_t} + \| \phi \|_{\dot{B}^\frac{1}{2}_{2, \infty}} + \| \p_t \phi \|_{L^1_t} \Big) \| u \|_{F^\pm_\lambda}
     \end{align*}
and hence
        $$  \big\| \phi(t) \, u \big\|_{F^\pm_\lambda} \lesa \Big( \| \phi \|_{L^\infty_t} + \| \phi \|_{\dot{B}^\frac{1}{2}_{2, \infty}} + \| \p_t \phi \|_{L^1_t} \Big) \|  u \|_{F^\pm_\lambda}.$$
Applying this inequality with $\phi(t) = \rho(\frac{t}{T})$ and noting that $\| \rho(\frac{t}{T}) \|_{\dot{B}^\frac{1}{2}_{2, \infty}} \approx \| \rho(t) \|_{\dot{B}^\frac{1}{2}_{2, \infty}}$, we deduce that $\| \rho(\frac{t}{T}) u \|_{F^\pm_\lambda} \lesa \| u \|_{F^\pm_\lambda} $ as required. The $G^\pm_\lambda$ version follows from the $F^\pm_\lambda$ estimate together with  another  application of Corollary \ref{cor - N invariant under cutoffs} to deduce that
        $$\lambda^{ - \frac{n+1}{4n}} \big\| \rho( \tfrac{t}{T}) u \|_{\mc{Y}^\pm} \lesa \Big( \| \rho\|_{L^\infty_t} + \lambda^{ - \frac{n+1}{4n}} T^{ \frac{3n-1}{4n} -1} \| \p_t \rho \|_{L^\frac{4n}{3n-1}_t} \Big) \| u \|_{ G^\pm_\lambda} \lesa \| u \|_{G^\pm_\lambda}$$
provided $T\lambda \g 1$. Finally, the $\mc{N}^\pm_\lambda$ estimate follows by again applying Corollary \ref{cor - N invariant under cutoffs} and noting that  since $\ind_{(-1, 1)} \in \dot{B}^\frac{1}{2}_{2, \infty}$, by rescaling we have
    $$ \| \ind_{(-T,T)}(t) \|_{\dot{B}^\frac{1}{2}_{2, \infty}(\RR)} \approx \| \ind_{(1, 1)}(t) \|_{\dot{B}^\frac{1}{2}_{2, \infty}(\RR)}< \infty.$$

\textbf{(iii)} There are a number of ways to prove this, for instance it is possible to argue directly using the definition of $\mc{N}^\pm_\lambda$ see \cite[Prop 6.2]{Tataru2001}. Here we use an alternative argument based on Theorem \ref{thm - F controls V2}. For maps $\phi : \RR \rightarrow \dot{H}^s(\RR^n)$ we let
            $$ |\phi|_{V^{s, p}}^p = \sup_{(t_j) \in \mc{Z}} \sum_j \| \phi(t_{j+1}) - \phi(t_j) \|_{\dot{H}^s}^p.$$
Arguing as in Lemma \ref{lem - left and right limits always exist in Vp spaces}, we deduce that if $|\phi|_{V^{s, 2}}<\infty$, then there exists $\phi_{\pm \infty} \in \dot{H}^s$ such that $\lim_{t \rightarrow \pm \infty} \| \phi(t) - \phi_{\pm \infty} \|_{\dot{H}^s} = 0$. In particular, the scattering result we require would follow by showing that $| \mc{U}_\pm(-t) u |_{V^{s, 2}} < \infty$. To this end note that an application of Theorem \ref{thm - F controls V2} together with the fact that multiplication by $\rho(t)$ commutes with the homogeneous solution operator $\mc{U}_\pm(-t)$ gives
        $$ \big| \rho(\tfrac{t}{T}) \mc{U}_\pm(-t) u \big|_{V^{s, 2}} \lesa \big\| \rho(\tfrac{t}{T}) u \big\|_{F^{s, \pm}}.$$
Consequently it is enough to show that $ |\phi|_{V^{s, 2}} \les \sup_T \big| \rho( \tfrac{t}{T}) \phi \big|_{V^{s, 2}}$. To this end, take any increasing sequence $(t_j) \in \mc{Z}$. Let $N \in \NN$ and choose $T> \max_{|j|\les N+1} |t_j|$. Then
    $$\sum_{|j| \les N} \| \phi(t_{j+1}) - \phi(t_j) \|_{\dot{H}^s}^2 = \sum_{|j| \les N} \| \rho(\tfrac{t_{j+1}}{T}) \phi(t_{j+1}) - \rho(\tfrac{t_{j}}{T})\phi(t_j) \|_{\dot{H}^s}^2 \les \sup_{T>0} \big| \rho( \tfrac{t}{T}) \phi \big|_{V^{s, 2}}.$$
Hence letting $N \rightarrow \infty$ and taking the sup over $(t_j) \in \mc{Z}$, we get $|\phi|_{V^{s, 2}} \les \sup_{T>0} |\rho(\tfrac{t}{T}) \phi |_{V^{s, 2}}$ as required.
\end{proof}

%------------------------------------------------------------------------------------------------------------------------------%

\bibliographystyle{amsplain}

\providecommand{\bysame}{\leavevmode\hbox to3em{\hrulefill}\thinspace}
\providecommand{\MR}{\relax\ifhmode\unskip\space\fi MR }
\providecommand{\MRhref}[2]{%
  \href{http://www.ams.org/mathscinet-getitem?mr=#1}{#2}
}
\providecommand{\href}[2]{#2}

\end{document}